\documentclass[11pt, usletter,reqno]{amsart}
\usepackage{chngcntr}

\newif\ifdraft\draftfalse
\newif\ifcite
\newif\ifblow

\usepackage{amsfonts,amssymb, amsmath}
\usepackage{oldgerm}
\usepackage{amscd}

\usepackage[hyperindex]{hyperref}  
\ifdraft\ifcite\usepackage{showkeys}\else\usepackage[notcite,notref]{showkeys}\fi\fi

\usepackage{amsmath, amsthm, amssymb, eucal, amscd, xypic}
\usepackage{mathpazo}

\makeatletter

\theoremstyle{plain}
\newtheorem{proposition}[equation]{Proposition}
\newtheorem{Property}[equation]{Property}
\newtheorem{theorem}[equation]{Theorem}

\newtheorem{lemma}[equation]{Lemma}
\newtheorem{conjecture}[equation]{Conjecture}
\newtheorem{query}[equation]{Q}
\newtheorem{corollary}[equation]{Corollary}

\newtheorem{prop}[equation]{Proposition}

\theoremstyle{remark}

\theoremstyle{definition}
\newtheorem{definition}[equation]{Definition}

\theoremstyle{remark}

\newtheorem{remark}[equation]{Remark}

\newtheorem{example}[equation]{Example}

\def\define{\def}

\define\R{\mathbb{R}}
\define\Q{\mathbb{Q}}

\define\Im{\operatorname{Im}}

\define\Ext{\operatorname{Ext}}

\DeclareMathOperator{\Coh}{\mathrm{Coh}}
\DeclareMathOperator{\sw}{\mathrm{sw}}
\DeclareMathOperator{\Hdg}{\mathrm{Hdg}}
\DeclareMathOperator{\EXTpan}{\mathrm{EXTPAN}}
\DeclareMathOperator{\EXTPAN}{\mathrm{EXTPAN}}
\DeclareMathOperator{\Extpan}{\mathrm{Extpan}}
\DeclareMathOperator{\can}{\mathrm{can}}

\DeclareMathOperator{\coker}{\mathrm{coker}}

\DeclareMathOperator{\Isom}{\mathrm{Isom}}

\define\Hom{\operatorname{Hom}}

\newcommand\Jac{\operatorname{Jac}}
\newcommand\Pic{\operatorname{Pic}}

\define\sing{\operatorname{sing}}
\define\ad{\rm ad\,}

\newcommand\mer{\mathrm{mer}}

\define\g{\gamma}

\define\beq{\begin{equation}}
\define\eeq{\end{equation}}

\def\bp{\mathbb P}

\def\bean{\begin{eqnarray}}
\def\eean{\end{eqnarray}}
\def\bea{\begin{eqnarray*}}
\def\eea{\end{eqnarray*}}

\def\x0{X_{s_0}}
\def\2p{\bp^1\times \bp^1}
\def\g{\gamma}

\def\dim{{\rm dim}}

\DeclareMathOperator\Tor{\mathrm{Tor}}
\DeclareMathOperator\Biext{\mathrm{Biext}}
\DeclareMathOperator{\sExtpan}{\mathbf{Extpan}}
\DeclareMathOperator{\sExt}{\mathbf{Ext}}
\DeclareMathOperator\sHom{\mathbf{Hom}}
\DeclareMathOperator\sNF{\mathbf{NF}}
\DeclareMathOperator\sANF{\mathbf{ANF}}

\newcommand\metric{|\, \, |}

\def\ben{\begin{equation}}
\def\een{\end{equation}}

\newcommand\delr{\Delta^{r}}
\newcommand\delsr{\Delta^{*r}}

\makeatletter
\newcommand{\leqnomode}{\tagsleft@true}
\newcommand{\reqnomode}{\tagsleft@false}
\makeatother

\newcommand\effm{\mathit{effm}}
\newcommand\topp{\mathrm{top}}
\DeclareMathOperator{\hPic}{\widehat{\Pic}}

\DeclareMathOperator\id{\mathrm{id}}
\DeclareMathOperator\EXT{\mathrm{EXT}}
\DeclareMathOperator\IH{\mathrm{IH}}

\newcommand\arr{\ifinner \to\else\longrightarrow\fi}
\newcommand\arrto{\ifinner\mapsto\else\longmapsto\fi}

\newcommand*\Laplace{\mathop{}\!\mathbin\bigtriangleup}

\def\QQ{\mathbb Q}

\def\ZZ{\mathbb Z}

\newcommand\calB{\mathcal B}

\newcommand\calH{\mathcal{H}}

\newcommand\calJ{\mathcal J}

\newcommand\calL{\mathcal L}

\newcommand\calV{\mathcal V}

\newcommand\Ad{\operatorname{Ad}}

\newcommand\nilp{\mathrm{nilp}}

 \newcommand\rH{\mathrm{H}}

\newcommand\pr{\mathop{\mathrm{pr}}\nolimits}
\newcommand\codim{\mathop{\mathrm{codim}}\nolimits}

\newcommand\End{\mathop{\mathrm{End}}\nolimits}
\newcommand\Aut{\mathop{\mathrm{Aut}}\nolimits}

\newcommand\im{\mathop{\mathrm{im}}\nolimits}
\newcommand\cl{\mathop{\mathrm{cl}}\nolimits}

\newcommand\MHS{\mathop{\mathrm{MHS}}\nolimits}

\newcommand\VMHS{\mathop{\mathrm{VMHS}}\nolimits}

\newcommand\NF{\mathop{\mathrm{NF}}\nolimits}
\newcommand\ANF{\mathop{\mathrm{ANF}}\nolimits}

\newcommand\Perv{\mathop{\mathrm{Perv}}}
\newcommand\Gr{\mathop{\mathrm{Gr}}\nolimits}

\newcommand\IC{\mathop{\mathrm{IC}}\nolimits}
\newcommand\Loc{\mathrm{Loc}}

\newcommand\chv{\calH^{\vee}}

\begin{document}
\title{Jumps in the Archimedean Height}
\author{Patrick Brosnan}
\address{Department of Mathematics\\
  University of Maryland\\
  College Park, MD USA}
\email{pbrosnan\@math.umd.edu}
\author{Gregory Pearlstein}
\address{Department of Mathematics\\
Texas A\&M University\\
College Station, TX USA}
\email{gpearl\@math.tamu.edu}

\begin{abstract}
  We introduce a pairing on local intersection cohomology groups of
  variations of pure Hodge structure, which we call the asymptotic
  height pairing.  Our original application of this pairing was to
  answer a question on the Ceresa cycle posed by R.~Hain and D.~Reed.
  (This question has since been answered independently by Hain.)  Here
   we show that a certain analytic line bundle,
  called the biextension line bundle, defined in terms of normal
  functions, always extends to any smooth partial compactification of
  the base.  We then show that the asymptotic height pairing on intersection cohomology
  governs the extension of the natural metric on this line bundle
  studied by Hain and Reed (as well as, more recently, by several
  other authors).  We also prove a positivity property of the
  asymptotic height pairing, which generalizes results of a recent
  preprint of J.~Burgos Gill, D.~Holmes and R. de Jong, along with a
  continuity property of the pairing in the normal function case.
  Moreover, we show that the asymptotic height pairing arises in a
  natural way from certain Mumford-Grothendieck biextensions
  associated to normal functions.

\end{abstract}

\maketitle

\tableofcontents

\section{Introduction}
Let $\calH$ be a weight $-1$ polarizable variation of pure Hodge
structure with integral coefficients over a complex manifold $S$.
Assume that the local system $\calH_{\ZZ}$ associated to $\calH$ is
torsion-free.  The elements of the group $\NF(S,\calH)$ of normal
functions into $\calH$ can be thought of either as certain holomorphic
sections of the Griffiths intermediate Jacobian $\calJ(\calH)\arr S$,
or as Yoneda extensions of $\ZZ$ by $\calH$ in the category $\VMHS(S)$
of variations of mixed Hodge structure over $S$, i.e., elements of
$\Ext^1_{\VMHS(S)}(\ZZ,\calH)$.  Define 
$\NF(S,\mathcal H)^{\vee} = \Ext^1_{\VMHS(S)}(\calH,\ZZ(1))$ and 
note that via duality (cf. Proposition~\eqref{p.dual})  
$$
  \NF(S,\mathcal H)^{\vee} \cong\NF(S,\mathcal H^{\vee})
$$ 
where $\mathcal H^{\vee} = \Hom(\mathcal H,\ZZ(1))$.

\par A biextension variation of type $\mathcal H$ is a variation of mixed Hodge 
structure $\mathcal V$ over $S$ equipped with isomorphisms
\begin{equation}
     Gr^W_0(\mathcal V)\cong\mathbb Z(0),\qquad
     Gr^W_{-1}(\mathcal V)\cong\mathcal H,\qquad
     Gr^W_{-2}(\mathcal V)\cong\mathbb Z(1).
     \label{eq:type-H}
\end{equation}
To any biextension variation $\mathcal V$ over $S$ there is an associated 
$C^{\infty}$ function [cf. \eqref{eq:JDG-height}]
\begin{equation}
          h:S\to\mathbb R            \label{eq:delta-height}
        \end{equation}
        originally studied by Hain~\cite{HainBiext}.  The function $h$
        can be 
defined by Deligne's $\delta$-splitting (Theorem $(2.20)$, \cite{CKS86}), which 
measures how far the fibers $\mathcal V_s$ are from being $\mathbb R$-split mixed 
Hodge structures.

\par Given $\nu\in\NF(S,\calH)$ and $\omega\in\NF(S,\mathcal H)^{\vee}$, the set of 
biextension variations of type $(\nu,\omega)$ consists of all biextension variations 
over $S$ such that 
\begin{align}
  (W_0/W_{-2})\calV    &\cong \nu;\label{e.Biextension}\\
  (W_{-1}/W_{-3})\calV  &\cong \omega.\nonumber
\end{align}
For any open subset $U$ of $S$, $B(\nu,\omega)(U)$ is defined to be the set of isomorphism classes
of biextension variations of type $(\nu,\omega)$ over $U$.

\par Hain and Reed noticed that $\calB(\nu,\omega)$ is naturally 
an $\mathcal{O}_S^{\times}$-torsor over $S$.  Therefore the sections of $\calB(\nu,\omega)$ can be identified with the non-vanishing sections of an analytic line bundle $\calL=\calL(\nu,\omega)$ over $S$ which 
Hain and Reed have called the \emph{biextension line bundle}~\cite{HainReedJDG}.

\par In~\cite{HainReedJDG}, Hain and Reed further observed that $\mathcal L(\nu,\omega)$ comes with a 
natural metric defined by the requirement that
\begin{equation}
         |\mathcal V| = \exp(-h(\mathcal V))                     \label{eq:hr-metric}
\end{equation}
for any $\mathcal V\in B(\nu,\omega)$.  In the case where $\nu$ and $\omega$ arise from a family
of homologically trivial algebraic cycles, the function \eqref{eq:delta-height} is the archimedean
height pairing.

\par Suppose that $j:S\hookrightarrow\overline{S}$ is an inclusion of
$S$ as a Zariski open subset of a complex manifold $\overline{S}$.  Let 
$\mathcal U$ be a weight $-1$ polarizable variation of pure Hodge structure 
with integral coefficients on $S$.  Assume that $\mathcal U_{\mathbb Z}$
is torsion free. Then, M.~Saito~\cite{SaitoANF} defined the group of admissible 
normal functions $\ANF(S,\mathcal U)_{\bar S}$  to be set of extensions of $\mathbb Z(0)$ 
by $\mathcal U$ in the category of admissible variations of mixed Hodge 
structure, i.e.
$$
    \ANF(S,\mathcal U)_{\bar S} :=\Ext^1_{\VMHS(S)_{\bar S}^{\rm ad}}(\mathbb Z(0),\mathcal U)
$$
Normal functions of geometric origin (e.g. arising from families of algebraic cycles) are admissible.

\begin{remark} A choice of polarization $Q$ of $\mathcal U$ determines
  a morphism $a_Q:\mathcal{U}\to\mathcal{U}^{\vee}$ given by
  $\beta\mapsto Q(\beta,\_)$.  This induces a map
  $\ANF(S,\mathcal{U})_{\bar S} \to \ANF(S,\mathcal{U}^{\vee})_{\bar
    S}$ which we denote by $\nu\mapsto\nu^{\vee}$. 
\end{remark}      

\par Let $\nu\in\ANF(S,\calH)_{\bar S}$ and $\omega\in\ANF(S,\calH^{\vee})_{\bar S}$.
Then, one can ask two questions.
\begin{query}
\label{query:Q1}
  Does the line bundle $\calL$ extend to a line bundle $\overline{\calL}$
  on $\overline{S}$?  If fact, what we really seek is an extension
  $\overline{\calL}$ satisfying the following property:
  For
  $U$ an open subset of $\overline{S}$, the restrictions of the  non-vanishing sections in
  $\overline{\calL}(U)$ to $\calL(U)$ are biextension variations which
  are admissible as variations of mixed Hodge structure.
\end{query}
\begin{query}
\label{query:Q2}
Does the metric $\metric$ extend to $\overline{\calL}$?
\end{query}

Note that, in the (most interesting) case where $\bar S$ is a smooth, 
projective variety, Q~\ref{query:Q1} can be rephrased (via the GAGA 
principle~\cite{GAGA}) as asking whether or not $\calL$ is algebraic.  In Theorem~\ref{eb1}
we show that Q~\ref{query:Q1} has a positive answer. While there are, in
general, many extensions of $\calL$ to $\bar S$, we
show in \S\ref{sec:merext} that there is a canonical
meromorphic extension.   Moreover, there is a canonical  
choice of extension $\overline{\calL}_{\can}\in \Pic \bar S\otimes\QQ$.
(See Remark~\ref{canext}.)

\par On the other hand, we will give a criterion in terms of the
local intersection cohomology of $\nu$ and  $\omega$ for Q~\ref{query:Q2} 
to have a negative answer and, using this criterion, we will give an example 
where the metric does not extend.

\begin{remark}\label{remexth}
  Note that there are many examples of analytic line bundles over
  smooth complex algebraic varieties $S$ which do not extend to
  $\bar S$ with $\bar S$ as in Q\ref{query:Q1}.  For example,
  in~\cite{SerreProlong}, Serre gives uncountably many examples
  of analytic line bundles on $\mathbb{C}^2\setminus\{(0,0)\}$ which do
  not extend to $\mathbb{C}^2$.   Consequently, it is not obvious 
  that Q\ref{query:Q1} has a positive answer even in the case that
  the codimension of $\bar S\setminus S$ in $\bar S$ is greater than $1$.
  
  D.~Lear's 1991 University of Washington (unpublished) PhD thesis
  treats special cases of Q\ref{query:Q1} and~Q\ref{query:Q2}.  For
  Q\ref{query:Q1}, Lear handles the case where $S$ is a curve and the
  local monodromy is unipotent with Jordan blocks of size at most $2$.
  In this case, there are local monodromy invariant splittings of the
  local systems underlying the variations of Hodge structure
  corresponding to the normal functions $\nu$ and $\omega$.  While
  Lear does not show explicitly that the extension $\overline{\calL}$
  satisfies the property in Q\ref{query:Q1}, we believe that his
  arguments could be used to establish this.   However, we remark that
  in this case, the monodromy invariant splitting makes the arguments
  that we use in \S\ref{sec:mix} much simpler.
\end{remark}

\par We now sketch our results concerning intersection cohomology
and obstructions to Q\ref{query:Q2}.  Let $\bar S=\Delta^r$ be a polydisk with local coordinates 
$(s_1,\dots,s_r)$ and $S=\Delta^{*r}$ be the complement of the divisor
$s_1\dots s_r=0$.  Let $\mathcal H$ be a torsion free variation on $S$
with unipotent monodromy and let $\IH^1(\mathcal{H})$ denote the local intersection 
cohomology group attached to $\mathcal{H}$.  If $i:\{0\}\to \Delta^r$ is the inclusion 
of the origin, then, in terms of the intersection complex $\IC(\mathcal{H})$ on 
$\Delta^r$ associated to $\mathcal{H}$,  $\IH^1(\mathcal{H})=H^{-r+1}(i^*\IC(\mathcal{H}))$.
The group $\IH^1(\mathcal{H})$ is a subgroup of $H^1(\Delta^{*r},\mathcal{H})$.
Moreover, since $\mathcal{H}$ is a variation of weight $-1$,
the natural map (cf. Theorem~\eqref{factor}) 
$$
     \sing:\ANF(\Delta^{*r},\mathcal{H})_{\Delta^r}\to 
            H^1(\Delta^{*r},\mathcal{H})
$$
factors through $\IH^1(\Delta^r,\mathcal{H})$~\cite{BFNP}.
We then define, for each $t\in\mathbb{Q}_{\geq 0}^r$, a pairing
\begin{equation}
h(t):\IH^1(\mathcal{H})\otimes\IH^1(\mathcal{H}^*)\to \mathbb{Q}, 
\label{eq:def-ahp}
\end{equation}
which we call the asymptotic height pairing. 

\begin{remark} It will be clear from context if $h$ refers to the 
height $h(\mathcal V)$ of a variation of mixed Hodge structure, or the
asymptotoic height pairing $h(t)$.
\end{remark}

We will show that the pairing $h$ has the following properties 
(cf. Theorem~\eqref{ahp0}, Proposition~\eqref{ahp1}
and Theorem~\ref{cext}).
\begin{enumerate}
\item It is homogeneous of degree $1$ in $t$: $h(\lambda t)(\alpha,\beta)
=\lambda h(t)(\alpha,\beta)$ for $\lambda\in\mathbb{Q}_{\geq 0}$. 
\item If $(\nu,\omega)\in\ANF(\Delta^{*r},\mathcal{H})\times
  \ANF(\Delta^{*r},\mathcal{H}^{\vee})$, then
  $h(\nu,\omega)(t)$ is a rational function which extends
  continuously to the cone $\mathbb{R}_{\geq 0}^{r}$. 
 
\item For each fixed $t$, $h(t)$ induces a morphism of Hodge structures
with $\mathbb{Q}$ given the pure Hodge structure of weight $0$.  Thus 
$h(t)$ factors through a pairing 
\begin{equation}\label{HodgePairing}
\bar h(t):
\Gr^W_0\IH^1(\mathcal{H})\otimes\Gr^W_0\IH^1(\mathcal{H}^*)\to \mathbb{Q}(0),
\end{equation}
which is a morphism of Hodge structures.
\item Suppose $\phi:\mathcal{H}\otimes\mathcal{H}\to \mathbb{Q}(1)$
  is a polarization.  Then, for $t\in\mathbb{Q}_+^r$,  the pairing on
  $\Gr^W_0\IH^1(\mathcal{H})$
  induced by $\phi$
  and \eqref{HodgePairing} is a polarization.   
\end{enumerate}

\begin{remark} In section \ref{higher-p} we show that $h(t)$ can be generalized to a pairing
\eqref{higherhp2} on the higher intersection cohomology groups 
$\IH^p(\mathcal H)$ as well.
\end{remark}

\par Returning to Q~\ref{query:Q2}, it follows from Theorem~\eqref{mod_m} below that 
the metric $\metric$ on $\mathcal L$ does not extend from $S$ to $\bar S$ if $h(t)$ 
is not linear in $t$.


To illustrate the consequences of the results listed
in (iii) and (iv) above, we give the following
corollary. 

\begin{corollary}\label{e.jumping} $h(t)(\nu,\nu^{\vee})\geq 0$ for all $\nu$,
$t$.  Moreover, $h(t)(\nu,\nu^{\vee}) = 0$ for all $t$ if and only if 
$\text{\rm sing}(\nu)=0$.
\end{corollary}
\begin{proof} $\IH^1(\mathcal{H})$ carries a mixed Hodge structure
such that $W_0\IH^1(\mathcal{H})=\IH^1(\mathcal{H})$.  Furthermore, the map
$\sing:\ANF(\Delta^{*r},\mathcal{H})_{\Delta^r}\to \IH^1(\mathcal{H})$
factors through the subgroup $\Hdg\IH^1(\mathcal{H})$ consisting of 
classes in the image of morphisms in $\Hom_{\MHS}(\mathbb{Z}(0),\IH^1(\mathcal{H}))$.  
Therefore (cf.~\eqref{HodgePairing}), it follows from (iii) and (iv) above 
that $h(t)(\nu,\nu^{\vee})\geq 0$.  Moreover, $h(t)(\nu,\nu^{\vee})=0$ for 
all $t$ if and only if $\text{\rm sing}(\nu)=0$.
\end{proof}

\par Singularities of normal functions can also be defined
in the non-normal crossing case (see~\cite{BFNP} or
\eqref{singdef} below).
Let $S$ is a complex manifold and $j:S\hookrightarrow\bar S$
is an inclusion of $S$ as a Zariski open subset of a complex
manifold $\bar S$. We say that $\nu\in\ANF(S,\mathcal H)_{\bar S}$ 
is {\it singular} at $p\in\bar S - S$ if $\sing_p(\nu)\neq 0$. 

\subsection{Applications}

\subsection*{Griffiths--Green Program}  Let $X$ be a complex projective 
manifold of dimension $2n$ and $L\to X$ be a very ample line bundle.  Let 
$\bar P = |L| = \mathbb{P}(H^0(X,L))$, and let $\mathcal X\subset X\times\bar P$ 
be the incidence variety consisting of pairs $(x,\sigma)\in X\times\bar P$ such that 
$\sigma(x)=0$.  Let $\pi:\mathcal X\to\bar P$ denote the projection map 
$\pi(x,\sigma) = \sigma$.  For $\sigma\in\bar P$ let 
$X_{\sigma} = \pi^{-1}(\sigma)$.  Let $P\subset\bar P$ denote the locus
of points over which $\pi$ is smooth, and $\hat X = \bar P -P$.

\par Let $\zeta$ be a primitive integral, non-torsion Hodge class on $X$ of 
type $(n,n)$.  Then, via Deligne cohomology, $\zeta$ determines an
admissible  normal function
$$
       \nu_{\zeta}:P\to J(\mathcal H)
$$
where $\mathcal H$ is the variation of Hodge structure on the variable 
cohomology of the smooth hyperplane sections of $X$.  

\begin{conjecture}\label{bfnp2} For every triple $(X,L,\zeta)$ as above, there exists an 
integer $d>0$ such that after replacing $L$ by $L^d$, the normal function 
$\nu_{\zeta}$ is singular at some point of $\hat X$.
\end{conjecture}

\begin{theorem}\cite{GG,BFNP,dCM}\label{bfnp} The Hodge conjecture is equivalent to 
Conjecture~\eqref{bfnp2}.
\end{theorem} 

 \par By Theorem $(7.18)$ of~\cite{BFNP}, Conjecture~\eqref{bfnp2} is equivalent to the
following statement:

\begin{conjecture}\label{bfnp3} For every triple $(X,L,\zeta)$ as above, there exists
$d>0$ and a resolution of singularities $\bar S\to\bar P$ of $\hat X$ for 
$\bar P = |L^d|$ such that $f^*\nu_{\zeta}$ is singular at some point of $\bar S - S$.
\end{conjecture}

\par So, using Corollary~\ref{e.jumping}, the Hodge conjecture would
imply that there exists a $d>0$ and an $\bar S$ as above such that the
biextension metric on $\mathcal{L}(\nu_{\zeta},\nu_{\zeta}^{\vee})$
does not extend to
$\overline{\mathcal{L}}(\nu_{\zeta},\nu_{\zeta}^{\vee})$.

\subsection*{Ceresa Cycle} Let $C$ be a smooth projective curve of genus 
$g>2$.  Pick a point $p\in C$.  Then, the maps $x\mapsto x-p$ and 
$x\mapsto p-x$ give two embeddings of $C$ into $\Jac(C)$. The difference,
denoted $C-C^-$, is a homologically trivial algebraic 1-cycle on $\Jac(C)$ 
which has a well defined image in the intermediate Jacobian of the primitive 
part of the cohomology of $\Jac(C)$, independent of $p$.

\par Let $\mathcal M_g$ denote the moduli space of smooth projective
curves of genus $g>2$.  Then, application of the previous construction
determines a normal function $\nu = \nu_{C-C^-}$ over
$\mathcal M_g$.  In Theorem $(8.3)$ of~\cite{HainMSRI}, Richard
Hain showed that is the \lq\lq basic\rq\rq{} normal function over
$\mathcal M_g$.

\par In~\cite{HainReedJDG}, Hain and Reed showed explicitly that 
$\mathcal L(\nu,\nu^{\vee})$ extended to a line bundle over 
$\overline{\mathcal M_g}$, and asked if the metric also extends to 
$\overline{\mathcal M_g}$.  Our calculations show that it does not.   
(Hain has since given an independent proof of this result.  
See~\cite{HainJump}.) 

\subsection{Existence and Regularity}\label{sec:eandr} 

%

\subsection*{Existence of $\mathcal L$,
  Local Case} In the case where $\mathcal H$
is a torsion free variation of pure Hodge structure of weight $-1$
over a punctured disk $\Delta^{*r}$,
it follows from the results of \S\ref{sec:mix} that $\mathcal L$
has an analytic extension over $\Delta^r$. (As mentioned above, this is proved 
in greater generality 
in Theorem~\ref{eb1}.)

\par Moreover, assuming the monodromy of $\mathcal H$ is unipotent, the biextension metric has a 
continuous extension on the complement of a codimension 2 subset after a suitable 
renormalization.  In the case where $\omega = \nu^{\vee}$ the biextension metric has a 
plurisubharmonic extension as described below.  More precisely, given a biextension variation $\mathcal V$ 
such that $Gr^W_{-1}(\mathcal V)=\mathcal H$ has unipotent monodromy, it follows 
from~(Theorem $(5.37)$~\cite{GregJDG}) that there exists a unique element 
\begin{equation}
          \mu(\mathcal V)\in V_r                
\end{equation}
with the following property:

\begin{Property}
Let $\phi:\Delta\to \Delta^r$ be a holomorphic map such that
$\phi(0)=0$ and $\phi(\Delta^*)\subset\Delta^{*r}$. Suppose that  $\phi(s)=(\phi_1(s),\ldots, \phi_r(s))$ 
with $\phi_i(s)=a_is^{m_i} + \text{higher order terms}$ and $m_j>0$. 
Then 
\begin{equation}
         h(\mathcal V_{\phi(s)})\sim -\mu(\mathcal V)(m_1,\ldots,m_r)\log|s|    
\label{eq:mu-def}
\end{equation}
for $s$ close to $0$. (In other words, the difference between the left
side and the right is bounded near $0$.)
\end{Property}


\begin{remark}\label{rat-mu} The value of $\mu(m)$ depends only on the local monodromy 
logarithm $N(m)=\sum_j\, m_j N_j$, the weight filtration $W(V_{\mathbb Q})$ and 
the choice of positive generators $1\in Gr^W_0(V_{\mathbb Q})$ and $1^{\vee}\in Gr^W_{-2}(V_{\mathbb Q})$. 
More precisely, as long as $m_1,\dots,m_r\geq 0$, the admissibility of $\mathcal V$
implies the existence of a rational direct sum decomposition
\begin{equation}
        V_{\mathbb Q} = \mathbb Q v_0\oplus U_{\mathbb Q}\oplus\mathbb Q v_{-2}
        \label{eq:rat-mu-2}
\end{equation}
such that
\begin{itemize}
\item[---] $v_0$ projects to $1\in Gr^W_0$ and $N(m)(v_0)\in W_{-2}$;
\item[---] $U_{\mathbb Q}$ is an $N(m)$-invariant subspace of $W_{-1}$;
\item[---] $v_{-2}\in W_{-2}$ projects to $1^{\vee}$.
\end{itemize}
Then, by Theorem $(5.19)$ of~\cite{GregJDG}, for any such decomposition \eqref{eq:rat-mu-2}
we have $N(m)v_0 = \mu(m)v_{-2}$.  Indeed, in the language of~\cite{GregJDG}, such a 
decomposition of $V_{\mathbb Q}$ defines a grading $Y'$ of $W$ (cf. \S \ref{sec:vmhs}) such that
$[Y',N]$ lowers $W$ by $2$. For future use, we define $\mu_j = \mu(\epsilon_j)$ where $\epsilon_j$ 
is the $j$'th unit vector.
\end{remark}

\begin{theorem}\label{mod_m} If $\mathcal V$ is a biextension variation of type $(\nu,\omega)$ 
over $\Delta^{*r}$ with unipotent monodromy then 
$h(m)(\nu,\omega) = \sum_j\, m_j\mu_j-\mu(\mathcal V)$.
\end{theorem}
\begin{proof} See Theorem~\eqref{JDG5}
\end{proof}

\subsection*{Regularity of the Biextension Metric}  Let $\mathcal V$ be a 
biextension variation with unipotent monodromy defined on the complement 
of the divisor $D=\{s_1\cdots s_k=0\}$ in a polydisk $\Delta^r$ with local 
coordinates $(s_1,\dots,s_r)$.  In this case, it follows by Theorem $(5.8)$ 
of~\cite{HP} that $\log|\mathcal V|$ is locally $L^1$.

\par Let $E$ denote the singular locus of $D$ and $D_j\subset D$ denote the divisor 
defined by $s_j=0$.  Given a point $p\in (D-E)\cap D_j$ let $\Delta(p)\subset\Delta^r$ 
denote the disk defined by $s_i = s_i(p)$ for $i\neq j$.  Then, $\mathcal V$ restricts 
to an admissible biextension variation on $\Delta^*(p) = \Delta(p)-\{p\}$ with unipotent 
monodromy, and hence
$$
        h(\left.\mathcal V\right|_{\Delta^*(p)})
               \sim -\mu_j\log|s_j|
$$
in the notation of Remark~\eqref{rat-mu}.
 
\par Define 
\begin{equation}
        \bar h(\mathcal V) = h(\mathcal V) 
         +\sum_{j=1}^k\, \mu_j\log|s_j|  :\Delta^r-D\to [-\infty,\infty)   
		\label{eq:smooth-nc}
\end{equation}
Then, for any point $p\in D-E$, by restriction to a small polydisk containing $p$, 
we can reduce the study of the regularity of $h(\mathcal V)$ at $p$ to the case 
$k=1$.  

\par Assume therefore that $k=1$ and $\mathcal V$ has unipotent monodromy 
about $s_1=0$.   

\begin{theorem}\label{sm-continous-0} Let $\mathcal V$ be 
a biextension variation with unipotent monodromy defined on the complement 
of the divisor $s_1=0$ in $\Delta^r$.  Then, $\bar h(\mathcal V)$ extends
continuously to $\Delta^r$.  
\end{theorem}
\begin{proof} See Theorem~\eqref{sm-continuous}.
\end{proof}

\begin{remark} The proof of Theorem \eqref{sm-continous-0} involves showing that 
in the case of a biextension over $\Delta^*$
$$
        \lim_{s\to 0}\, h(\mathcal V) + \mu\log(s) = ht(N,F,W)
$$
where $(e^{zN}F,W)$ is the associated nilpotent orbit. See Theorem \eqref{disk} for 
details.
\end{remark}

\par Returning now to the general case $\Delta^r-D$ where $D$ is given by the vanishing
of $s_1\cdots s_k=0$ we first recall the following result in the case where 
$\omega = \nu^{\vee}$:

\begin{theorem}\cite{PP} If $\mathcal V$ is 
a biextension variation of type $(\nu,\nu^{\vee})$ on a complex manifold 
$S$ then $h(\mathcal V)$ is plurisubharmonic function on $S$.
\end{theorem}

\par To continue, we recall the following [See $(7.1)$ and $(7.2)$~\cite{Kiselman}]:  Let $\{f_j\}$ be 
a sequence of plurisubharmonic functions on a domain $\Omega\subseteq\mathbb C^r$ which are locally 
bounded from above.  Given $z\in\Omega$ let 
$$
        f(z) = \limsup_{j\to\infty}\, f_j(z)
$$
and define $f^*(z) = \limsup_{w\to z}\, f(w)$.  Then, $f^*$ is plurisubharmonic.

\begin{theorem}\label{pluri-ext} Let $\mathcal V$ be a biextension of type 
$(\nu,\nu^{\vee})$ with unipotent monodromy on the complement of 
$D=\{s_1\cdots s_k=0\}$ in the polydisk $\Delta^r$.  Suppose that 
$\mu_1,\dots,\mu_k \geq 0$. Then, $\bar h(\mathcal V)$  has a 
unique plurisubharmonic extension to $\Delta^r$.
\end{theorem}
\begin{proof} By Theorem~\eqref{sm-continous-0} it follows that 
$g = \bar h(\mathcal V)\in C^0(\Delta^r-D)$ extends to $\bar g\in C^0(\Delta^r-E)$.
 Moreover, since $h(\mathcal V)$ is plurisubharmonic on $\Delta^r-D$ 
and all $\mu_j\geq 0$ by assumption, it follows that $g$ is plurisubharmonic on 
$\Delta^r-D$.  In the paragraphs below, we shall prove that $\bar g$ is 
plurisubharmonic on $\Delta^r-E$.  Since $E$ has codimension $2$, it then 
follows from a theorem of Grauert and Remmert\cite{GR} that $\bar g$ has a 
unique plurisubharmonic extension to $\Delta^r$.

\par To prove that $\bar g$ is plurisubharmonic on $\Delta^r-E$, consider the sequence of functions
$$
       f_j = \bar g + \frac{1}{j}\log|s_1\dots s_k|:\Delta^r-E\to[-\infty,\infty)
$$
where $f_j = -\infty$ on $D-E$.   Clearly, each $f_j$ is plurisubharmonic on $\Delta^r-E$ since 
\begin{itemize}
\item[---] $g$ and $\log|s_1\dots s_k|$ are plurisubharmonic 
on $\Delta^r-D$;
\item[---] $f_j$ is upper semicontinuous on $\Delta^r-E$ as the sum of the continuous function 
$\bar g$ and the upper semicontinuous function $\frac{1}{j}\log|s_1\dots s_k|$.
\item[---] $f_j$ has the subaveraging property at any point of $D-E$;
\end{itemize}
Moreover, since $\bar g$ is continuous and all $\mu_j\geq 0$ it follows that each $f_j$ is 
locally bounded above.  Accordingly,  $f^*$ is plurisubharmonic.  Finally, since $\bar g$
is continuous on $\Delta^r-E$ it follows form the definition of the sequence $f_j$ that
$f^* = \bar g$.
\end{proof}

\par Having constructed an extension of $\bar h(\mathcal V)$ to $\Delta^r$, we can inquire
as to its value at $(0,\dots,0)\in\Delta^r$.  To this end, let 
$m=(m_1,\dots,m_r)\in\mathbb Z^r_{>0}$ and $\phi:\Delta\to\Delta^r$
be the test curve defined by $s_j = s^{m_j}$, $j=1,\dots,r$.  Then, by Theorem \ref{mod_m}
$$
\aligned
       \phi^*(\bar h(\mathcal V)) 
        &= \phi^*(h(\mathcal V)) + (\sum_j\, m_j\mu_j)\log|s| 
        \sim (-\mu(\mathcal V)(m)+\sum_j\, m_j\mu_j)\log|s|  \\ 
        &\sim h(m)(\nu,\nu^{\vee})\log|s|
\endaligned
$$
By Corollary \eqref{e.jumping}, we have the following two cases:
\begin{itemize}
\item[(a)] If $\sing_0(\nu)\neq 0$ then $h(m)>0$ and hence 
$\lim_{s\to 0}\, \phi^*(\bar h(\mathcal V)) = -\infty$.
\item[(b)] If $\sing_0(\nu)=0$ then $h(m)=0$ and hence $\phi^*(\bar h(\mathcal V))$ is bounded near zero.
Since $\bar h(\mathcal V)$ plurisubharmonic it follows from upper semicontinuity that 
$\bar h(\mathcal V)$ is bounded near zero.
\end{itemize}

\subsection{Mumford-Grothendieck Biextensions} There is a close
relationship between the concept of biextension variation
from~\eqref{eq:type-H} and the concept of a biextension introduced by
Mumford and studied extensively by Grothendieck in~\cite{SGA71}.
Essentially, as we vary the normal functions $\nu$ and $\omega$, the
set of all biextension variations $\mathcal{V}$ forms a biextension
(in the Mumford-Grothendieck sense) of
$\NF(S,\mathcal{H})\times\NF(S,\mathcal{H}^{\vee})$ by the sheaf
$\mathcal{O}_S^{\times}$.  (See Corollary~\ref{me3}.)  Starting in \S\ref{meb} we exploit this
connection, and we use it in \S\ref{tp} to define another pairing on
intersection cohomology with values in $\mathbb{Q}/\mathbb{Z}$ which we
use to answer Question~\ref{query:Q1}.

Unfortunately, the two terminologies clash.  We think that we have
solved the problems by writing the sections involving
Mumford-Grothendieck biextensions in such a way that it is always
clear which type of biextension we mean.  But we would like to warn
the reader (and apologize for the fact) that it will sometimes be
necessary to decide this based on context.

The Mumford-Grothendieck biextensions we deal with come out of
another notion studied by Grothendieck in~\cite{SGA71}, the notion of
\emph{extensions panach\'ees}.
These are objects $X$ in an abelian
category equipped with a two step filtration.  For example,
biextension variations are \emph{extensions panach\`ees} of
$\mathbb{Z}$ by $\mathcal{H}$ by $\mathbb{Z}(1)$.  In~\S\ref{meb}, we
work out the (slightly subtle) conditions under which the set of
isomorphism classes of \emph{extension panach\'ees} give rise to a
Mumford-Grothendieck biextension.  Then we apply these to certain
Mumford-Grothendieck biextensions, mainly of topological type, which
arise in connection with the asymptotic height pairing.
In \S\ref{gjp}, we use these results, to  generalize the asymptotic height pairing from the case of
variations with unipotent monodromy in the normal crossing case
to arbitrary pure variations
on $S$ in a neighborhood of some $s\in \bar S$ (where $S$ is a
Zariski open subset of a complex manifold $\bar S$). 

In this paper, we translate the term \emph{extensions panach\'ees} as ``mixed extensions.'' The (probably better) term ``blended extensions'' is used by
D.~Bertrand in the abstract to~\cite{Bertrand13}.

\subsection{Related Work} The thesis~\cite{Lear} of D. Lear contains the initial results of the asymptotics
of the height over a punctured disk.  An explicit example of the non-linearity of $\mu$ appears in 
Example $(5.44)$ of~\cite{GregJDG}.  The asymptotics of the biextension bundle also appear in 
the recent work of J. Burgos Gill, D. Holmes and R. de Jong \cite{BHR2} and \cite{BHJ}.  Another 
approach to limits of heights having to do with Feynman amplitudes is given in~\cite{ABGF}.  A preliminary discussion 
of heights from the viewpoint of log geometry and compactification of mixed period domains may be found
at the end of~\cite{MR2784750}, with additional results appearing in~\cite{KNU-IV}.

\subsection{Acknowledgments} This paper has taken a long time to write
up and, over the years, many people have helped us.  Questions
\eqref{query:Q1} and \eqref{query:Q2} were initially suggested to the
second author by R. Hain, and we are both very grateful to him for
posing the questions and for his support and interest. During the
authors' visit to the Institute for Advanced Study in 2004--2005,
P.~Griffiths introduced us to the notion of singularities of
normal functions and P.~Deligne wrote a letter to the second author
giving an example of jumps~\cite{DeligneLetter}.

\par The authors have also had extensive discussions of the asymptotics of the height pairing with
C. Schnell, M. Kerr and J. Lewis.  We thank B. Moonen for informing us about the sections in~\cite{SGA71} 
containing the \emph{extensions panach\'ees}.  We also thank H. Boas and S. Biard for helpful discussions 
concerning extensions of plurisubharmonic functions.

\par The authors thank Institute Henri Poincare and the Max Planck Institute for Mathematics for partial 
support during the final stages of this work.  This work was also partially supported by NSF grants DMS 
1361159 (P. Brosnan) and DMS 1361120 (G. Pearlstein). 

\section{Variations of Mixed Hodge Structure}\label{sec:vmhs}

\subsection*{Weights} Suppose $M$ is any object with a weight
filtration $W$.  For example, $M$ could be a mixed Hodge structure, a
variation of mixed Hodge structure or a mixed Hodge module.  We say
that $M$ has \emph{weights in an interval $[a,b]$} if $\Gr^W_k M=0$
for $k\not\in [a,b]$.

\subsection*{Mixed Hodge Structures} In this paper, all mixed Hodge structure 
are defined over $A=\mathbb Z$, $\mathbb Q$ or $\mathbb R$.  We let 
$\mathbb F=\mathbb Q$ or $\mathbb R$.  For an integer $n$, we let $A(n)$ denote
the pure $A$-Hodge structure of weight $-2n$ with underlying $A$-module $A$
[not $(2\pi i)^n A$].
If $H$ and $K$ are mixed Hodge structures and $n\in\mathbb{Z}$, an element 
$f\in\Hom_{\rm MHS}(H,K(-n))$ is called a $(n,n)$ morphism of Hodge structures.  
A morphism of mixed Hodge structures is a $(0,0)$ morphism.

\par A mixed Hodge structure $(F,W)$ induces a unique, functorial bigrading 
[Theorem $(2.13)$,~\cite{CKS86}]
\begin{equation}
            V_{\mathbb C} = \bigoplus_{p,q}\, I^{p,q}   \label{eq:D-bigrading}
\end{equation}
on the underlying complex vector space $V_{\mathbb C}$ such that 
\begin{itemize}
\item[(a)] $F^p = \oplus_{a\geq p}\, I^{a,b}$;
\item[(b)] $W_k = \oplus_{a+b\leq k}\, I^{a,b}$;
\item[(c)] $\bar{I^{p,q}} = I^{q,p}\mod\oplus_{a<q,b<p}\, I^{a,b}$.
\end{itemize}
A mixed Hodge structure is said to be split over $\mathbb R$ if 
$\bar{I}^{p,q} = I^{q,p}$.  In this case
\begin{equation}
              I^{p,q} = F^p\cap{\bar F^q}\cap W_{p+q}        \label{eq:split-eq-1}
\end{equation}

\par If $K$ is a field of characteristic zero and $W$ is an increasing 
filtration of a finite dimensional $K$-vector space $V$ then a grading of 
$W$ is a semisimple endomorphism $Y$ of $V$ with integral eigenvalues 
such that 
$$
        W_k = E_k(Y)\bigoplus W_{k-1}
$$
for each index $k$, where $E_k(Y)$ denotes the $k$-eigenspace of $Y$.  In 
particular, a mixed Hodge structure $(F,W)$ determines a $\mathbb C$-grading 
$Y = Y_{(F,W)}$ of $W$ by the rule
\begin{equation}
        E_k(Y) = \bigoplus_{p+q=k}\, I^{p,q}           \label{eq:D-grading}
\end{equation}
In what follows, the grading $Y_{(F,W)}$ will be called the Deligne grading of 
$(F,W)$. 

\par The set of all gradings of $W$ will be denoted $\mathcal Y(W)$.  The
subalgebra of $gl(V)$ consisting of all elements $\alpha$ such that 
$$
          \alpha(W_k)\subseteq W_{k-\ell}
$$
for all $k$ will be denoted $W_{-\ell}gl(V)$.

\begin{proposition} [Prop. $(2.2)$~\cite{CKS86}] $\exp(W_{-1}gl(V))$
acts simply transitively on $\mathcal Y(W)$ by the adjoint action.
\end{proposition}

\begin{proposition}\label{split-2} A mixed Hodge structure $(F,W)$ is split over
$\mathbb R$ if and only if $Y = \bar Y$ where $Y$ is the Deligne
grading of $(F,W)$.  
\end{proposition}
\begin{proof} If $(F,W)$ is $\mathbb R$-split then $Y = \bar Y$ since 
$\bar{I}^{p,q} = I^{q,p}$.  Conversely, suppose that $Y = \bar Y$ and
let $\oplus_{p+q=k}\, H^{p,q}$ be the Hodge decomposition of $Gr^W_k$.  Define
$$
        J^{p,q} = \{ v\in E_{p+q}(Y) \mid [v]\in H^{p,q}\subseteq Gr^W_{p+q}\}
$$
Then, $V_{\mathbb C} = \oplus_{p,q}\, J^{p,q}$ satisfies conditions 
$(a)$--$(c)$ of \eqref{eq:D-bigrading}:  Condition $(a)$ follows 
from the fact that $Y$ is a grading of $W$ which preserves $F$.
Condition $(b)$ follows from the fact that $Y$ is a grading of 
$W$.  Finally, condition $(c)$ follows from the fact that 
${\bar J}^{p,q} = J^{q,p}$.  Thus, $I^{p,q} = J^{p,q}$ by
the uniqueness of~\eqref{eq:D-bigrading}, and hence $(F,W)$ is split 
over $\mathbb R$. 
\end{proof}

\par If $(F,W)$ is a mixed Hodge structure with underlying complex vector 
space $V_{\mathbb C}$ then 
\begin{equation}
        \Lambda^{-1,-1}_{(F,W)} = \bigoplus_{p,q<0}\, gl(V_{\mathbb C})^{p,q}
        \label{eq:Lambda}
\end{equation}
Note that by property $(c)$ of \eqref{eq:D-bigrading}, $\Lambda^{-1,-1}$ is 
closed under complex conjugation.  It follows from the defining properties 
(a)--(c) of \eqref{eq:D-bigrading} that 
\begin{equation}
      e^{\lambda}I^{p,q}_{(F,W)} = I^{p,q}_{(e^{\lambda}F,W)}
      \label{eq:lamba-transform}
\end{equation}

\begin{theorem}[Theorem $(2.20)$~\cite{CKS86}] Given a mixed Hodge 
structure $(F,W)$ with underlying real vector space $V_{\mathbb R}$ there 
exists a unique real element
\begin{equation}
        \delta\in\Lambda^{-1,-1}_{(F,W)}\cap gl(V_{\mathbb R})       
		\label{eq:D-splitting}
\end{equation}
such that $(e^{-i\delta}F,W)$ is a mixed Hodge structure which is split over 
$\mathbb R$.  Moreover, $\delta$ commutes with all $(r,r)$-morphisms of 
$(F,W)$.
\end{theorem}

\par The splitting~\eqref{eq:D-splitting} is henceforth called the Deligne 
splitting or $\delta$-splitting of $(F,W)$.  The defining equation (ibid) 
for $\delta$ is 
\begin{equation}
         \overline{Y_{(F,W)}} = e^{-2i\text{\rm ad}\,\delta}Y_{(F,W)}  
			\label{eq:delta-def}
\end{equation}

\begin{definition}\label{1v} An unoriented biextension in the category of $A$-mixed Hodge 
structures is an $A$-mixed Hodge structure $(F,W)$ with weights in $[-2,0]$ 
such that $Gr^W_0$, $Gr^W_{-2}$ have rank 1.  
A biextension mixed Hodge structure is an unoriented biextension equipped 
with a choice of isomorphisms $Gr^W_0\cong A(0)$ and 
$Gr^W_{-2}\cong A(1)$.  The positive generators of $A(0)$ and $A(1)$
will be denoted $1$ and $1^{\vee}$ respectively.
\end{definition}

\par If $(F,W)$ is a biextension mixed Hodge structure then $\eta$ is 
the unique $(-1,-1)$-morphism of $(F,W)$ such that the induced map 
$Gr^W_0\to Gr^W_{-2}$ sends $1$ to $1^{\vee}$.  The group $(\mathbb C,+)$ acts 
additively on the set of biextension mixed Hodge structures by the rule
\begin{equation}
        t+(F,W) = (e^{t\eta}F,W)             \label{eq:linear-action}
\end{equation}
Since $\eta$ is a $(-1,-1)$-morphism, $t+(F,W) = (F,W)$ if and only if $t=0$. 

\begin{definition} The height of a biextension mixed Hodge
structure $(F,W)$ is the unique real number $h(F,W)$ such 
that 
\begin{equation}
	2\pi\delta_{(F,W)} = h(F,W)\eta          \label{eq:JDG-height}
\end{equation}
\end{definition}

\begin{lemma}  Let $(F,W)$ be a biextension mixed Hodge structure.  
Then,
\begin{itemize}
\item[(a)] For any complex number $t$, 
\begin{equation}
        h(e^{t\eta}F,W) = h(F,W) + 2\pi\text{\rm Im}(t)\eta  
	\label{eq:mult-prop}
\end{equation}
\item[(b)] $h(F,W)$ depends only the underlying isomorphism
class of $(F,W)$ as an $\mathbb R$-biextension, i.e. $h(F,W)$
is invariant under isomorphisms of $\mathbb R$-mixed Hodge
structures such that the induced maps
$$
    Gr^W_0\to Gr^W_0,\qquad Gr^W_{-2}\to Gr^W_{-2}
$$
are the identity.
\item[(c)] $h(F,W)=0$ if and only if $(F,W)$ is $\mathbb R$-split.
\end{itemize}
\end{lemma}
\begin{proof} Item $(a)$ follows by direct computation from equations
\eqref{eq:delta-def} and \eqref{eq:linear-action}.  Item $(b)$
follows from the fact that $\delta$ is an isomorphism invariant 
of $\mathbb R$-mixed Hodge structures, i.e. an isomorphism of 
$\mathbb R$-mixed Hodge structures induces an isomorphism of  
Deligne bigradings, and hence intertwines $\delta$-invariants.
Regarding item $(c)$, observe that by construction 
$\delta_{(F,W)}$ vanishes if and only if $(F,W)$ is split over 
$\mathbb R$.
\end{proof}

\par In the case where $X$ is a smooth complex projective variety of 
dimension $n$ and $Z$ and $W$ are null homologous cycles of dimensions $d$ 
and $e$ such that $|Z|\cap|W|$ is empty and $d+e=n-1$ there exists a 
canonical subquotient $(F,W)$ of the mixed Hodge structure on 
$H_{2d+1}(X-|W|,|Z|)$ such that $h(F,W)$ is the archimedean height of 
the pair $(Z,W)$.  See ~\cite{HainBiext} and equation equation $(5.4)$ 
of~\cite{GregJDG} for details.

\par If $(F,W)$ is a biextension mixed Hodge structure let $\tilde B = \tilde B(F,W)$
denote the $(\mathbb C,+)$-orbit of $(F,W)$ under~\eqref{eq:linear-action}.
Since $\eta$ is a $(-1,-1)$-morphism of $(F,W)$, any two elements of $\tilde B(F,W)$ induce 
the same mixed Hodge structures on $Gr^W_k$ as well as $W_0/W_{-2}$ and $W_{-1}$.  Let $B$ be the 
quotient of $\tilde B$ by $(A,+)\subset (\mathbb C,+)$.  By~\eqref{eq:mult-prop}, the height function 
\eqref{eq:JDG-height} descends to a function $h:B\to\mathbb R$.   For $b\in B$ we define
\begin{equation}
           |b| = \exp(-h(b))          \label{eq:height-def}
\end{equation}

\par In the case $A=\mathbb Z$, the set $B$ has a simply transitive $\mathbb C^*$ action 
given by  
\begin{equation}
         t.[(\tilde F,W)] = [(e^{\frac{1}{2\pi i}\log(t)\eta}\tilde F,W)]
         \label{eq:Gm-action}
\end{equation}
since $\frac{1}{2\pi i}\log(t)$ is well defined modulo $\mathbb Z$.  

\begin{proposition} Let $(F,W)$ be a $\mathbb Z$-biextension mixed Hodge structure
and $t\in\mathbb C^*$.  Then, $|t.b| = |t||b|$.
\end{proposition}
\begin{proof} By~\eqref{eq:mult-prop}, $h(t.b) = h(b) - \log|t|\eta$.
\end{proof}

\subsection*{Classifying Spaces} A polarization of a pure $A$-Hodge structure $H$ of weight $k$
is a $(-1)^k$-symmetric, non-degenerate bilinear form
$$
           Q:H\otimes H\to A(-k)
$$
which is a morphism of Hodge structure such that if $C$ is the Weil operator which acts on $H^{p,q}$ 
as $i^{p-q}$ then
$
            <u,v> = Q(Cu,\bar v)
$
is a positive definite hermitian inner product on $H_{\mathbb C}$.

\par A graded-polarization of an $A$-mixed Hodge structure $(F,W)$ is a 
collection of non-degenerate bilinear forms  
$$
       Q_k:Gr^W_k\otimes Gr^W_k\to A(-k)
$$
which polarize the pure Hodge structure $FGr^W_k$ for each index $k$.

\par Given a graded-polarized mixed Hodge structure $\{(F,W),Q_{\bullet}\}$ with underlying $A$-module $V_A$
let $X$ denote the generalized flag variety consisting of all decreasing filtrations $\tilde F$
of $V_{\mathbb C}$ such that 
$
       \dim\, \tilde F^p = \dim F^p
$
for each index $p$.  Let $\mathcal M\subset X$ denote the classifying space of all $\tilde F\in X$
such that 
\begin{itemize}
\item[---] $(\tilde F,W)$ is a mixed Hodge structure with the same graded Hodge numbers as $(F,W)$;
\item[---] $Q_{\bullet}$ is a graded-polarization of $(\tilde F,W)$.
\end{itemize}

\par Let $GL(V_{\mathbb C})^W$ denote the subgroup of $GL(V_{\mathbb C})$ consisting of elements which
preserve $W$.  Define $G_{\mathbb C}$ to be the complex subgroup of $GL(V_{\mathbb C})^W$ consisting of 
elements which act by isometries of $Q_{\bullet}$ on $Gr^W$.  Let 
$G_{\mathbb R} = G_{\mathbb C}\cap GL(V_{\mathbb R})$ and $G\subset G_{\mathbb C}$ be the set consisting
of elements of $G_{\mathbb C}$ which act by real isometries of $Q_{\bullet}$ on $Gr^W$.  Then, $G$ is a 
real Lie group which acts transitively on $\mathcal M$ by biholomorphisms~\cite{PearlsteinManuscripta}.  
The "compact dual" $\check{\mathcal M}$ is defined to be the $G_{\mathbb C}$ orbit of $F$ in $X$.
In general, $\check{\mathcal M}$ is not compact due to the fact that $G_{\mathbb C}$ contains the 
complex subgroup $G_{-1}$ consisting of element of $GL(V_{\mathbb C})^W$ which act trivially on 
$Gr^W$.  Since $G$ also contains $G_{-1}$, viewed as a real Lie group, it follows that 
$G_{\mathbb C}$ is not the complexification of $G$ unless $G_{-1} = 1$.

\par See~\cite{PearlsteinManuscripta} and the references therein for further details on classifying spaces 
of mixed Hodge structures.  In general, variations of mixed Hodge structure do not have good asymptotic 
behavior in the absence of the existence of a graded-polarization.  

\subsection*{Period Maps.} The axioms of an admissible variation of mixed Hodge structure are given
in~\cite{SZ,Kashiwara}.  As described in~\cite{PearlsteinManuscripta}, given a  variation of 
(graded-polarized) mixed Hodge structure $\mathcal V\to S$ we let
\begin{equation}
           \varphi:S\to\Gamma\backslash\mathcal M         \label{eq:period-map}
\end{equation}
denote the corresponding period map, determined by a choice of point 
$s_0\in S$ and $\Gamma$ is the monodromy group of $\mathcal V$.  

\par For the remainder of this section, we fix a polydisk $\Delta^r$ with 
local coordinates $(s_1,\dots,s_r)$ and let $\Delta^{*r}$ denote the
complement of $s_1\cdots s_r=0$.   Let $(z_1,\dots,z_r)$ denote the standard 
Euclidean coordinates on $\mathbb C^r$ and $\mathfrak h^r$ be the product of 
upper half-planes defined by ${\rm Im}(z_1),\dots,{\rm Im}(z_r)>0$.  Let 
$\mathfrak h^r\to\Delta^{*r}$ be the universal covering defined by 
$s_j = e^{2\pi i z_j}$.

\par An admissible period map~\cite{SZ,Kashiwara} 
$\varphi:\Delta^{*r}\to\Gamma\backslash\mathcal M$ with unipotent monodromy
then has a local normal form~\cite{PearlsteinManuscripta}
\begin{equation}
        F(z_1,\dots,z_r) = e^{\sum_j\, z_j N_j}e^{\Gamma(s)}F_{\infty}    
        \label{eq:lnf}
\end{equation}
where $T_j = e^{N_j}$, $F_{\infty}$ is the limit Hodge filtration and 
$\Gamma(s)$ is a holomorphic function on $\Delta^r$ which vanishes at the 
origin and takes values in a vector space complement to the stabilizer of 
$F_{\infty}$ in $Lie(G_{\mathbb C})$.  See 
\cite{PearlsteinManuscripta} for details.

\par The local normal form~\eqref{eq:lnf} for variations of pure Hodge structure 
appears in~\cite{Luminy}, and works in the setting of admissible $\mathbb R$-VMHS 
with unipotent monodromy.  Indeed, the special case $\Gamma(s) = 0$ is an admissible
nilpotent orbit (or an infinitesimal mixed Hodge module).

\par Passage from the period map $\varphi:\Delta^{*r}\to\Gamma\backslash\mathcal M$
to the local normal form~\eqref{eq:lnf} may involve replacing $\Delta^{*r}$ by a polydisk
of smaller radius.  Since we are only interested in the asymptotic behavior of the period
map in our discussions involving the local normal form, we generally omit this step.
However, when it becomes necessary to explicitly restrict to a smaller disk we write
\begin{equation}
        \Delta^r_a=\{\, s\in\Delta^r\mid |s_1|,\dots,|s_r|<a\,\}.      \label{eq:small-disk}
\end{equation}

\begin{remark} Variations of pure, polarized Hodge structure with unipotent
monodromy are admissible by the results of Schmid~\cite{Schmid73} and
Cattani, Kaplan and Schmid~\cite{CKS86}.  
If $\mathcal V$ is a variation of pure Hodge structure,
we write $\check{\mathcal D}$ and $\mathcal D$ in place of 
$\check{\mathcal M}$ and $\mathcal M$.  
\end{remark}

\subsection{Nilpotent Orbits}  The nilpotent orbit attached to an
admissible variation $\mathcal V\to\Delta^{*r}$ is the map
$\theta(z) = \exp(\sum_j\, z_j N_j)F_{\infty}$ from 
$\mathbb C^r\to\check{\mathcal M}$
obtained by setting $\Gamma=0$ in~\eqref{eq:lnf}.  In language 
of~\cite{Kashiwara}, the data $(N_1,\dots,N_r;F,W)$
determines an infinitesimal mixed Hodge module (IMHM)
We shall often use the alternative term admissible nilpotent orbit instead of IMHM.  In the
case where $\mathcal M$ is a classifying space of pure Hodge structures,
an IMHM is just a nilpotent orbit in the sense of Schmid~\cite{Schmid73}.
If $\mathcal V\to\Delta^{*r}$ is an admissible variation of mixed Hodge
structure, we also use the notation $\mathcal V_{\nilp}$ to denote the
associated nilpotent orbit.

\begin{definition} Let $\mathcal D$ be a classifying space of pure Hodge
structures of weight $k$ polarized by $Q$ and $\mathfrak g_{\mathbb R}$ denote the Lie
algebra consisting of infinitesimal, real automorphisms of $Q$.  Then,
an $r$-variable nilpotent orbit $\theta$ with values in $\mathcal D$
is a map $\theta:\mathbb C^r\to\check{\mathcal D}$ of the form
\begin{equation}
       \theta(z) = \exp(\sum_j\, z_j N_j)F            \label{eq:nil-orbit-1}
\end{equation}
where $N_1,\dots,N_r\in\mathfrak g_{\mathbb R}$ are commuting nilpotent
endomorphisms and $F\in\check{\mathcal D}$ such that
\begin{itemize}
\item[(i)] $N_j(F^p)\subseteq F^{p-1}$;
\item[(ii)] There exists a constant $a$ such that if $\Im(z_1),\dots,\Im(z_r)>a$
then $\theta(z)\in\mathcal D$.
\end{itemize}
\end{definition}

\par In particular, since $\mathcal D$ encodes the weight of the Hodge structures
it parametrizes, the data of a nilpotent orbit with values in $\mathcal D$ reduces 
to $(N_1,\dots,N_r;F)$. If the data of $\mathcal D$ is defined over a subfield $K$ 
of $\mathbb R$, one can define the notion of an $K$-nilpotent orbit by requiring 
that $N_1,\dots,N_r\in\mathfrak g_K = Lie({\rm Aut}_K(Q))$.

\begin{lemma}\label{scale-nilp} If $(N_1,\dots,N_r;F)$ generate a $K$-nilpotent 
orbit of pure Hodge structure then so does 
$$
     (t_1 N_1,\dots,t_r N_r;F)
$$ 
for any choice of positive scalars $t_1,\dots,t_r\in K_+$.
\end{lemma}
\begin{proof} Conditions $(i)$ and $(ii)$ remain true
upon replacing $N_j$ by $t_j N_j$.  
\end{proof}

\par To continue, we recall that if $N$ is a nilpotent endomorphism of a finite 
dimensional vector space $V$ defined over a field of characteristic zero then there
exists a unique, increasing filtration $W=W(N)$ of $V$ such that [cf.~\cite{CKS86}]
\begin{itemize}
\item[(a)] $N(W_j)\subseteq W_{j-2}$;
\item[(b)] $N^{\ell}:Gr^W_{\ell}\to Gr^W_{-\ell}$ is an isomorphism;
\end{itemize}
for all $j$ and $\ell$.

\begin{theorem}\label{mono-cone} (See~\cite{CKS86}) Let $\theta$ be a nilpotent orbit
of pure Hodge structure of weight $k$ and 
$$
      \mathcal C = \{\sum_j\, a_j N_j \mid a_1,\dots,a_r>0\,\}
$$
Then, the map $N\mapsto W(N)$ is constant on $\mathcal C$.
Pick $N\in\mathcal C$ and define $W(N)[-k]_j = W_{j-k}(N)$.
Then,
\begin{equation}
        (F,W(N)[-k])              \label{eq:lmhs-1}
\end{equation}
is a mixed Hodge structure with respect to which each $N_j$
is a $(-1,-1)$-morphism.
\end{theorem}

\par The pair \eqref{eq:lmhs-1} is called the limit mixed Hodge structure of $\theta$.
Analysis of the possible nilpotent orbits which can arise can then be reduced
to the study of nilpotent orbits with limit mixed Hodge structure
which are split over $\mathbb R$ by the following results of 
Cattani, Kaplan and Schmid~\cite{CKS86}.

\begin{theorem}\label{eq:rsplit-orbit}~\cite{CKS86} If $\theta(z_1,\dots,z_r)$ is a 
nilpotent orbit of pure Hodge structure such that \eqref{eq:lmhs-1} is 
split over $\mathbb R$ then $\theta(z)$ takes values in the appropriate 
classifying space of pure Hodge structure as soon as $\Im(z_1),\dots,\Im(z_r)>0$.
\end{theorem}

\begin{theorem}\cite{CKS86}\label{eq:rsplit-orbit-2} If $\theta(z_1,\dots,z_r)$ is a nilpotent
orbit of pure Hodge structure and $(\tilde F,W) = (e^{-i\delta}F,W)$
is the $\delta$-splitting of the limit mixed Hodge structure 
\eqref{eq:lmhs-1} then $\tilde\theta(z_1,\dots,z_r) = e^{\sum_j\, z_jN_j}\tilde F$ is 
a nilpotent orbit of pure Hodge structure.
\end{theorem}


\par The $\text{\rm SL}_2$-orbit theorem of \cite{Schmid73,CKS86}
involves the $\text{\rm sl}_2$ or canonical splitting
\begin{equation}
          (F,W)\mapsto (e^{-\xi}F,W)         \label{eq:sl2-split}
\end{equation}
where $\xi$ is given by universal Lie polynomials in the Hodge
components of $\delta$.  In particular, the mixed Hodge structures 
$(F,W)$, $(e^{-i\delta}F,W)$ and $(e^{-\xi}F,W)$ all induce the same
pure Hodge structures on $Gr^W$ since $\delta,\xi\in\Lambda^{-1,-1}_{(F,W)}$.

\begin{definition}\label{orbit-split} Let $\theta(z) = e^{\sum_j\, z_j N_j}F$ 
be a nilpotent orbit with values in $\mathcal D$ and limit mixed Hodge structure 
$(F,W)$.  Let $(\tilde F,W)$ be the Deligne $\delta$-splitting of $(F,W)$.  Then,
$
       \tilde\theta(z) = e^{\sum_j\, z_j N_j}\tilde F
$
is called the $\delta$-splitting of $\theta$.  Likewise, if $(\hat F,W)$
is the $\text{\rm sl}_2$-splitting of $(F,W)$ then 
$
        \hat\theta(z) = e^{\sum_j\, z_j N_j}\hat F
$
is called the $\text{\rm sl}_2$-splitting of $\theta$. 
\end{definition}

\begin{remark}\label{K-orbit} Theorem~\eqref{eq:rsplit-orbit-2} holds mutatis mutandis for $\hat\theta$.
Since the property of being a $K$-nilpotent orbit does not involve the limit Hodge filtration,
passage to the $\delta$ or $\text{\rm sl}_2$-splitting does not effect the property of 
being a $K$-nilpotent orbit.
\end{remark}

\par For later use, we record that if $(F,W)$ is a biextension mixed
Hodge structure then 
\begin{equation}
          W_{-2}gl(V) = gl(V)^{-1,-1}_{(F,W)}    \label{eq:w2}
\end{equation}
for every point $F\in\mathcal M$.  Moreover, in the graded-polarized
case,
\begin{equation}
          \eta\in Z(\mathfrak g_{\mathbb C})      \label{eq:eta-central}
\end{equation}
i.e. $\eta$ commutes with every element of $\mathfrak g_{\mathbb C}$.
Indeed, any element of $\mathfrak g_{\mathbb C}$ must act trivially
on $Gr^W_0$ and $Gr^W_{-2}$ and $\eta$ maps $Gr^W_0$ to $Gr^W_{-2}$
and annihilates $W_{-1}(V)$.

\subsection{Heights of Nilpotent Orbits}  Let $V$ be a finite
dimensional vector space over a field of characteristic zero
equipped with a nilpotent endomorphism $N$ and increasing 
filtration $W$ such that $N(W_k)\subseteq W_k$ for each index
$k$.  Then, there exists at most one increasing filtration 
$M=M(N,W)$, called the relative weight filtration of $W$
with respect to $N$ (see~\cite{SZ}), such that
\begin{itemize}
\item[(a)] $N(M_k)\subseteq M_{k-2}$ for all $k$;
\item[(b)] For any $k$ and $\ell$, the induced map
$
            N^{\ell}:Gr^M_{\ell+k}Gr^W_k\to Gr^M_{k-\ell}Gr^W_k
$
is an isomorphism. 
\end{itemize}

\begin{definition}~\cite{SZ} The data $(N,F,W)$ defines an admissible 
nilpotent orbit $\theta(z) = (e^{zN}F,W)$ with values in the classifying space 
$\mathcal M$ of graded-polarized mixed Hodge structure provided that
\begin{itemize}
\item[---] $N$ is nilpotent, preserves $W$ and acts by infinitesimal 
isometries on $Gr^W$;
\item[---] $M=M(N,W)$ exists; 
\item[---] $F\in\check{\mathcal M}$ and $N(F^p)\subseteq F^{p-1}$;
\item[---] $e^{zN}F$ induces nilpotent orbits of pure, polarized Hodge 
structure on $Gr^W$.
\end{itemize}
\end{definition}

\begin{theorem}~\cite{SZ} If $(F,N,W)$ defines an admissible nilpotent orbit
and $M=M(N,W)$ then $(F,M)$ is a mixed Hodge structure for which $N$ is a 
$(-1,-1)$-morphism.
\end{theorem}

\par As in the pure case, $(F,M)$ is called the limit mixed Hodge 
structure of the nilpotent orbit generated by $(N,F,W)$.  An admissible nilpotent 
orbit $\theta$ induces nilpotent orbits of pure Hodge structure on $Gr^W$ such 
that the limit mixed Hodge structure of $\theta$ induces the limit mixed Hodge 
structure of the graded orbits.  The analogs of Theorems~\eqref{eq:rsplit-orbit}
and \eqref{eq:rsplit-orbit-2} hold mutatis mutandis in the mixed case.

\par As in Definition~\eqref{orbit-split}, one defines the $\delta$-splitting
and the $\text{\rm sl}_2$-splitting of an admissible nilpotent orbit $\theta$ 
by replacing by replacing the limit mixed Hodge structure of $\theta$ with the 
corresponding Deligne or $\text{\rm sl}_2$-splitting.

\begin{lemma}[Deligne~\cite{DL,BP-Comp}]\label{DL-1} Let $(N,F,W)$ generate
an admissible nilpotent orbit $\theta$ with $\delta$-splitting generated by
$(N,\tilde F,W)$.  Define,
\begin{equation}
           Y(N,F,W) = \text{\rm Ad}(e^{-iN})Y_{(e^{iN}\tilde F,W)},\qquad
           \tilde Y(N,F,W) = Y_{(\tilde F,M)}   
           \label{eq:DL-grading}
\end{equation}
Then, writing $Y = Y(N,F,W)$ and $\tilde Y = \tilde Y(N,F,W)$,
\begin{itemize}
\item[(i)] $Y = \bar Y$;
\item[(ii)] $Y$ preserves $\tilde F$;
\item[(ii)] $Y$ commutes with $\tilde Y$ ($\implies Y$ preserves $M$);
\item[(iii)] If $N = \sum_{j\geq 0}\, N_{-j}$ relative to ${\rm ad}\, Y$ then $N_0$ and $H = \tilde Y - Y$ 
form an $\text{\rm sl}_2$-pair, with associated  $\text{\rm sl}_2$-triple $(N_0,H,N_0^+)$.
\item[(iv)] $[N-N_0,N_0^+] = 0$.
\end{itemize}
\end{lemma}

\begin{remark} $ $
\begin{itemize}
\item[---] In the language of Deligne systems, $Y$ is the grading attached to $N$ and 
$\tilde Y$ (cf. \cite{BP-Comp}).
\item[---] For $j>0$ it follows that $N_{-j}$ is highest weight $j-2$ for $(N_0,H,N_0^+)$.
This forces $N_{-1} = 0$ and $[N_0,N_{-2}] = 0$.
\item[---] In the case where $(e^{zN}\tilde F,W)$ is a biextension mixed Hodge structure 
then $N = N_0 + N_{-2}$.
\end{itemize}
\end{remark}

\begin{theorem}[Theorem $(5.19)$~\cite{GregJDG}]\label{mu} For an admissible
biextension variation $\mathcal V\to\Delta^*$ with nilpotent orbit generated
by $(N,F,W)$,  
\begin{equation}
         \mu(\mathcal V) = N_{-2}\eta          \label{eq:mu-def-2}
\end{equation}
where $N = N_0 + N_{-2}$ defined as in Lemma~\eqref{DL-1}.
\end{theorem}
\begin{proof} Theorem $(5.19)$ of~\cite{GregJDG} asserts that any grading $Y$ of $W$ such 
that $[Y,N]\in W_{-2}gl(V)$ can be used to compute $\mu(\mathcal V)$ using~\eqref{eq:mu-def-2}.
\end{proof}

\begin{definition} Let $(N,F,W)$ generate an admissible nilpotent orbit
of biextension type with Deligne splitting $\delta$ of the limit mixed Hodge structure
and $Y=Y(N,F,W)$ defined as in~\eqref{eq:DL-grading}.  Let 
$\delta = \sum_{k\leq 0}\,\delta_k$ with $[Y,\delta_k] = k\delta_k$.  Then, 
the limit height $ht(N,F,W)$ of the nilpotent orbit $(e^{zN}F,W)$ is given 
by the formula
\begin{equation}
            2\pi\delta_{-2} = ht(N,F,W)\eta           \label{eq:limit-height}
\end{equation}
\end{definition} 

\begin{remark} A biextension mixed Hodge structure $(F,W)$ on $V$ determines an admissible
nilpotent orbit by setting $N=0$.  In this case $M=W$, $\delta$ is the 
Deligne splitting of $(F,W)$ and $Y(N,F,W)$ is a grading of $W$.  Since 
$\delta\in W_{-2}gl(V)$ it then follows from the short length of $W$ that 
$\delta = \delta_{-2}$ relative to $Y(N,F,W)$ and hence $ht(N,F,W) = h(F,W)$.
\end{remark}

\section{Regularity Results}\label{sec:reg}  The following result justifies the name 
limit height (Theorem \eqref{disk}): Let $\mathcal V\to\Delta^*$ be an admissible 
biextension VMHS with nilpotent orbit $(e^{zN}F_{\infty},W)$ and asymptotic 
height
$
         h(\mathcal V)\cong -\mu\log|s|
$.
Then,
$$
        \bar h(s) = h(\mathcal V) + \mu\log|s|
$$
has a continuous extension to $\Delta$ with $\bar h(0) = ht(N,F_{\infty},W)$.

\subsection*{$\text{\rm SL}_2$-orbits} In this section we review the theory of ${\rm SL}_2$ orbits for
nilpotent orbits of biextension type following~\cite{GregJDG}.  
Let $\mathcal M$ be the ambient classifying space with associated Lie group $G$ and 
$\tilde G$ be the subgroup of $G$ which acts as real transformations on $W_k/W_{k-2}$ for all $k$.

\par Let $(e^{zN}F,W)$ be an admissible nilpotent orbit of biextension type with 
limit mixed Hodge structure $(F,M)$, let $(\tilde F,M) = (e^{-i\delta}F,M)$ denote
the Deligne splitting of $(F,M)$ and $\tilde Y = Y(N,F,W)$, $Y = Y(N,F,W)$, and 
$N=N_0 + N_{-2}$, $N_0^+$ etc. as in Lemma \eqref{DL-1}.   

\par To continue, we recall the $SL_2$-orbit Theorem of~\cite{GregJDG}:

\begin{theorem}[Theorem 4.2~\cite{GregJDG}]\label{SL2-JDG} Let $(e^{zN}F,W)$ be an admissible nilpotent 
orbit of biextension type.  Then, there exists 
$$
    \chi\in Lie(\tilde G)\cap\ker(N)\cap\Lambda^{-1,-1}_{(\tilde F,M)}
$$
and distinguished real analytic function $g:(a,\infty)\to\tilde G$ such that 
\begin{itemize}
\item[(a)] $e^{iyN}.F = g(y)e^{iyN}.\tilde F$;
\item[(b)] $g(y)$ and $g^{-1}(y)$ have convergent series expansions about
$\infty$ of the form
$$
\aligned
       g(y) &= e^{\chi}(1 + g_1 y^{-1} + g_2 y^{-2} + \cdots)      \\
  g^{-1}(y) &= (1 + f_1 y^{-1} + f_2 y^{-2} + \cdots)e^{-\chi}         
\endaligned
$$
with $g_k$, $f_k\in \ker(\ad N_0)^{k+1}\cap\ker(\ad N_{-2})$;
\item[(c)] $\delta$, $\chi$ and the coefficients $g_k$ are related by
the formula
$$
      e^{i\delta} 
      = e^{\chi}\left(1 + \sum_{k>0}\, \frac{1}{k!}(-i)^k(\ad\, N_0)^k\,g_k
                 \right)
$$
\end{itemize}
Moreover, $\chi$ can expressed as a universal Lie polynomial over $\Q(\sqrt{-1})$ in the Hodge components 
$\delta^{r,s}$ of $\delta$ with respect to $(\tilde F,W)$.  Likewise, the coefficients $g_k$ and $f_k$  
can be expressed as universal, non-commuting polynomials over $\Q(\sqrt{-1})$ in $\delta^{r,s}$ and $\ad\, N_0^+$.
\end{theorem}

\begin{corollary}[Corollary $(4.3)$~\cite{GregJDG}]\label{jdg-cor} Let $\mathcal V\to\Delta^*$ be an 
admissible biextension variation, with period map $F(z):\mathfrak h\to\mathcal M$ and nilpotent orbit 
$e^{zN}F$.  Let $F_0 = e^{iN_0}\tilde F$.  Then, (cf.~Theorem~\eqref{SL2-JDG}), there exists a 
distinguished, real--analytic 
function $\gamma(z)$ with values in $Lie(\tilde G)$ such that, for $\text{\rm Im(z)}$ sufficiently large,
\begin{itemize}
\item[(i)]  $F(z) = e^{xN}g(y)e^{iyN_{-2}}y^{-H/2}e^{\gamma(z)}.F_0$;
\item[(ii)] $|\gamma(z)| = O(\Im(z)^{\beta}e^{-2\pi\text{\rm Im}(z)})$ as $y\to\infty$ and
$x$ restricted to a finite subinterval of $\mathbb R$, for some constant $\beta\in\R$.
\end{itemize}
\end{corollary}

\subsection*{Punctured Disk.} Let $\mathcal V\to\Delta^*$ be an admissible biextension with nilpotent
orbit $(e^{zN}F,W)$ and define $\delta$, $Y$, $\tilde Y$ etc. as in the previous subsection.  Observe 
that since $\tilde Y$ and $Y$ commute, we can write 
$$
          \delta = \sum_{a,b}\, \delta_{a,b},\qquad [\tilde Y,\delta_{a,b}] = a\delta_{a,b},\quad
                                                    [Y,\delta_{a,b}] = b\delta_{a,b}
$$
Define $\chi = \sum_{a,b}\, \chi_{a,b}$ similarly.  Let 
$$
       \delta(j) = \sum_{a-b=j}\,\delta_{a,b},\qquad
       \chi(j) = \sum_{a-b=j}\chi_{a,b}
$$
and note that $\delta(j)$ and $\chi(j)$ are the respective projections of $\delta$ and $\chi$
onto the $j$-eignespace of $H$.  In particular, since $N = N_0 + N_{-2}$ and $N_{-2}$ is central in 
$\mathfrak g_{\mathbb C}$ it follows that $[N_0,\delta(j)] = 0$ and $[N_0,\chi(j)] = 0$. Consequently, 
$\delta(j)$ and $\chi(j)$ vanish for $j>0$.

\begin{lemma}\label{g-twist} $g^{\ddag}(y) = \text{\rm Ad}(y^{H/2})g(y)$ is given by a convergent power series
in powers of $y^{-1/2}$ with 
$$
	\lim_{y\to\infty}\, g^{\ddag}(y) = e^{i\delta_{-2}}
$$
where $\delta = \delta_0 + \delta_{-1} + \delta_{-2}$ relative to $\text{\rm ad}\,Y$.
\end{lemma}
\begin{proof} To compute the limiting value, it is sufficient to show that
\begin{itemize}
\item[(a)] $\lim_{y\to\infty}\, \text{\rm Ad}(y^{H/2})(e^{-\chi}g(y)) =  1$; 
\item[(b)] $\lim_{y\to\infty}\, \text{\rm Ad}(y^{H/2})e^{\chi} =  e^{i\delta_{-2}}$;
\end{itemize}
The intermediate computations will show that $g^{\ddag}(y)$ has a power series expansion of the 
stated form.

\par To verify $(a)$, recall that 
$$
	g(y) = e^{\chi}(1 + g_1 y^{-1} + g_2 y^{-2} + \cdots)   
$$
with $g_k\in\ker(\text{\rm ad}(N_0)^{k+1})$. Let $g(y)(j)$ denote the projection of $e^{-\chi}g(y)$
to the $j$-eigenspace of $\text{\rm ad}(H)$.   (Note the slight conflict with the notation
used above for $\delta$ and $\chi$.)  Define
$$
       g_-(y) = \sum_{j<0}\, g(y)(j),\qquad
       g_0(y) = g(y)(0),\qquad g_+(y) = \sum_{j>0}\, g(y)(j)
$$
Then,
$$
          e^{-\chi}g(y) = g_-(y) + g_0(y) + g_+(y).
$$
By construction, $\text{\rm Ad}(y^{H/2})g_-(y)$ is a convergent power series in $y^{-1/2}$ with 
constant term zero.  Likewise, $\text{\rm Ad}(y^{H/2})g_0(y) = g(y)(0)$ is a convergent series in 
$y^{-1/2}$ with constant term $1$.

\par To continue, observe that $g_k(j) = 0$ unless $k\geq j$ since $\text{\rm ad}(N_0)^{k+1}g_k = 0$.  
Consequently, for $j>0$ we have 
$$
       g(y)(j) = \sum_{k\geq j}\, g_k(j)y^{-k} = y^{-j}\sum_{k\geq j}\, g_k(j)y^{j-k}
$$
As such, for $j>0$, $\text{\rm Ad}(y^{H/2})g(y)(j)$ is a convergent power series in $y^{-1/2}$
with constant term zero. Combining this with the previously computations, this shows that 
$g^{\ddag}(y)$ has a convergent series expansion in $y^{-1/2}$ with constant term $1$.  This
establishes $(a)$.

\par To verify $(b)$, we note that 
$$
       e^{iyN}e^{i\delta}\tilde F = e^{\chi}e^{-\chi}g(y)e^{iyN}\tilde F.
$$
Since $N_{-2}$ is central in $\mathfrak g_{\mathbb C}$, it follows that 
$$
      e^{iyN_0}e^{i\delta}\tilde F = e^{\chi}e^{-\chi}g(y)e^{iyN_0}\tilde F. 
$$
Applying $y^{H/2}$ to each side of this equation, and using the fact that $H$ preserves 
$\tilde F$ it follows that 
$$
       e^{iN_0}e^{i\sum_{j\leq 0}y^{j/2}\delta(j)}\tilde F 
       = e^{\sum_{j\leq 0}\, y^{j/2}\chi(j)}\text{\rm Ad}(y^{H/2})(e^{-\chi}g(y))e^{iN_0}\tilde F 
$$
Using part $(a)$, and taking the limit at $y\to\infty$ yields
$$
    e^{iN_0}e^{i\delta(0)}\tilde F = e^{\chi(0)}e^{iN_0}\tilde F 
$$
In particular, since $[N_0,\chi] = 0$ it follows that $[N_0,\chi(j)] = 0$ for each $j$.
Applying $e^{-iN_0}$ to the previous equation therefore implies
$$
       e^{i\delta(0)}\tilde F = e^{\chi(0)}\tilde F    
$$
Finally, $\delta(0)$ and $\chi(0)$ take values in $\Lambda^{-1,-1}_{(\tilde F,M)}$ and the 
map
$
        \lambda\in \Lambda^{-1,-1}_{(\tilde F,M)}\to e^{\lambda}\tilde F
$
is injective.  Thus, $i\delta(0) = \chi(0)$.    

\par To finish the proof of part $(b)$, observe that $\delta$ is real and $\chi$ is required
to act by real transformations on $W_k/W_{k-2}$.   Therefore $i\delta(0) = \chi(0)$ implies 
that $\chi(0) = \chi_{-2,-2} = i\delta_{-2,-2}$.  Indeed, 
$$
       i\delta(0) = \chi(0) \iff i\delta_{p,p} = \chi_{p,p}
$$
for each $p$.  In order for $\chi$ to act by real transformations on $W_k/W_{k-2}$ we must have
$p=-2$.  Since $W_{-2}Lie(\mathbb G_{\mathbb C})$ is central, 
$$
           [\tilde Y,\delta_{-2}] = [H+Y,\delta_{-2}] = [Y,\delta_{-2}] = -2\delta_{-2}
$$
Consequently, $\chi(0)=i\delta_{-2,-2} = i\delta_{-2}$.  Using the fact that
$\chi(j)=0$ for $j>0$, the computation of the limit in (b) then follows
directly.  
\end{proof}

\begin{theorem}\label{disk} Let $\mathcal V\to\Delta^*$ be an admissible biextension with unipotent monodromy 
$T=e^N$ and limit mixed Hodge structure $(F,M)$. Then
$$
         \lim_{s\to 0}\, h(s) = \lim_{s\to 0}\, h(\mathcal V) + \mu\log|s| = ht(N,F,W).
$$
\end{theorem} 
\begin{proof}  Let $(\tilde F,M) = (e^{-i\delta}F,M)$ be Deligne's $\delta$-splitting of the limit mixed Hodge
structure $(F,M)$ of $\mathcal V$.  Let $F(z)$ be the lifting of the period map of $\mathcal V$ to the upper 
half-plane. By Corollary~\eqref{jdg-cor}, we have
$$
      \delta_{(F(z),W)} = \delta_{(e^{xN}g(y)e^{iyN_{-2}}y^{-H/2}e^{\gamma(z)}.F_0,W)}
$$
where $F_0 = e^{iN_0}\tilde F$, and we use $g.F_0$ instead of $gF$ to break up long computations.
Moreover, $|\gamma(z)| = O(\text{\rm Im}(z)^{\beta}e^{-2\pi\text{\rm Im}(z)})$ as $y\to\infty$ and
$x$ restricted to a finite subinterval of $\mathbb R$.  Since $xN$ is real, it follows from
that  
$$
\aligned
     \delta_{(F(z),W)} &= \text{\rm Ad}(e^{xN})\delta_{(g(y)e^{iyN_{-2}}y^{-H/2}e^{\gamma(z)}.F_0,W)}   \\
                       &= \delta_{(g(y)e^{iyN_{-2}}y^{-H/2}e^{\gamma(z)}.F_0,W)}.
\endaligned
$$
Likewise, since $N_{-2}\in W_{-2}\mathfrak g_{\mathbb C}$ is central,
it follows (using \eqref{eq:JDG-height} and \eqref{eq:mult-prop}) that
$$
\aligned
     \delta_{(F(z),W)} &= \delta_{(g(y)e^{iyN_{-2}}y^{-H/2}e^{\gamma(z)}.F_0,W)} \\
                       &= \delta_{(g(y)y^{-H/2}e^{\gamma(z)}.F_0,W)} + yN_{-2}.
\endaligned
$$
Consequently,
$$
\aligned
       \delta_{(F(z),W)} - yN_{-2} &= \delta_{(g(y)y^{-H/2}e^{\gamma(z)}.F_0,W)} \\
                                   &= \delta_{(y^{-H/2}y^{H/2}g(y)y^{-H/2}e^{\gamma(z)}.F_0,W)} \\
                                   &= \text{\rm Ad}(y^{-H/2})\delta_{(y^{H/2}g(y)y^{-H/2}e^{\gamma(z)}.F_0,W)} \\
                                   &=  \delta_{(y^{H/2}g(y)y^{-H/2}e^{\gamma(z)}.F_0,W)}
\endaligned
$$
since $W_{-2} gl(V)\subset Z(\mathfrak g_{\mathbb C})$ and hence ${\rm Ad}(y^{-H/2})$ acts trivially
on $W_{-2} gl(V)$.  Therefore, 
\begin{equation}
        \lim_{\text{\rm Im}(z)\to\infty}\,
            \delta_{(F(z),W)} - yN_{-2} = \delta_{(e^{i\delta_{-2}}e^{iN_0}.\tilde F,W)}
         \label{eq:step-n}
\end{equation}
for $\text{\rm Re}(z)$ restricted to an interval of finite length.

\par To finish the proof, we note that $(e^{iN_0}.\tilde F,W)$ is a mixed Hodge structure which is 
split over $\mathbb R$ with $Y_{(e^{iN_0}.\tilde F,W)} = Y$.   
Consequently,
$$
        \delta_{(e^{i\delta_{-2}}e^{iN_0}.\tilde F,W)} 
            = \delta_{-2} + \delta_{(e^{iN_0}.\tilde F,W)} = \delta_{-2}.
$$
Therefore, using equations \eqref{eq:limit-height}, \eqref{eq:mu-def-2} and \eqref{eq:JDG-height},
equation \eqref{eq:step-n} becomes
$$
          \lim_{\text{\rm Im}(z)\to\infty}\, h(F(z),W) + \mu\log|s| = ht(N,F,W).
$$    
\end{proof}

\subsection*{Smooth Divisor.} Consider now the case where $\mathcal V$ is a biextension variation over
the complement of $s_1 = 0$ in a polydisk $\Delta^r$ with coordinates $(s_1,\dots,s_r)$. Let 
\begin{equation}
        F(z_1,s_2,\dots,s_r) = e^{z_1 N_1}e^{\Gamma(s)}.F_{\infty}    \label{eq:smooth-1} 
\end{equation}
denote the local normal form of the period map of $\mathcal V$ after lifting to $U\times\Delta^{r-1}$
where $U$ is the upper half-plane with coordinate $z_1$ and $s_1 = e^{2\pi i z_1}$.   To simplify notation, 
we shall write $N$ in place of $N_1$ and $z = x + i y$ in place of $z_1$ where convenient. 

\par To continue, we recall~\cite{HP} that 
$$
          [N_1,\Gamma(0,s_2,\dots,s_r)]=0
$$
Accordingly, we let
$$
         \Gamma_1 = \Gamma(0,s_2,\dots,s_r),\qquad e^{\Gamma(s)} = e^{\Gamma^1}e^{\Gamma_1}
$$
and note that $\Gamma^1(0,s_2,\dots,s_r) = 0$ by construction, and hence $s_1|\Gamma^1$.

\par Let $(\tilde F_{\infty},M) = (e^{-i\delta}.F_{\infty},M)$ be the Deligne splitting of the limit
mixed Hodge structure of $\mathcal V$.  Define $N = N_0 + N_{-2}$ and $H = \tilde Y - Y$ etc.~as in 
Theorem~\eqref{SL2-JDG} for the nilpotent orbit $(e^{zN}.F_{\infty},W)$.  Then, since $N_{-2}$ is
central in $\mathfrak g_{\mathbb C}$ and $[N,\Gamma_1]=0$ it follows that $[N_0,\Gamma_1]=0$.  
This forces
$$
          \Gamma_1 = \sum_{k\leq 0}\, \Gamma_{1,k},\qquad [H,\Gamma_{1,k}] = k\Gamma_{1,k}.
$$
Let $\Gamma^1 = \sum_k\, \Gamma^1_k$ where $[H,\Gamma^1_k] = k \Gamma^1_k$.

\begin{theorem}\label{sm-continuous} Let $\mathcal V$ be 
a biextension variation with unipotent monodromy defined on the complement 
of the divisor $s_1=0$ in $\Delta^r$.  Then, 
$$
          \bar h(\mathcal V) = h(\mathcal V) + \mu\log|s_1|
$$ 
extends continuously to $\Delta^r$.  
\end{theorem}
\begin{proof} By the above remarks, we can write
$$
\aligned
         F(z,s_2,\dots,s_r)   &= e^{xN}e^{iyN}e^{\Gamma(s)}e^{-iyN}e^{iyN}.F_{\infty} \\
                              &= e^{xN}e^{iyN}e^{\Gamma^1(s)}e^{\Gamma_1(s)}e^{-iyN}e^{iyN}.F_{\infty} \\
							  &= e^{xN}e^{iyN}e^{\Gamma^1(s)}e^{-iyN}e^{\Gamma_1(s)}e^{iyN}.F_{\infty} \\
                              &= e^{xN}e^{iyN_0}e^{iyN_{-2}}e^{\Gamma^1(s)}e^{-iyN_{-2}}e^{-iyN_0}
                                 e^{\Gamma_1(s)}e^{iyN}.F_{\infty} \\
                              &= e^{xN}e^{iyN_0}e^{\Gamma^1(s)}e^{-iyN_0}e^{\Gamma_1(s)}
 								 e^{iyN}.F_{\infty}.
\endaligned
$$

\par By theorem~\eqref{SL2-JDG}, we can then write
$$
\aligned
        e^{iyN}F_{\infty} &= g(y)e^{iyN}.\tilde F_{\infty} \\   
                          &= g(y)e^{iyN_{-2}}e^{iyN_0}\tilde F_{\infty} \\
                          &= e^{iyN_{-2}}g(y)e^{iyN_0}\tilde F_{\infty}  \\
                          &= e^{iyN_{-2}}g(y)y^{-H/2}e^{iN_0}\tilde F_{\infty} \\
                          &= e^{iyN_{-2}}y^{-H/2} g^{\ddag}(y)e^{iN_0}\tilde F_{\infty}
\endaligned 
$$
where $g^{\ddag}(y) = \text{\rm Ad}(y^{H/2})g(y)$ as in Lemma \eqref{g-twist}.

\par Combining the previous two paragraphs yields
$$
\aligned
        \delta_{(F(z,s_2,\dots,s_r),W)} 
           &= \delta_{(e^{xN}e^{iyN_0}e^{\Gamma^1(s)}e^{-iyN_0}e^{\Gamma_1(s)}e^{iyN}.F_{\infty},W)} \\
           &= \delta_{(e^{iyN_0}e^{\Gamma^1(s)}e^{-iyN_0}e^{\Gamma_1(s)}e^{iyN}.F_{\infty},W)} \\
           &= \delta_{(e^{iyN_0}e^{\Gamma^1(s)}e^{-iyN_0}e^{\Gamma_1(s)}
                       e^{iyN_{-2}}y^{-H/2} g^{\ddag}(y)e^{iN_0}\tilde F_{\infty},W)} \\
           &= \delta_{(e^{iyN_{-2}}e^{iyN_0}e^{\Gamma^1(s)}e^{-iyN_0}e^{\Gamma_1(s)}
                       y^{-H/2} g^{\ddag}(y)e^{iN_0}\tilde F_{\infty},W)} \\
            &= \delta_{(e^{iyN_0}e^{\Gamma^1(s)}e^{-iyN_0}e^{\Gamma_1(s)}
                       y^{-H/2} g^{\ddag}(y)e^{iN_0}\tilde F_{\infty},W)} + yN_{-2}.
\endaligned
$$

\par To complete the proof, let $\Gamma^1_{\ddag} = \text{\rm Ad}(y^{H/2})\Gamma^1$, and observe
that $\Gamma^1_{\ddag}$ is a finite sum of functions $(\log|s_1|)^{k/2}\, s_1 f$ where $f$ is a 
holomorphic function on $\Delta^r$ with values in $\mathfrak q$, and hence has a continuous 
extension to $\Delta^r$.  Likewise, 
$$
          \Gamma^{\ddag}_1 = \text{\rm Ad}(y^{H/2})\Gamma_1
$$  
is a finite sum of functions of the form $(\log|s_1|)^{-k/2}\,f$ for $k\geq 0$ and $f$ a holomorphic 
function on $\Delta^r$ with values in $\mathfrak q$ which is independent of $s_1$, and hence
also has a continuous extension to $\Delta^r$.  Finally, $g^{\ddag}(y)$
is independent of $(s_2,\dots,s_r)$ and has a convergent series expansion in powers of 
$(\log|s_1|)^{-1/2}$ by Lemma \eqref{g-twist}.

\par Combining the above, it follows that 
$$
\aligned
        \delta_{(F(z,s_2,\dots,s_r),W)} - y N_{-2}
          &= \delta_{(e^{iyN_0}e^{\Gamma^1(s)}e^{-iyN_0}e^{\Gamma_1(s)}
                       y^{-H/2} g^{\ddag}(y)e^{iN_0}\tilde F_{\infty},W)} \\
          &= \delta_{(y^{-H/2}e^{iN_0}e^{\Gamma^1_{\ddag}(s)}e^{-iN_0}
                      e^{\Gamma_1^{\ddag}(s)}g^{\ddag}(y)e^{iN_0}.\tilde F_{\infty},W)} \\
          &={\rm Ad}(y^{-H/2})\delta_{(e^{iN_0}e^{\Gamma^1_{\ddag}(s)}e^{-iN_0}
                      e^{\Gamma_1^{\ddag}(s)}g^{\ddag}(y)e^{iN_0}.\tilde F_{\infty},W)} \\
		  &= \delta_{(e^{iN_0}e^{\Gamma^1_{\ddag}(s)}e^{-iN_0}
                      e^{\Gamma_1^{\ddag}(s)}g^{\ddag}(y)e^{iN_0}.\tilde F_{\infty},W)} 
\endaligned
$$
since $\text{\rm Ad}(y^{-H/2})$ acts trivially on $W_{-2}\mathfrak g_{\mathbb C}$.  By the remarks
of the previous paragraph 
$$
       e^{iN_0}e^{\Gamma^1_{\ddag}(s)}e^{-iN_0}
                      e^{\Gamma_1^{\ddag}(s)}g^{\ddag}(y)e^{iN_0}
$$
has a continuous extension to $\Delta^r$, and hence so does $\delta_{(F(z,s_2,\dots,s_r),W)} - y N_{-2}$.
Using equations \eqref{eq:limit-height}, \eqref{eq:mu-def-2} and \eqref{eq:JDG-height}, it follows
as in the last paragraph of the proof of Theorem~\eqref{disk} that $\bar h(\mathcal V)$ has a 
continuous extension to $\Delta^r$.
\end{proof}

\begin{remark} Since $\bar h(\mathcal V)$ extends continuously to $\Delta^r$ its value at any
point $(0,s_2,\dots,s_r)$ can be computed by taking the limit along $s\to (s,s_2,\dots,s_r)$.
Since this corresponds to a 1 parameter degeneration, it follows from Theorem \eqref{disk}
that $\bar h(\mathcal V)(0,s_2,\dots,s_r)$ depends only on the limit mixed Hodge structure
of $\mathcal V$ at $(0,s_2,\dots,s_r)$.
\end{remark}

\section{Existence of Biextensions}\label{sec:mix}  In this section
we prove the following result, which shows that over $\Delta^{*r}$ a
pair of admissible normal functions can be glued together to give an
admissible biextension.  For the remainder of this section, $A$ is
either $\mathbb Z$, $\mathbb Q$ or $\mathbb R$, and 
\begin{itemize}
\item[---] $\Delta^r$ is a polydisk with coordinates $(s_1,\dots,s_r)$;
\item[---] $D$ is the divisor defined by $s_1\cdots s_k=0$;
\item[---] $\mathcal H$ is a torsion free, polarized variation of pure $A$-Hodge structure 
of weight $-1$ on $\Delta^r-D$ with quasi-unipotent monodromy;
\item[---] $\nu$ is an admissible extension of $A-{\rm VMHS}$
$$
          0 \to \mathcal H\to \mathcal V_{\nu}\to A(0)\to 0
$$
with respect to inclusion of $\Delta^r-D$ into $\Delta^r$;
\item[---] $\omega$is an admissible extension of $A-{\rm VMHS}$
$$
          0 \to A(1) \to \mathcal V_{\omega}^{\vee}\to\mathcal H\to 0 
$$
with respect to inclusion of $\Delta^r-D$ into $\Delta^r$.
\end{itemize}

\begin{theorem}\label{sec:mix-1} The set $B(\nu,\omega)$ of admissible biextension 
variations of type $(\nu,\omega)$ over $\Delta^r-D\hookrightarrow \Delta^r$ is non-empty.
\end{theorem}

\par We prove Theorem~\eqref{sec:mix-1} in the case of unipotent monodromy first,
and hence assume that all VMHS have unipotent monodromy unless otherwise noted.
By formally allowing the monodromy about $s_j=0$ to be trivial, we
can reduce the proof of Theorem~\eqref{sec:mix-1} to the case where
$D$ is the divisor $s_1\dots s_r=0$.  The outline of the proof is as
follows:  The normal functions $\nu$ and $\omega$ have local normal
forms~\eqref{eq:lnf}.  Likewise, as a consequence of the ${\rm SL}_2$-orbit
theorem if $F:\mathfrak h\to\mathcal M$ is the period 
map of an admissible nilpotent orbit then 
$$
        \lim_{{\rm Im}(z)\to\infty}\, Y_{(F(z),W)}
$$
exists as a grading of $W$.  Moreover, this grading will preserve the
$I^{p,q}$'s of the ${\rm sl}_2$-splitting of the limit mixed Hodge
structure of $(F(z),W)$.  This produces a trigrading of the underlying
vector space.  For $\nu$ and $\omega$, we must then glue these two
trigradings together, and use the resulting new vector space as the
model for the biextension variation.  It remains then to add an integral 
structure, local monodromy logarithms et. cetera.

\par A sketch of the proof of Theorem~\eqref{sec:mix-1} in the quasi-unipotent case is a 
follows:  A quasi-unipotent VMHS can be viewed as a VMHS on the pullback to the unipotent
setting which has the extra property that when the $i$'th coordinate is multiplied by a
suitable root of unity, the Hodge filtration changes by the appropriate semisimple factor.
Gluing $\nu$ and $\omega$ together in the quasi-unipotent case therefore reduces
to the unipotent case, together with the fact that the semisimple parts of the local
monodromy are morphisms of Hodge structure in the limit, and are therefore compatible
with the gluing procedure.

\subsection*{Trigraded Vector Spaces.}  The {\it weight filtration} of a trigraded vector space
\begin{equation}
         L = \bigoplus_{p,q,i} L^{p,q}_i              \label{eq:tri-def}
\end{equation}
is the increasing filtration $W_k(L) = \bigoplus_{i\leq k,p,q}\, L^{p,q}_i$.
The subspaces $L_i =\bigoplus_{p,q}\, L^{p,q}_i$ define a grading of $W_{\bullet}(L)$.
Likewise, we set $L^{p,q} = \oplus_i\, L^{p,q}_i$ and call
$$
          L = \bigoplus_{p,q}\, L^{p,q}
$$
is the associated {\it bigrading} of $L$.  We define the {\it relative weight
filtration} of $L$ to be the increasing filtration $M_k(L) = \oplus_{p+q\leq k,i}\, L^{p,q}_i$.
If $L$ is a complex vector space with real form $L_{\mathbb R}$ then a trigrading of $L$ is said 
to split over $\mathbb R$ if $\bar L^{p,q}_i = L^{q,p}_i$.  In this case, the decreasing 
filtration $F^p(L) = \oplus_{a\geq p,b}\, L^{a,b}$
pairs with $M$ to define an $\mathbb R$-split mixed Hodge structure $(F,M)$ which induces
pure Hodge structures on $Gr^W$ and has Deligne bigrading $I^{p,q} = L^{p,q}$.

\par If $U$ and $V$ are vector spaces equipped with linear maps $f:U\to Z$ and $g:V\to Z$
then the fiber product
$$
        U\times_Z V = \{ (u,v)\in U\times V \mid f(u) = g(v)\}
$$
is called the {\it gluing} of $U$ and $V$ along $Z$.  We let $\pi_U:U\times_Z V\to U$ and 
$\pi_V:U\times_Z V\to V$ denote the natural projection maps.

\par Let $U$ and $V$ be trigraded vector spaces.  Then $U$ {\it abuts} $V$ if there
exists an index $\ell$ such that  
\begin{itemize}
\item[(a)] $U_i=0$ for $i>\ell$ and $V_i=0$ for $i<\ell$;
\item[(b)] There exists an isomorphism $Gr^W_{\ell}(U)\cong Gr^W_{\ell}(V)$ for which the composite map
$$
        \sigma:U_{\ell}\to Gr^W_{\ell}(U)\cong Gr^W_{\ell}(V)\to V_{\ell}
$$
induces an isomorphism $U^{p,q}_{\ell}\to V^{p,q}_{\ell}$ for each bi-index $(p,q)$.
\end{itemize}
 
\par If $U$ abuts $V$ (at $\ell$) we define $Gr^W_{\ell}\subset Gr^W_{\ell}(U)\times Gr^W_{\ell}(V)$ to 
be the subspace consisting of points $([u],[v])$ which are identified under the isomorphism $(b)$.  We then 
define
$$
        U\star V = U\times_{Gr^W_{\ell}} V
$$
relative to the linear maps $f:U\to U_{\ell}\to Gr^W_{\ell}$ and $g:V\to V_{\ell}\to Gr^W_{\ell}$.

\par Let $\iota_U$ from $U\to U\star V$ denote the linear map which sends $u\in U_i$ to $(u,0)\in U\star V$
for $i<\ell$ and maps $u\in U_{\ell}$ to $(u,\sigma(u))\in U\star V$. Let $\iota_V$ from $V\to U\star V$
be the map that send $v\in V_i$ to $(0,v)\in U\star V$ for $i>\ell$ and maps
$v\in V_{\ell}$ to $(\sigma^{-1}(v),v)\in U\star V$.  Accordingly, we can introduce a trigrading 
on $U\star V$ by setting
\begin{equation}
(U\star V)^{p,q}_i = \left\{\begin{aligned} \iota_U(U^{p,q}_i),\qquad &&i\leq\ell \\
                                            \iota_V(V^{p,q}_i),\qquad &&i>\ell
                              \end{aligned}\right.\label{eq:tri-bi}
\end{equation}
If the trigradings of $U$ and $V$ are split over $\mathbb R$ and the isomorphism 
$Gr^W_{\ell}(U)\to Gr^W_{\ell}(V)$ intertwines complex conjugation then the
trigrading \eqref{eq:tri-bi} is split over $\mathbb R$.

\subsection*{Endomorphisms.} Let $U$ and $V$ be trigraded vector spaces which abut (at $\ell$).  
Let $\alpha_U$ and $\alpha_V$ be endomorphisms of $U$ and $V$ respectively which preserve the 
corresponding weight filtrations $W(U)$ and $W(V)$.  Assume that $\alpha_U$ and $\alpha_V$ induce the same 
map on $Gr^W_{\ell}$.  For $\eta\in (U\star V)_i$ define
\begin{equation}
     (\alpha_U\star \alpha_V)(\eta) =  
     \left\{\begin{aligned} \iota_U\circ\alpha_U\circ\pi_U(\eta),\qquad &&i\leq\ell \\
                            \iota_V\circ\alpha_V\circ\pi_V(\eta),\qquad &&i>\ell.
                              \end{aligned}\right.    \label{eq:endo-glue}
\end{equation}
Then, $\alpha = \alpha_U\star\alpha_V$ is an endomorphism of $U\star V$ which preserves 
$W(U\star V)$ such that 
$$
      \pi_U\circ\alpha\circ\iota_U = \alpha_U,\qquad
      \pi_V\circ\alpha\circ\iota_V = \alpha_V.
$$.

\begin{lemma} The construction \eqref{eq:endo-glue} is compatible with the bigrading \eqref{eq:tri-bi}, i.e.
if $\alpha_U(U^{a,b})\subseteq U^{p+a,q+b}$ and $\alpha_V(V^{a,b})\subseteq V^{p+a,q+b}$
then
\begin{equation}
     \alpha(U\star V)^{a,b} \subseteq (U\star V)^{p+a,q+b}
     \label{eq:ipq}
\end{equation}
\end{lemma}

\begin{proof} This follows from that fact that $\pi_U\circ\iota_U$ is the identity on $U$,
$\pi_V\circ\iota_V$ is the identity on $V$ and both the maps $\pi_U$, $\pi_V$, $\iota_U$,
$\iota_V$ all preserve the relevant bigradings by construction.
\end{proof}

\subsection*{Limiting Splitting.} Let $\psi:\Delta^*\to\Gamma\backslash\mathcal M$ be the period map of 
an admissible extension of $\mathbb R(0)$ by a variation of pure, polarized $\mathbb R$-Hodge structure 
of weight $-1$. Let $V$ denote the underlying vector space of $\psi$,
let $W$ denote the weight filtration, and assume the monodromy is unipotent 
and given by $T=e^N$. Let 
$$
         F(z):\frak h\to\mathcal M,\qquad F(z+1) = T F(z)
$$
be a lifting of $\psi$ to a map from the upper half-plane $\frak h$ into $\mathcal M$.  Then, by~\cite[Theorem 3.9]{bpannals}, it follows that
\begin{equation}
       Y^{\ddag} = \lim_{\text{Im}(z)\to\infty}\, Y_{(F(z),W)}    \label{eq:limit-split}
\end{equation}
exists, where the limit is taken with $\text{Re}(z)$ restricted to a finite interval.  
More precisely,
$$
       Y^{\ddag} = \rm{Ad}(e^{-iN})Y_{(e^{iN}\hat F,W)}
$$
where $(\hat F,M) = (e^{-\xi}F,M)$ is ${\rm sl}_2$-splitting of the limit mixed Hodge structure $(F,M)$
of $F:\mathfrak h\to\mathcal M$.

\begin{corollary}\label{limit-trigrading} The subspaces
$$
      L^{p,q}_k = \{\, v\in I^{p,q}_{(\hat F,M)} \mid Y^{\ddag}(v) = kv\,\}
$$
form a trigrading of $V$ which is split over $\mathbb R$ such that $W(L) = W$, $F(L) = \hat F$ and 
$M(L) = M$ as in \eqref{eq:tri-def}.
\end{corollary}
       
\par For future use, we record that in this setting, the element $\xi$ appearing~\eqref{eq:sl2-split}
in the  ${\rm sl}_2$-splitting of $(F,M)$ also preserves the weight filtration W~\cite{BP-Comp}. 
The next result is only used in the quasi-unipotent case:

\begin{prop}\label{DS-morphism} If $\g$ is a morphism of $(F,M)$ which preserves $W$ such that 
$[\g,N] = 0$ then $[\g,Y^{\ddag}] = 0$.
\end{prop}
\begin{proof} Since $\g$ is a morphism of $(F,M)$ it follows that $\g$ commutes with the Hodge 
components $\{\delta^{p,q}\}$ of the Deligne splitting of $(F,M)$.  As $\xi$ is given by universal 
Lie polynomials in the Hodge components of $\delta$, it follows that $\g$ commutes with $\xi$.  
Therefore, $\g$ is a morphism of $(\hat F,M) = (e^{-\xi}F,M)$.  By~\cite{bpannals},
$$
          Y^{\ddag} = {\rm Ad}(e^{-iN})Y_{(e^{iN}\hat F,W)}.
$$
As $\gamma$ commutes with $N$ and preserves $W$, it follows that $\gamma$ is of type 
$(0,0)$ with respect to $(e^{iN}\hat F,W)$ as well.  Thus, $[\gamma,Y^{\ddag}]=0$.
\end{proof}

\subsection*{Limit Mixed Hodge Structure.}
We now return to the notation from the beginning of this section.
In particular, $\mathcal{V}_{\nu}$ and $\mathcal{V}_{\omega}$ are as in the
paragraph before Theorem~\ref{sec:mix-1}.

Let $\Delta^*\subset\Delta^{*r}$ be the punctured disk defined by the
equation $s_1 = \cdots = s_r$, and fix a point $s_0\in\Delta^*$. Let
$V$ be the fiber of $\mathcal V_{\nu}$ over $s_0$.  Then,
applying Corollary~\ref{limit-trigrading} to the 
restriction
of $\mathcal V_{\nu}$ to $\Delta^*\subset\Delta^{*r}$ produces a
trigrading $V^{p,q}_k$ of $V$ which is split over $\mathbb R$.
Likewise, let $U$ be the fiber of
$\mathcal V_{\omega}^{\vee}$ over $s_0$.  Then, using duality,  restriction of
$\mathcal V_{\omega}^{\vee}$ to $\Delta^*\subset\Delta^{*r}$ produces
a trigrading $U^{p,q}_k$ of $U$ which is also split over $\mathbb R$.

\par Since $U$ and $V$ abut at $\ell=-1$, we can form the trigraded vector space
$$      
           B = U\star V
$$
with associated trigrading which is split over $\mathbb R$.  The associated bigrading 
determines an $\mathbb R$-split mixed Hodge structure
$(\hat F_{\infty}, M)=(\hat F_{\infty}(B),M(B))$
as in~\eqref{eq:tri-def}.

\par Let $(F_{\nu},M^{\nu})$ be the limit mixed Hodge structure of $\mathcal V_{\nu}$ and 
$(F_{\omega},M^{\omega})$ be the limit mixed Hodge structure $\mathcal V_{\omega}^{\vee}$.  
Let $\xi_{\nu}\in gl(V)$ and $\xi_{\omega}\in gl(U)$ be the corresponding endomorphisms \eqref{eq:sl2-split}
of $V$ and $U$.
Recall that $\xi_{\nu}$ and $\xi_{\omega}$ preserve the weight filtrations $W(V)$ and $W(U)$ 
respectively.  Moreover, they induce the same action on $Gr^W_{-1}$ since $\mathcal V_{\nu}$ and 
$\mathcal V_{\omega}^{\vee}$ induce the same limit mixed Hodge structure on $Gr^W_{-1}$.  Define
\begin{equation}
         F_{\infty} = e^{\xi}\hat F_{\infty}                     \label{eq:limit-MHS}
\end{equation}
where $\xi = \xi_U\star\xi_V$.   Since $\xi_{\nu}$ and $\xi_{\omega}$ act trivially on $Gr^M(V)$
and $Gr^M(U)$ it follows that $\xi$ acts trivially on $Gr^M$ by \eqref{eq:ipq}. 

\begin{remark} The mixed Hodge structure $(\hat F_{\infty},M)$ projects to $(\hat F_{\nu},M^{\nu})$
and $(\hat F_{\omega},M^{\omega})$ on $V$ and $U$ respectively.   The monodromy logarithms
$N_j^V$ and $N_j^U$ of $\mathcal V_{\nu}$ and $\mathcal V^{\vee}_{\omega}$ are $(-1,-1)$-morphisms
of $(\hat F_{\nu},M^{\nu})$ and $(\hat F_{\omega},M^{\omega})$.  For future use, we also note
that $W_{-2}\End(B)$ consists of $(-1,-1)$-morphisms of $(\hat F_{\infty},M)$ and 
$(F_{\infty},M)$.  This follows from the fact that $B_0=B_0^{0,0}$ while
$B_{-2}=B_{-2}^{-1,-1}$. 
\end{remark}

\subsection*{Local System.} The monodromy logarithms $N_j^V$ and $N_j^U$ induce the same action 
on $Gr^W_{-1}$ and hence glue together to define an endomorphism $\tilde N_j$ of $B_{\mathbb R}$.
By the previous remark, $\tilde N_j$ is a $(-1,-1)$-morphism of $(\hat F_{\infty},M)$.  To show
that the endomorphisms $\{\tilde N_j\}$ define a local system over $\Delta^{*r}$ with fiber 
$B_{\mathbb R}$, it is sufficient to show that $[\tilde N_j,\tilde N_k]=0$ for all $j$ and $k$.

\par Let $\{V^{p,q}_i\}$ be the trigrading of $V$ of Corollary \eqref{limit-trigrading}, and let 
$v_0\in V_0$ be a lifting of $1\in Gr^W_0$.  Set $b_0= \iota_V(v_0)$. Then, since $v_0$ is type
$(0,0)$ with respect to the bigrading of $V$, the same is true of $b_0$ with respect to the bigrading
of $B$.  Moreover, tracing through the above definitions, one sees that 
$$
      \pi_V[\tilde N_j,\tilde N_k]b_0 = [N_j^V,N_k^V]v_0 = 0
$$
Therefore, $[\tilde N_j,\tilde N_k]b_0$ is an element of $W_{-2}(B)$ which is of type $(-2,-2)$
with respect to the bigrading of $B$.  But $W_{-2}(B)$ is pure of type $(-1,-1)$, and hence
$
        [\tilde N_j,\tilde N_k]b_0 = 0
$.
Likewise, 
$
        \pi_U[\tilde N_j,\tilde N_k]\iota_U(u) = [N_j^U,N_k^U]u = 0
$.
Consequently, $[\tilde N_j,\tilde N_k] =0$.

\subsection*{$A$-Structure.} Pick an element $v_A\in V_A$ which projects to 
$1\in Gr^W_0(V)$ and let $b_A = \iota_V(v_A)$.
Define
$$
        B_A = A b_A\oplus\iota_U(U_A)
$$
Then, $B_A\otimes\mathbb R = B_{\mathbb R}$ since $v_A = v_0 + v_{-1}$
with $v_0$, $v_{-1}$ real and $v_0$ as above in $V^{0,0}_0$.

\par Let $u_{-2}$ be the generator of $W_{-2}(U_A)\cong A(1)$ corresponding
to $1^{\vee}$ (with $1^{\vee}$ a generator of $A(1)$ as in
Definition~\ref{1v}).  Set 
$b_{-2} = \iota_U(u_{-2})$.  Note that $b_{-2}\in B^{-1,-1}_{-2}$.  Let $\eta$ be
the endomorphism of $B_{\mathbb R}$ which annihilates $W_{-1}(B)$ and maps
$b_A$ to $b_{-2}$ (or, equivalently, $b_0$ to $b_{-2}$).  We are going to
set 
\begin{equation}
       N_j = \tilde N_j + c_j\eta					\label{eq:new-monodromy}
\end{equation}
for some scalar $c_j$ to be determined below. 

\par Since $W_{-2}(B)$ has rank $1$ and $\tilde N_j$ is nilpotent and preserves $W(B)$, 
it follows that $\tilde N_j$ acts trivially on $W_{-2}(B)$.  Consequently, $[\tilde N_j,\eta]=0$ 
and hence $[N_j,N_k]=0$.  Likewise, $[N_j,\eta] = 0$ and hence
$
         T_j = e^{N_j} = e^{\tilde N_j+c_j\eta} = (\tilde T_j)(1+c_j\eta)  
$.
Therefore, 
$$
	T_j(b_A) = \tilde T_j(b_A) + c_j u_{-2}
                       = b_A + (\tilde T_j-1)b_A + c_j u_{-2}
$$
In particular, because $(T_j^V-1)v$ maps to an integral class in $Gr^W_{-1}$,
it follows that for suitable choice of real scalar $c_j$, $T_j(b)\in B_A$.
Likewise, $T_j$ acts as $T_j^U$ on $U$, and hence $T_j$ preserves $B_A$.

\subsection*{Hodge filtration.}
By rescaling $s$ if necessary, we can assume that 
$\mathcal{V}_{\nu}$ and $\mathcal V_{\omega}^{\vee}$
have local normal forms 
$\Gamma^V$ and $\Gamma^U$
on $\Delta^{r}$ 
as in \eqref{eq:lnf}.  Then, $\Gamma^V$ preserves 
$W(V)$, $\Gamma^U$ preserves $W(U)$ and the induced actions on $Gr^W_{-1}$ agree since the induced 
VHS on $Gr^W_{-1}$ coincide.  Therefore, we can glue $\Gamma^V$ to $\Gamma^U$ to define a 
function $\Gamma$ on $\Delta^r$ which preserves $W(U\star V)$ and acts by infinitesimal isometries on $\Gr^W B$.  
Define
\begin{equation}
        F(z) = e^{\sum_j\, z_j N_j}e^{\Gamma(s)}F_{\infty}             \label{eq:glue-lnf}
\end{equation}
where $F_{\infty}$ is the filtration \eqref{eq:limit-MHS}, and $N_j$ is the endomorphism 
\eqref{eq:new-monodromy}.  Then, $F(z)$ induces the corresponding Hodge filtrations of 
$\mathcal V_{\nu}$ on $Gr^W_j(V)$ for $j=0$, $j=-1$, and $\mathcal V_{\omega}^{\vee}$ for 
$j=-1$, $-2$.  Clearly, $F(z)$ is holomorphic and descends to a period map over $\Delta^{*r}$
with monodromy $\{T_j\}$.

\subsection*{Horizontality.}  We need to show that $\frac{\partial}{\partial z_j}F^p(z)\subseteq F^{p-1}(z)$.
By manipulation of \eqref{eq:glue-lnf}, this is equivalent to 
$$
       \left(e^{-\text{\rm ad}\Gamma}N_j + 2\pi i s_j e^{-\Gamma}\frac{\partial}{\partial s_j}e^{\Gamma}\right)
		(F_{\infty}^p)\subseteq F^{p-1}_{\infty}
$$
By construction, this holds modulo $W_{-2}\End(B)$, which is sufficient to prove horizontality for $p>0$.
For the case $p=0$, observe that $W_{-2}\End(B)$ consists of $(-1,-1)$-morphisms
of $(F_{\infty},M)$.  For $p<0$, horizontality follows from horizontality modulo $W_{-2}B$ and
the case $p=0$.

\subsection*{Admissibility.} To show that \eqref{eq:glue-lnf} defines an admissible normal function,
it remains to show that (i) the Hodge bundles extend holomorphically over $\Delta^r$ and induces the
limit filtration of Schmid on each $Gr^W$, and (ii) the required relative weight filtrations exist.

\par The first condition follows from that fact that 
$
      \exp(-\sum_j\, z_j N_j)F(z) = e^{\Gamma(s)}F_{\infty}
$
which obviously extends to a holomorphic filtration on $\Delta^r$, and $F_{\infty}$ induces 
Schmid's filtration on $Gr^W$ by construction.

\par To prove the existence of the required relative weight filtrations, we recall the following
result from \cite[Theorem 2.20]{SZ}: Let ${}_{(i)}M$ denote the
relative weight filtration of $N_{W_i}$ and $W$.  Suppose that ${}_{(k-1)}M$ exists.  Then,
${}_{(k)}M$ exists if and only if for all $j>0$,
$$
        (N^j W_j)\cap W_{k-1}\subset N^j W_{k-1} + {}_{(k-1)}M_{k-j-1}
$$

\par Set $N=N_j$ and $W=W(B)$.  Then, the relative weight filtration of $N$ restricted to $W_{-1}$ 
exists and equals the image of $M(N_j^U,W(U))$ under $\iota_U$.  Therefore, we need only check 
the above for $k=0$.  Since $N = \tilde N_j + c_j\eta$, 
$$
		(N^j W_j)\cap W_{-1} = N^j W_{-1}+\text{\rm Im}(c_j\eta) 
		\subset N^j W_{-1} + {}_{(-1)}M_{-2}
$$
because $W_{-2}\subset {}_{(-1)}M_{-2}$ since $W_{-2}$ is pure of type $(-1,-1)$ for the limit 
mixed Hodge structure.

\par This completes the proof of Theorem~\ref{sec:mix-1}.

\subsection*{Quasi-Unipotent Monodromy}  

\par In this section, we extend Theorem~\ref{sec:mix-1} to the case of admissible normal functions
with quasi-unipotent monodromy, and give an analog of the local normal form~\eqref{eq:lnf} for 
admissible variations with quasi-unipotent monodromy. 

\par Let $T$ be an automorphism of a finite dimensional complex vector space.
Then, by the Jordan decomposition theorem, $T = T_s + T_n$ where $T_s$
is semisimple, $T_n$ is nilpotent and $[T_s,T_n] = 0$.  Moreover, there
exists polynomials $p$ and $q$ without constant term such 
that $T_s = p(T)$ and $T_n = q(T)$.  In particular, since $T$ is an 
automorphism, $T_s$ is invertible.  Let $T_u = 1 + T_s^{-1}T_n$.
Then, $T_u$ is unipotent and $T = T_s T_u = T_u T_s$ is called the
multiplicative Jordan decomposition of $T$.

\begin{remark} If $T$ is a automorphism of a finite dimensional rational
vector space, and $T_s$ has finite order $m$, then both $T_s$ and $T_u$
are rational.  To see this, note that $T_u = e^{\alpha}$ and hence
$T^m = T_u^m = e^{m\alpha}$ is rational.  Therefore, $\alpha$ is rational,
and hence so is $T_u$ and $T_s = T T_u^{-1}$.
\end{remark}

\begin{prop} Let $T$ and $T'$ be commuting automorphisms of a finite dimensional
complex vector space, and let $T=T_s + T_n$ and $T' = T'_s + T'_n$ be the Jordan
decomposition of $T$ and $T'$.   Then, $\{T_s,T_n,T'_s,T'_n\}$ is a set of 
mutually commuting endomorphisms.
\end{prop}
\begin{proof} Each of the elements listed is a polynomial in $T$ or $T'$.
\end{proof}

\begin{corollary} Let $T$ and $T'$ be commuting automorphisms of a finite dimensional
complex vector space, and let $T=T_s T_u$ and $T' = T'_s T'_u$ be the multiplicative
Jordan decompositions of $T$ and $T'$.  Then, $\{T_s,T_u,T'_s,T'_u\}$ is a set of 
mutually commuting automorphisms.
\end{corollary}

\par Let $(\Delta^r,s)$ be a polydisk with holomorphic coordinates $(s_1,\dots,s_r)$
and $\mathcal V\to\Delta^{*r}$ be a variation of mixed Hodge structure
defined on the complement of the divisor $s_1\cdots s_r = 0$.  Assume that 
$\mathcal V$ has quasi-unipotent monodromy $\g_i$ about $s_i =0$, and let 
$\g_i = \g_{i,s}\g_{i,u}$ be the multiplicative Jordan decomposition of 
$\g_i$.  Then, by the previous Corollary, $\{\g_{j,s},\g_{j,u},\g_{k,s},\g_{k,u}\}$
is a set of mutually commuting automorphisms of the reference fiber of $\mathcal V$
for any pair of indices $j$ and $k$.  Let $m_i>0$ denote the order of $\g_{i,s}$
and $N_i = \log\g_{i,u}$.

\par Assume that $\mathcal V\to(\Delta^{*r},s)$ is admissible.  Let $(\Delta^{*r},t)$ 
be a polydisk with coordinates $(t_1,\dots,t_r)$ and 
$
       f:(\Delta^{*r},t)\to (\Delta^{*r},s)
$
denote the covering map $t_j = s_j^{m_j}$.  Then, $f^*(\mathcal V)$ is an
admissible variation of mixed Hodge structure over $(\Delta^{*r},t)$ with
unipotent monodromy $e^{m_j N_j}$ about $t_j=0$.  Let
$$
          \tilde F:({\mathfrak h}^r,w)\to\mathcal M
$$
denote the lifting of the period map of $f^*(\mathcal V)$ to the product
of upper half-planes ${\mathfrak h}^r\subset\mathbb C^r$ with Cartesian coordinates 
$(w_1,\dots,w_r)$ relative to the covering map $t_j = e^{2\pi i w_j}$.

\par For $w\in\mathbb C^r$ let $N(w) = \sum_k\, w_k m_k N_k$, and 
\begin{equation}
     \tilde\psi(w) = e^{-N(w)}\tilde F(w).        \label{eq:p-lnf-1}
\end{equation}
Let $\{e_1,\dots,e_r\}$ denote the standard Cartesian basis of $\mathbb C^r$.  Then,
$\tilde\psi(w+e_j) = \tilde\psi(w)$, and hence $\tilde\psi$ descends to a map
$$
     \psi:\Delta^{*r}\to\check{\mathcal M}.
$$
By admissibility, $\psi$ extends to a holomorphic map $\Delta^r\to\check{\mathcal M}$,
such that 
\begin{equation}
          (m_1N_1,\dots,m_r N_r;F_{\infty},W)                   \label{eq:p-nilp}
\end{equation}
is an admissible nilpotent orbit, where $F_{\infty} = \psi(0)$.  

Let $\mathfrak g_{\mathbb C} = \oplus_{p,q<0}\,\mathfrak g^{p,q}$
relative to the limit mixed Hodge structure $(F_{\infty},M)$ of \eqref{eq:p-nilp}.
Then, 
$$      
           \mathfrak q = \oplus_{a<0,b}\mathfrak g^{a,b}
$$
is a vector space complement to the stabilizer of $F_{\infty}$ in $\mathfrak g_{\mathbb C}$.
Therefore, there exists a unique holomorphic function 
$$
         \Gamma:\Delta^r\to\mathfrak q
$$
such that $e^{\Gamma(t)}F_{\infty} = \psi(t)$, after shrinking $\Delta^r$ as necessary.
This gives the local normal form
\begin{equation}
        \tilde F(w) = e^{N(w)}e^{\Gamma(t)}F_{\infty}.         \label{eq:q-lnf}
\end{equation}
 
\par The next result asserts that each $\g_{j,s}$ is a morphism of the limit mixed Hodge
structure.  In the geometric case this is due to J. Steenbrink~\cite{Steenbrink}.

\begin{prop}\label{ss-morphism} Each $\g_{j,s}$ is morphism of $(F_{\infty},M)$.  Moreover,
\begin{equation}
               \Gamma\circ\rho_j(t) = \text{\rm Ad}(\g_{j,s})\Gamma(t)   \label{eq:q-lnf-4}
\end{equation}
where $\rho_j(t_1,\dots,t_r) = (t_1,\dots,e^{2\pi i/m_j}t_j,\dots,t_r)$.
\end{prop}
\begin{proof} To prove that $\g_{j,s}$ is a morphism of the limit mixed Hodge structure we begin
with the observation that
\begin{equation}
          \psi\circ\rho_j(t) = \g_{j,s}\psi(t)                       \label{eq:q-lnf-2}
\end{equation}
To see this we note that we have a commutative diagram
$$
\begin{CD}
      ({\mathfrak h}^r,w) @>{z_k = m_k w_k}>> ({\mathfrak h}^r,z) @>{F(z)} >> \mathcal M \\
      @V{t_k = e^{2\pi i w_k}}VV       @VV{s_k = e^{2\pi i z_k}}V  @VVV \\
      (\Delta^{*r},t) @>{s_k = t_j^{m_k}}>> (\Delta^{*r},s) @>{\varphi}>>\Gamma\backslash\mathcal M
\end{CD}
$$
where $F:({\mathfrak h}^r,z)\to\mathcal M$ is the local lifting of the period map of $\mathcal V\to(\Delta^{*r},s)$
relative to the covering map $s_k = e^{2\pi iz_k}$ where $(z_1,\dots,z_r)$ are Cartesian coordinates
on the product of upper half-planes ${\mathfrak h}^r\subset\mathbb C^r$.  In particular, since $F(z+e_j) = \g_j F(z)$
it follows that $\tilde F(w + (1/m_j)e_j) = \g_j\tilde F(w)$.  Therefore,
$$
      \psi\circ\rho_j(t) = e^{-N(w+(1/m_j)e_j)}\tilde F(w+(1/m_j)e_j) 
                         = e^{-N(w+(1/m_j)e_j)}\g_j\tilde F(w) = \g_{j,s}\psi(t).
$$ 
Taking the limit as $t\to 0$ yields $F_{\infty} = \g_{j,s}F_{\infty}$.

\par By the remark at the beginning of this section, in the case where $A=\mathbb Z$ or $A=\mathbb Q$,
$\gamma_{j,s}$ is rational since $\g_{j,s}$ is finite order and $\g_j$ is rational.  To prove that 
$\g_{j,s}$ is a morphism of $(F,M)$ it remains to check that $\g_{j,s}$ preserves $M$.  This
follows from the defining properties of the relative weight filtration and the fact that:
\begin{itemize}
\item[---] Since $\g_j$ preserves $W$ so does $\g_{j,s} = p(\g_j)$;
\item[---] $\g_{j,s}$ commutes with $N_1,\dots,N_r$ and $M = M(N_1 + \cdots + N_r,W)$.
\end{itemize}
\par To verify \eqref{eq:q-lnf-4}, note that by the derivation of \eqref{eq:q-lnf} and equation 
\eqref{eq:q-lnf-2}, we have
$$
      \psi\circ\rho_j(t) = e^{\Gamma\circ\rho_j(t)}F_{\infty} = \g_{j,s}\psi(t) 
       = \g_{j,s}e^{\Gamma(t)}F_{\infty}.
$$
Since $\g_{j,s}$ preserves $F_{\infty}$ by the previous paragraphs, it follows that 
$$
        e^{\Gamma\circ\rho_j(t)}F_{\infty} = \g_{j,s}e^{\Gamma(t)}\g_{j,s}^{-1}F_{\infty}.
$$
In particular, since $\g_{j,s}$ is a morphism of $(F_{\infty},M)$ the adjoint action of 
$\g_{j,s}$ preserves $\mathfrak q$.  Because $\mathfrak q$ is a vector space complement
to the stabilizer of $F_{\infty}$ in $\mathfrak g_{\mathbb C}$, it then follows that 
$\Gamma\circ\rho_j(t) = \Ad(\g_{j,s})\Gamma(t)$.
\end{proof}

\par To continue, we note that by commutative diagram of Proposition~\eqref{ss-morphism}, 
we have $\tilde F(w_1,\dots,w_r) = F(m_1 w_1,\dots,m_r w_r)$.  Setting $w_j = z_j/m_j$ it
then follows from \eqref{eq:q-lnf} that 
\begin{equation}
      F(z_1,\dots,z_r) = \tilde F(z_1/m_1,\dots,z_r/m_r)
                       = e^{\sum_k\, z_k N_k}e^{\Gamma(v_1,\dots,v_r)}F_{\infty} \label{eq:q-lnf-3}
\end{equation}
where $v_j = e^{2\pi iz_j/m_j}$.

\begin{theorem}\label{q-lnf} 
\begin{itemize} Let
\item[(1)] $\theta(z_1,\dots,z_r) = (e^{\sum_k\, z_k N_k}F_{\infty},W)$ be an admissible
nilpotent orbit with limit mixed Hodge structure $(F_{\infty},M)$ and values in the a classifying 
space $\mathcal M$ of graded-polarized mixed Hodge structure for 
$\text{\rm Im}(z_1),\dots,\text{\rm Im}(z_r)\gg 0$;
\item[(2)]   $\{\g_{1,s},\dots,\g_{r,s}\}$ be a set of semisimple automorphisms of
$(F_{\infty},M)$ which act by finite order isometries on $Gr^W$.  Let $m_j$ be the order of $\g_{j,s}$ and 
$\g_{j,u} = e^{N_j}$.  Assume that
$
         \{\g_{1,s},\g_{1,u},\dots,\g_{r,s},\g_{r,u}\}
$
is a set of mutually commuting automorphisms; 
\item[(3)] $\Gamma:\Delta^r\to\mathfrak q$ be a holomorphic function which vanishes at the origin 
and satisfies \eqref{eq:q-lnf-4}.
\end{itemize}
Suppose that 
\begin{equation}
       \tilde F(w_1,\dots,w_r) = e^{\sum_k\, w_k m_k N_k}e^{\Gamma(t)}F_{\infty}   \label{eq:q-lnf-5}
\end{equation}
satisfies Griffiths infinitesimal period relation on the upper half-plane $({\mathfrak h}^r,w)$ where 
$t_j = e^{2\pi i w_j}$. Define $F(z)$ from $\tilde F(w)$ via equation~\eqref{eq:q-lnf-3}.  Then, 
\begin{itemize}
\item[(i)] $F(z_1,\dots,z_j+1,\dots,z_r) = \g_i F(z_1,\dots,z_r)$ where $\g_i = \g_{i,s}\g_{i,u}$;
\item[(ii)] There exists a positive real number $A$ such that 
$$
       F(z_1,\dots,z_r) = e^{\sum_j\, z_j N_j}e^{\Gamma(v_1,\dots,v_r)}F_{\infty}
$$
descends to an admissible period map on the set of points in $(\Delta^{*r},s)$ 
such that $|s_j| = |e^{2\pi i z_j}|<A$.
\end{itemize}
\end{theorem}
\begin{proof}  Since $v_j = e^{2\pi iz_j/m_j}$ it follows that changing $z_j$ to $z_j+1$
changes $v_j$ to $e^{2\pi i/m_j}v_j$.  Therefore, by equation \eqref{eq:q-lnf-4},
$$
\aligned
      F(z_1,\dots,z_j+1,\dots,z_r) 
      &= e^{N_j}e^{\sum_k\, z_k N_k}\g_{j,s}e^{\Gamma(v_1,\dots,v_r)}\g_{,s}^{-1}F_{\infty} \\
      &= \g_{j,u}\g_{j,s}e^{\sum_k\, z_k N_k}e^{\Gamma(v_1,\dots,v_r)}F_{\infty} 
\endaligned
$$
using hypothesis $(2)$ and $(3)$.  This proves part $(i)$.

\par To prove part $(ii)$, we need to show that 
\begin{itemize}
\item[(a)] $F(z)$ satisfies Griffiths infinitesimal period relation;
\item[(b)] $F(z)$ takes values in $\mathcal M$ for ${\rm Im}(z_1),\dots,{\rm Im z_r}\gg 0$;
\item[(c)] Verify the existence of the requisite relative weight filtrations.
\end{itemize}
Condition $(c)$ follows from the fact that $\theta$ is an admissible nilpotent orbit.
On the other hand, since $(\Delta^{*r},t)\to(\Delta^{*r},s)$ is a covering map, it is 
sufficient to verify conditions $(a)$ and $(b)$ for the map \eqref{eq:q-lnf-5}.
Condition $(a)$ holds for $\tilde F(w)$ by assumption.  To verify condition $(b)$, 
we need only show that $\tilde F(w)$ induces pure, polarized Hodge structures on $Gr^W$
for ${\rm Im}(w_1),\dots,{\rm Im w_r}\gg 0$.  This follows immediately from the fact
that $\theta$ is an admissible nilpotent orbit, $\tilde F(w)$ is horizontal and  
Theorem $(2.8)$~\cite{Luminy}.
\end{proof}

\begin{remark} In order to produce a variation with an integral (or rational structure) via
Theorem~\eqref{q-lnf}, each $\g_{j,s}$ and $\g_{j,u}$ must be integral (or rational).
\end{remark}

\begin{theorem}\label{sec:mix-q} Theorem~\eqref{sec:mix-1} remains valid in the setting where
$\mathcal H$ has quasi-unipotent monodromy.
\end{theorem}
\begin{proof} By Theorem~\eqref{q-lnf}, it is sufficient to work on the pullback to the 
unipotent case and ensure that the resulting function $\Gamma$ satisfies \eqref{eq:q-lnf-4}.
Each of the functions $\Gamma^V$ and $\Gamma^U$ attached to the normal functions satisfy
an appropriate version of \eqref{eq:q-lnf-4} relative to the semisimple automorphisms 
$\g_{j,s}^V$ and $\g_{j,s}^U$ respectively.   It follows from Proposition~\eqref{DS-morphism}
that $\g_{j,s}^V$ and $\g_{j,s}^U$ preserve the the trigradings of $V$ and $U$.  Since
they agree on $Gr^W_{-1}$ it follows that they glue together to an automorphism
$\g_{j,s}$ of $B$ which preserves the trigrading of $B$.  Therefore, the function $\Gamma$
produced by gluing $\Gamma^U$ and $\Gamma^V$ also satisfies \eqref{eq:q-lnf-4}. 
\par In the case where $A=\mathbb Z$ or $\mathbb Q$, we add the underlying $A$ structure
as follows:  The element $\eta$ appearing \eqref{eq:new-monodromy} commutes with each $\g_{j,s}$
since $\g_{j,s}\in gl(B)^{0,0}_0$ with respect to the induced trigrading on $gl(B)$,
and $Gr(\g_{j,s})$ acts trivially on the 1-dimensional factors $Gr^W_0$ and $Gr^W_{-2}$
by hypothesis (i.e. these factors are assumed to have trivial global monodromy). 
The construction of the $A$-structure therefore follows exactly as in the unipotent case.
\end{proof}

\section{Singularities of Normal Functions}\label{sec:sing}

\par In this section, we review the notion of the singularities of normal functions,
in the setting where $\mathcal H$ is an admissible variation of mixed Hodge structure
such that $W_{-1}\mathcal H = \mathcal H$.

\subsection{Admissible Normal Functions.}\label{s1.1}
Suppose $\bar S$ is a complex manifold and $S$ is a Zariski open subset.  
Suppose
$\mathcal{H}$ is an object in $\VMHS(S)_{\bar S}^{\rm ad}$ with
$W_{-1}\mathcal{H}=\mathcal{H}$ and underlying $\mathbb Z$-local
structure $\mathcal H_{\mathbb Z}$.  Following M.~Saito, we define
$$
  \ANF(S,\mathcal{H})_{\bar S}:=
  \Ext^1_{\VMHS(S)^{\ad}_{\bar S}}(\mathbb{Z},\mathcal{H})
$$
to be the group of admissible normal functions on $S$ relative to $\bar S$
with underlying variation of mixed Hodge structure $\mathcal H$.

\subsection{Singularities of admissible normal functions.}\label{s1.2}  
Let $S$, $\bar S$ and $\mathcal H$ be as in \S\ref{s1.1}, and let $s\in \bar
S\setminus S$.  Write $j:S\to \bar S$ and $i:\{s\}\to \bar S$ for the
inclusions.  
Suppose $\nu\in\ANF(S,\mathcal{H})_{\bar S}$ is an admissible normal
function.   Let 
\begin{equation}\label{e.S1}
0\to \mathcal{H}\to \mathcal{V}\to \mathbb{Z}\to 0
\end{equation}
by the extension class corresponding to $\nu$. 
By applying the forgetful
functor  from $\VMHS(S)^{\ad}_{\bar S}$ to the category $\Loc(S)$ of 
local systems on $S$, we wind up with an extension in the category 
of local systems on $S$.   Thus we can associate to $\nu$ a class
$\cl(\nu)\in \rH^1(S,\mathcal{H})$.   Tensoring with $\mathbb{Q}$,
we get a class $\cl_{\mathbb{Q}}\in \rH^1(S,\mathcal{H}_{\mathbb{Q}})=
\rH^1(\bar S,Rj_*\mathcal{H}_{\mathbb{Q}})$.
Finally, by pull-back via $i$, we get a class 
\begin{equation}\label{singdef}
\sing_s(\nu)\in
\rH^1_s(\mathcal{H}_{\mathbb{Q}}):=\rH^1(i^*Rj_*\mathcal{H}_{\mathbb{Q}}).
\end{equation}
This class is called the \emph{singularity} of $\nu$.  We write
$\cl_s(\nu)$ for the image of $\cl(\nu)$ in
$\rH^1(i^*Rj_*\mathcal{H})$.

\subsection{Singularities.}\label{s1.2b}  Suppose $\mathbb{F}=\mathbb Q$ or $\mathbb R$.  Let  
$\mathcal{L}$ denote a $\mathbb{F}$-local system on $S$ and $d=\dim_{\mathbb C}\, S$.
  The intersection complex of $\mathcal{L}$ is 
$$
\IC(\mathcal{L}):=j_{!*}(\mathcal{L}[d])
$$
The intersection and local intersection cohomology spaces are 
defined by 
\begin{align*}
  \IH^k(\bar S,\mathcal{L})&:=\rH^{k-d}(\bar S,\IC(\mathcal{L}))\\
  \IH^k_s(\mathcal{L})&:=\rH^{k-d}(i^*\IC(\mathcal{L})).
\end{align*}
There is a morphism $\IC(\mathcal{L})\to Rj_*\mathcal{L}$ which 
induces maps $\IH^k(\bar S,\mathcal{L})\to \rH^k(S,\mathcal{L})$ and 
$\IH^k_s(\mathcal{L})\to \rH^k_s(\mathcal{L})
:=\rH^k(i^*Rj_*\mathcal{H}_{\mathbb{Q}})$.   For $k=0$, these 
maps are isomorphisms, and for $k=1$ they are injections. 

Note that, by M.~Saito's theory of mixed Hodge modules, if $\calH$ is a 
$\mathbb{Q}$ variation of mixed Hodge structure, then $\rH^p_s(\calH)$ 
and $\IH^p_s(\calH)$ both carry canonical mixed Hodge structures. 

\begin{theorem}~\cite{BFNP}\label{factor} Suppose $\mathcal{H}$ is a weight $-1$ variation of
Hodge structure on $S$.  The map 
$$
    \cl_{\mathbb{Q}}: \ANF(S,\mathcal{H})_{\bar S}\to \rH^1(S,\mathcal{H}_{\mathbb{Q}})
$$ 
factors through $\IH^1(S,\mathcal{H}_{\mathbb{Q}})$.  Similarly, the map
$$
     \sing_s:\ANF(S,\mathcal{H})_{\bar S}\to \rH^1_s(\mathcal{H}_{\mathbb{Q}})
$$ 
factors through $\IH^1_s(\mathcal{H}_{\mathbb{Q}})$.
\end{theorem}

\subsection{Duality in the Weight $-1$ case.}  
If $\mathcal{H}$ is variation of pure Hodge structure of weight $-1$
on $S$, as in the introduction, we write
$\mathcal{H}^{\vee}:=\mathcal{H}^{*}(1)$.  This is another variation
of pure Hodge structure of weight $-1$ on $S$.  There is a canonical
morphism $\mathcal{H}\to (\mathcal{H}^{\vee})^{\vee}$ which is an
isomorphism if $\mathcal{H}$ is torsion-free.

\begin{proposition}\label{p.dual}
Suppose $\mathcal{H}$ is torsion free.  Then duality gives a 
canonical isomorphism 
$$
    \ANF(S,\mathcal{H})_{\bar S}\cong
    \Ext^1_{\VMHS(S)^{\ad}_{\bar S}}(\mathcal{H}^{\vee},\mathbb{Z}(1))
$$
\end{proposition}
\begin{proof}
  We have 
$$\ANF(S,\mathcal{H})_{\bar S}=
   \Ext^1_{\VMHS(S)^{\ad}_{\bar S}}(\mathbb{Z},\mathcal{H})
   \stackrel{D}{\to}\Ext^1_{\VMHS(S)^{\ad}_{\bar S}}(\mathcal{H}^*,\mathbb{Z}) 
    =\Ext^1_{\VMHS(S)^{\ad}_{\bar S}}(\mathcal{H}^{\vee},\mathbb{Z}(1)).
$$  
The map $D$, which is given by duality, is an isomorphism because $\mathcal{H}$ is torsion-free.
\end{proof}

\subsection{Admissible classes}\label{sec:adm}

\begin{definition} Let $j:S\to \bar S$ and $i:\{s\}\to \bar S$ be as
  in \ref{s1.1}.  We call a holomorphic map $\bar\varphi:\Delta\to
  \bar S$ a \emph{test curve at $s$} if $\bar\varphi(0)=s$ and
  $\bar\varphi(\Delta^*)\subset S$.  Write $\varphi:\Delta^*\to S$ for
  the restriction of $\bar\varphi$ to $\Delta^*$.  Suppose
  $\mathcal{L}$ is a $\mathbb{F}$-local system on $S$ for $\mathbb{F}=\mathbb Q$
  or $\mathbb R$ (as in \S\ref{s1.2}).  We call a class
  $\alpha\in\rH^1_s(\mathcal{L})$ \emph{admissible} if the pull-back
  $\varphi^*\alpha$ in $\rH^1_0(\varphi^*\mathcal{L})$ vanishes for
  every test curve through $s$.  Similarly, we call a class
  $\alpha\in\rH^1(S,\mathcal{L})$ admissible if $\varphi^*\alpha$
  vanishes in $\rH^1(\Delta^*,\varphi^*\mathcal{L})$ for all test
  curves through all points of $\bar S$.  We write
  $\rH^1_s(\mathcal{L})^{\ad}$ (resp. $\rH^1(S,\mathcal{L})^{\ad}$)
  for the set of admissible classes in $\rH^1_s(\mathcal{L})$
  (resp. $\rH^1(S,\mathcal{L})$).  Similarly, we write
  $\IH^1_s(\mathcal{L})^{\ad}:=\IH^1_s(\mathcal{L})\cap
  \rH^1_s(\mathcal{L})^{\ad}$ and $\IH^1(\bar S,\mathcal{L})^{\ad}:=
  \IH^1(\bar S,\mathcal{L})\cap \rH^1(S,\mathcal{L})^{\ad}$.
\end{definition}

\begin{theorem}\label{IH1Ad}
  Let $\mathcal{L}$ be an variation of pure Hodge structure on $S$.
  Then for $s\in\bar S$, we
  have $\IH^1_s(\mathcal{L})=\IH^1_s(\mathcal{L})^{\ad}$.  In other
  words, every class is admissible.
\end{theorem}

We will prove the theorem after the following lemma.

\begin{lemma}\label{l.H1Weights}
  Suppose $\mathcal{L}$ is a variation of pure Hodge structure of 
weight $w$ on $\Delta^*$.    Then $\Gr^W_k\rH^1_0(\mathcal{L})=0$
for $k<w+2$.  In other words, $\rH^1_0\mathcal{L}$ has weights
in the interval $[w+2,\infty)$.    
\end{lemma}

\begin{proof}[Proof of Lemma~\ref{l.H1Weights}]
  We have an exact sequence 
  \begin{equation}
    \label{eq:Brylinski}
    0\to \rH^0_0(\mathcal{L})\to \psi_{z,1}\mathcal{L}\stackrel{N}{\to} 
\psi_{z,1}\mathcal{L}(-1)\to \rH^1_0(\mathcal{L})\to 0
  \end{equation}
where $\psi_{z,1}$ denotes the unipotent nearby cycles functor and
$N$ denotes the logarithm of unipotent part of the monodromy (of $\mathcal{L}$ 
around $0$).   The weight filtration on $\psi_{z,1}\mathcal{L}$ 
is induced from the monodromy filtration of $N$ so that 
$$
N^k:\Gr^W_{w-k}(\psi_{z,1}\mathcal{L}(k))\to \Gr^W_{w-k} \psi_{z,1}
\mathcal{L}
$$
is an isomorphism for $k\geq 0$.
So $\coker(N:\psi_{z,1}\mathcal{L}(1)\to \psi_{z,1}\mathcal{L})$
has weights in the interval $[w,\infty)$.   The result follows immediately.
\end{proof}

\begin{proof}[Proof of Theorem~\ref{IH1Ad}]
Suppose $\bar\varphi:\Delta\to \bar S$ is a test curve through $s$. 
Write $f$ for the composition 
$$\IH^1_s(\mathcal{L})\to \rH^1_s(\mathcal{L})
\to \rH^1_0(\varphi^*\mathcal{L}).$$
Then $f$ is a morphism of mixed Hodge structures.   By the purity
theorem of Cattani-Kaplan-Schmid and Kashiwara-Kawai, 
the source of $f$ has weights in the interval $(-\infty,w+1]$. 
But by Lemma~\ref{l.H1Weights}, the target has weights in the 
interval $[w+2,\infty)$.   It follows that $f=0$. 
\end{proof}

\section{The asymptotic height pairing in
the normal crossing case}\label{s.ed}

The goal of this section is to define the asymptotic height pairing on
local intersection cohomology in the normal crossing case using a
complex $B^*$ which computes the local intersection cohomology in this
case.  (See~Theorem~\ref{ahp0}.)

\subsection{Normal crossing setup.}\label{ed1} 
In this section, we write $\mathbb{F}$ for a field which is either
$\mathbb{Q}$ or $\mathbb{R}$  and $r$ for a fixed
non-negative integer.  Let $\Delta^r$ denote a polydisk with local
coordinates $(s_1,\dots,s_r)$ and $\Delta^{*r}$ be the complement
of the divisor $s_1\cdots s_r =0$.  Let $\mathcal{L}$ be a local system of
$\mathbb{F}$-vector spaces on $\delsr$ with unipotent monodromy, and we
write $L$ for a fixed fiber of $\mathcal{L}$.  We let $T_1,\ldots,
T_r\in {\rm Aut}_{\mathbb{F}} L$ denote the monodromy operators.  Then, for each $i$, 
the monodromy logarithm $N_i=\log T_i$ is nilpotent.

Note that , if $\mathcal{L}^*$ denotes the local system dual to
$\mathcal{L}$, then the monodromy operators on the fiber $L^*$ are
given by $(T_i^*)^{-1}$ and the logarithms are given by $-N_i^*$.

\subsection{Some linear algebra.}\label{ed.0} Suppose $T$ is an endomorphism of a
finite dimensional vector space $V$ over a field $F$.  Write
$T^*\in\End V^*$ for the adjoint of $T$ and let $(\, ,\, ): V\otimes
V^*\to F$ denote the canonical pairing.  There is a natural perfect
bilinear pairing
\begin{align*}
  (\, , \, )_T: TV\otimes T^*V^*&\to F, \text{ given by}\\
              Tv\otimes T^*\lambda& \mapsto (v,T^*\lambda)=(Tv,\lambda).
\end{align*}
To see that $(\, ,\, )_T$ is well-defined, note that $Tv=0\Rightarrow
(v,T^*\lambda)=(Tv,\lambda)=0$.   Similarly, $T^*\lambda=0\Rightarrow 
(Tv,\lambda)=(v,T^*\lambda)=0$. To see that the pairing is perfect,
suppose $(Tv,T^*\lambda)_T=0$ for all $\lambda$.  Then $(Tv,\lambda)=0$ for all $\lambda$.
So $Tv=0$.  Similarly, $(Tv,T^*\lambda)_T=0$ for all $v$ implies that 
$T^*\lambda=0$.

\subsection{Koszul and partial Koszul complexes.}\label{ed2}
The Koszul and partial Koszul complexes are complexes of vector spaces
which compute the cohomology groups $\rH^p(\mathcal{L})=\rH^p(\delsr,\mathcal{L})$
and the intersection cohomology groups $\IH^p(\mathcal{L})=\IH^p(\Delta^r, \mathcal{L})$
when $\mathcal{L}$ is as in \S\ref{ed1}.  We follow the notation of Kashiwara and Kawai
and Cattani, Kaplan and Schmid
for these complexes.  (See \cite[\S 3.4]{KK} and \cite{CKS87}.)  

Let $E$ denote an $r$ dimensional $\mathbb{F}$-vector space with basis
$\{e_1,\dots ,e_r\}$.  As a $\mathbb{F}$-vector space, the Koszul complex is  $K(\mathcal{L})=L\otimes
\bigwedge^*E$.  It is graded by the usual grading of $\bigwedge^*E$, and it
has differential given by
\begin{equation}
  \label{ed2.1}
  d(l\otimes\omega)=\sum_{i=1}^r N_il\otimes (e_i\wedge\omega).
\end{equation}
There is a canonical isomorphism $\rH^p(\mathcal{L})=\rH^p K(\mathcal{L})$.  

Suppose $J=(j_1,\ldots, j_p)$ is a multi-index in $\{1,\ldots, r\}$.  
Let  $N_J$ denote the product $N_{j_1}\cdots N_{j_p}$, and write  $e_J=e_{j_1}\wedge \cdots\wedge e_{j_p}$.
The partial Koszul complex is the subcomplex $B(\mathcal{L})$ of
$K(\mathcal{L})$ given by $\sum_{J}\, N_J(L)\otimes e_J$.  By the results of Kashiwara, Kawai~\cite{KK}
and Cattani, Kaplan, Schmid~\cite{CKS87}, there is a canonical isomorphism
$\IH^p(\mathcal{L})=H^p(B(\mathcal{L}))$. 
In particular, the canonical map 
$\IH^p(\Delta^r, \mathcal{L})\to \rH^p(\Delta^{*r},\mathcal{L})$ is precisely the map induced by the 
inclusion of complexes $B(\mathcal{L})
\to K(\mathcal{L})$.  

 Note that a morphism $f:\mathcal{L}\to \mathcal{M}$ of local
systems on $\Delta^{*r}$ induces morphisms $f:K^*(\mathcal{L})\to K^*(\mathcal{M})$ and $f:B^*(\mathcal{L})\to B^*(\mathcal{M})$ in the obvious way:
$l\otimes \omega \mapsto f(l)\otimes\omega$.  These induce the corresponding
homomorphisms $\rH^*(\mathcal{L})\to \rH^*(\mathcal{M})$ and $\IH^*(\mathcal{L})\to \IH^*(\mathcal{M})$.  

Note also that the complexes $K(\mathcal{L})$ and $B(\mathcal{L})$ 
depend on only on the data $(L,N_1,\ldots, N_r)$.  

\subsection{}
If $(C,d)$ is a complex and $p$ is a an integer, we write 
$Z^p(C)=\{\alpha\in C^p:d\alpha=0\}$.

\subsection{}\label{ed3} Suppose $t\in\mathbb{F}^r$.  Write $N(t):=\sum_{i=1}^r
t_i N_i$, and write $\mathcal{L}_t$ for the local system on $\Delta^*$ 
with monodromy logarithm $N(t)\in\End L$. 
Then, if $\mathbb{F}\{e\}$ denotes the free $\mathbb{F}$-vector
space on one generator $e$, we have 
 $K(\mathcal{L}_t)=L \otimes \bigwedge^*\mathbb{F}\{e\}$. 
As a complex, $K(\mathcal{L}_t)$ is just the map $N(t):L\to L$ 
placed in degrees $0$ and $1$.  

The map of vector spaces $E\to \mathbb{F}\{e\}$
given by $e_i\mapsto t_ie$
induces a map $\lambda_t:\wedge^*E\to \wedge^* \mathbb{C}\{e\}$.
Write $\phi_t^{\dagger}$
for the map $\mathrm{id}_L\otimes\lambda_t: K(\mathcal{L})\to 
K(\mathcal{L}_t)$.

If $\alpha=\sum_{i=1}^r\alpha_i\otimes e_i\in K^1(\mathcal{L})$,
write $\alpha(t):=\sum_{i=1}^r t_i\alpha_i$.
Then note that $\phi_t^{\dagger}\alpha=\alpha(t)\otimes e$.

\begin{lemma}\label{ed4}
  The map $\phi_t^{\dagger}:K(\mathcal{L})\to K(\mathcal{L}_t)$ is a morphism of complexes.
\end{lemma}
\begin{proof}
Since $K(\mathcal{L}_t)$ is concentrated in degrees $0$ and $1$, we just need to check
that $d\phi_t^{\dagger}(l)=\phi_t^{\dagger}(dl)$ for $l\in K^0(\mathcal{L})$.  
We compute $\phi_t^{\dagger}dl=\phi_t^{\dagger}\sum N_il\otimes e_i=
\sum t_iN_il\otimes e=N(t)l\otimes e$ as desired. 
\end{proof}

\subsection{}\label{ed5}  Suppose $t\in\mathbb{Z}_{\geq 0}^r$ and pick a $a\in\Delta^*$.
Define a test curve $\phi_t:\Delta^*\to \delsr$ by setting 
$\phi_t(s)=a(s^{t_1},\ldots, s^{t_r})$.   Then $\phi_t$ induces a map
\begin{equation}\label{ed6}
\phi_t^*:\rH^p(\Delta^{*r},\mathcal{L})\to \rH^p(\Delta^*,\mathcal{L}). 
\end{equation}
We leave it to the reader to check that, under the isomorphism of \S\ref{ed2},
$\phi_t^*=\phi_t^{\dagger}$.  

\subsection{}\label{HodgeOnDelta} In the rest of this section,
$\mathcal{H}$
will denote a variation of pure, polarized $\mathbb{F}$-Hodge
structure
of weight $k$ with unipotent monodromy on the poly-punctured disk 
$\delsr$.  Write $H$ for the limit mixed Hodge structure of $\mathcal{H}$.
In the language of nearby cycles,
$H=\psi_{s_1}\cdots \psi_{s_r}\mathcal{H}$.
As a vector space, $H$
is isomorphic to the fiber of $\mathcal{H}$
over any chosen point $s\in \delsr$.
For $i=1,\ldots, r$,
we write $T_i$
for the monodromy operator and $N_i=\log T_i$.
The $N_i$
then induce morphisms $N_i:H\to H(-1)$
of the mixed Hodge structure $H$.
We write $\rH^p(\calH)=\rH^p(\delsr,\mathcal{H})$
and $\IH^p(\calH)=\IH^p(\delr,\mathcal{H})$
for the cohomologies of the local system underlying the variation
$\mathcal{H}$.
By~\cite{CKS87, KK}, these are computed in terms of complexes
$K(\mathcal{H})=H\otimes \wedge^* E$
and $B(\mathcal{H})$
as in \S\ref{ed2}.  But here the vector space $E$
is viewed as the pure, weight $2$,
Hodge structure $\mathbb{F}(-1)^r$.  The map
$$d:K^p(\calH)=H\otimes \wedge^p E\to 
K^{p+1}(\calH)=H\otimes \wedge^{p+1} E$$ is then a morphism of mixed Hodge
structures.   

Write $0\in\delr$ for the origin.  Then we have equalities of Hodge structures
$\rH^p(\calH)=\rH^p_0(\calH)$ and 
$\IH^p(\calH)=\IH^p_0(\calH)$
 where the groups $\rH^p_0(\calH)$ and  $\IH^P_0(\calH)$ 
groups defined in~\ref{s1.2} and~\ref{s1.2b} respectively. 

\subsection{}
Write $F$ for the Hodge filtration on $H$.
Then the data $(N_1,\ldots, N_r;F)$
defines a nilpotent orbit.  
This data in turn defines a variation of Hodge structure $\mathcal{H}_{\nilp}$
of weight $k$ on $\Delta_b^{*r}$ for some $b>0$ (see~\eqref{eq:small-disk}).  Explicitly, 
$\mathcal{H}_{\nilp}$ is the variation with local normal form 
$e^{\sum_{i=1}^r z_iN_i} F$ where the $z_i$ are points in $\mathbb{C}$ such that
$|e^{2\pi iz_i}|<b$.  
We then have equalities
of complexes of mixed Hodge structure: $K(\mathcal{H})=K(\mathcal{H}_{\nilp})$, 
$B(\mathcal{H})= B(\mathcal{H}_{\nilp})$.   In particular, $H^p(\mathcal{H})$ and 
$\IH^P(\mathcal{H))}$ depend only on the nilpotent orbit associated to $\mathcal{H}$.
More precisely:
$\mathcal H$ and $\mathcal H_{\nilp}$ restrict to the same local system over $\Delta_b^{*r}$, and hence
the complexes $B^{\bullet}(\mathcal H)$ and $B^{\bullet}(\mathcal H_{\nilp})$ are equal.  Likewise, all
Hodge theoretic data attached to these complexes depends only the limit mixed Hodge structure, which
agrees for $\mathcal H$ and $\mathcal H_{\nilp}$ by construction.

We now prove versions of Lemma~\ref{l.H1Weights} and Theorem~\ref{IH1Ad}
for the normal crossing case.  

\begin{lemma}\label{l.H1W2}
Suppose $r=1$. Then $W_{k+1} \rH^1(\mathcal{H})=0$.
\end{lemma}
\begin{proof}
The proof is essentially the same as the proof of Lemma~\ref{l.H1Weights}. Set $N=N_1$.  As in the proof of Lemma~\ref{l.H1Weights}, we get an exact
sequence of mixed Hodge structures
$$
0\to \ker N\to H\stackrel{N}{\to} H(-1)\to (\coker N)(-1)\to 0.
$$
Then we use the same reasoning as in Lemma~\ref{l.H1Weights} to show that 
$W_{k+1}\rH^1(\mathcal{H})=0$. 
\end{proof}

\begin{proposition}\label{ced6}  Let
  $\alpha=\sum_{i=1}^r\alpha_i\otimes e_i\in
  Z^1(B(\mathcal{H}))$ be a representative of an
  intersection cohomology class $\bar\alpha\in \IH^1\mathcal{H}$.  Then,
  for every $t\in\mathbb{F}_{\geq 0}^{r}$, there is an element
  $l(t)\in H$ such that
\begin{equation}\label{ed6.1}
  \alpha(t)=N(t)l(t).
\end{equation}
\end{proposition}

\begin{proof}
  Suppose $t\in \mathbb{F}_{\geq 0}^r$.
  Then, if $F$
  denotes the Hodge filtration of $H$,
  the triple $(H,F, N(t))$
  defines a nilpotent orbit in one variable, and, thus, a
  $\mathbb{F}$-variation of Hodge structure on the disk $\Delta^*$.
Write $\mathcal{H}_t$ for this variation.  Then, by~\ref{l.H1W2}, 
$W_{k+1}H^1(\mathcal{H}_t)=0$.   

As in~\ref{ed4}, we get a morphism of complexes
$\phi_t^{\dagger}:K(\mathcal{H})\to K(\mathcal{H}_t)$.
Moreover, it is easy to see that $\phi_t^{\dagger}$
is, in fact, a morphism of complexes of mixed Hodge structures.
By~\S\ref{ed3}, the composition
$B^1(\mathcal{H})\to K(\mathcal{H})\to K(\mathcal{H}_t)$
takes a class $\alpha\in B^1(\mathcal{H})$ to $\alpha(t)\otimes e$.

By the purity theorem of~\cite{CKS87}, $W_{k+1}\IH^1(\mathcal{H})= \IH^1(\mathcal{H})$.   So, it follows from Lemma~\ref{l.H1W2} that the composition
$\IH^1(\mathcal{H})\to \rH^1(\mathcal{H})\to \rH^1(\mathcal{H}_t)$ is zero.

So suppose $\alpha\in B^1(\mathcal{H})$ represents a class in $\IH^1(\mathcal{H})$.  Then $\phi_t^{\dagger}\alpha=\alpha(t)\otimes e=dl(t)=N(t)l(t)$ for some $l(t)\in H$. The result follows. 
\end{proof}

\subsection{} Suppose $X\in E^*$.
Then contraction (or interior product) with $X$
gives a map $\iota_X:\wedge^p E\to \wedge^{p-1} E$.   We then get a $(-1,-1)$
morphism of mixed Hodge structures
$\delta_X:K^p(\mathcal{H})\to K^{p-1}(\mathcal{H})$ by 
setting $\delta_X(h\otimes \omega)=h\otimes \iota_X(\omega)$.   Note 
that $\delta_X$ preserves the restricted Koszul complex $B(\mathcal{H})\subset
K(\mathcal{H})$.   It is also easy to see that $\delta_X^2=0$.  

Suppose $t=(t_1,\ldots, t_r)\in \mathbb{F}^{r}$.
Then we write $X(t)=\sum t_ie_i^*\in E^*$ and $\delta_t=\delta_{X(t)}$.   We 
write 
\begin{equation}\label{hei-nablat}
\Laplace_t:=d\delta_t+\delta_td.
\end{equation}
Note that, for $\alpha=\sum \alpha_i\otimes e_i\in K^1(\mathcal{H})$, we have 
\begin{equation}\label{eq:1}
\delta_t\alpha = \sum t_i\alpha_i=\alpha(t).
\end{equation}

The following lemma is a variant of~\cite[Lemma 3.7]{CKS87}. (Indeed, it
can easily be deduced from~\cite[Lemma 3.7]{CKS87}, but it is also easy
enough to give a direct proof.)

\begin{lemma}\label{cks-lemma}
For $t\in\mathbb{F}^r$, write $N(t):K^p(\mathcal{H})\to 
K^p(\mathcal{H})$
for the operator given by 
$N(t)(h\otimes \omega)=(N(t) h)\otimes \omega$.  
Then
 $\Laplace_t= N(t)$. 
\end{lemma}
\begin{proof}
Suppose $\alpha=h\otimes\omega\in K^p(\mathcal{H})$ with $h\in H$ and $\omega\in\wedge^p E$.  Then 
\begin{align*}
  \Laplace_t \alpha&=d\delta_t\alpha +\delta_td\alpha \\
&=d(h\otimes \iota_{X(t)}\omega)+ \delta_t(\sum_{i=1}^r N_i h\otimes e_i\wedge\omega)\\
&=\sum_{i=1}^r N_i h\otimes e_i\wedge \iota_{X(t)}\omega + 
\sum_{i=1}^r N_i h \otimes \iota_{X(t)} (e_i\wedge\omega)\\
&=\sum_{i=1}^r N_i h\otimes e_i\wedge \iota_{X(t)}\omega + 
\sum_{i=1}^r N_i h \otimes t_i\omega - 
\sum_{i=1}^r N_i h \otimes e_i\wedge \iota_{X(t)}\omega\\
&=\sum_{i=1}^r t_i N_ih\otimes \omega=N(t)\alpha.
\end{align*}
\end{proof}

\subsection{A pairing on the $B$ complex}\label{ss:qt}
Suppose $I=(i_1,\ldots, i_k)$ is a multi-index in $\{1,\ldots, r\}$. 
For $t\in\mathbb{F}^r$, set $t_I=t_{i_1}t_{i_2}\cdots t_{i_k}$.    
Using \S\ref{ed.0}, we get a pairing 
\begin{align}
q_t:B^p(\calH)\otimes B^p(\calH^*)&\to \mathbb{F}(-p)\,
\text{given by}\label{qtpair} \\
q_t( N_I h\otimes e_I,N_J^*\lambda\otimes e_J) &= 
  \begin{cases}
    t_I(h,N_J^*\lambda), & I=J;\\ 
                0 , & \text{else}.
  \end{cases}\nonumber
\end{align}
Here $h\in H$ and $\lambda\in H^{*}$, and $I=(i_1,\ldots, i_k)$ and $J=(j_1,\ldots, j_l)$ are multi-indices
with $i_1<\ldots <i_k$ and $j_1<\ldots < j_l$.  

\begin{remark}\label{signremark}
Note that, for $h\in H$ and $\mu\in H^*$, 
$(N_Ih,\mu)=(h,N_I^*\mu)$.   It follows that 
$$q_t(N_Ih\otimes e_I,
N_I^*\lambda\otimes e_I) = t_I(N_I h, \lambda).$$
As we mentioned in~\S\ref{ed1}, the monodromy logarithms on 
$\calH^*$ are give by $-N_i^*$.  So let $\sw:B^p\calH\otimes B^p\calH^*\to
B^p\calH^*\otimes B^p\calH$ be the map switching the factors.  
From this it follows that $q_t\circ\sw=(-1)^pq_t$.  
\end{remark}

\begin{proposition}
  For each $t\in\mathbb{F}^r$, the pairing $q_t$ in~\eqref{qtpair}
is a morphism of mixed Hodge structure.   If $t\in(\mathbb{F}^{\times})^r$,
then $q_t$ is a non-degenerate pairing.  
\end{proposition}

\begin{proof}
  The fact that the pairing is a morphism of Hodge structures follows
  from the fact that each $N_i$
  induces a morphism $N_i:H\to H(-1)$
  of mixed Hodge structures (plus a little bit of bookkeeping about
  the number of Tate twists).  The fact that the pairing is
  non-degenerate for $t\in(\mathbb{F}^{\times})^r$ follows from~\ref{ed.0}.
\end{proof}

\begin{proposition}\label{p.adjoint}
For $\alpha\in B(\calH)$, $\beta\in B(\calH^*)$ and $t\in\mathbb{F}^r$,
we have
\begin{align} 
q_t(d\alpha,\beta)&=q_t(\alpha,\delta_t\beta);\label{adjoint1}\\
q_t(\delta_t\alpha,\beta)&=q_t(\alpha,d\beta)\label{adjoint2}.
\end{align}
\end{proposition}

\begin{proof}
  By the symmetry in the definition (or Remark~\ref{signremark}), it suffices to prove
  \eqref{adjoint1}.  For this, suppose $h\in H$ and $\lambda\in
  H^*$, and set $\alpha=N_Jh\otimes e_J$,
  $\beta=N_K^*\lambda\otimes e_K$ for multi-indices
  $J,K\subset \{1,\ldots,
  r\}$.  Then both sides of \eqref{adjoint1} are
  $0$ unless  there is an element 
$i\in K\setminus J$ and an $\epsilon\in\{\pm 1\}$ such that 
$e_i\wedge e_J=\epsilon e_K$. 

Now, we compute
\begin{align*}
q_t(d\alpha,\beta)&=q_t(N_iN_Jh\otimes e_i\wedge e_J, N_K^*\lambda\otimes e_K)\\
&=\epsilon q_t(N_Kh\otimes e_K, N_K^*\lambda\otimes e_K)\\
&= \epsilon  t_K(h, N_K^*\lambda).
\end{align*}

On the other hand, 
\begin{align*}
  q_t(\alpha,\delta_t\beta)&=q_t(N_Jh\otimes e_J, \delta_tN_K^*\lambda\otimes e_K)\\
&=q_t(N_Jh\otimes e_J,\delta_t N_K^*\lambda\otimes (\epsilon e_i\wedge e_J))\\
&=q_t(N_Jh\otimes e_J, t_i N_K^*\lambda\otimes (\epsilon e_J))\\
&=\epsilon t_iq_t(N_Jh\otimes e_J, N_K^*\lambda\otimes e_J)\\
&=\epsilon t_i t_J (h, N_K^*\lambda)\\
&= \epsilon t_K (h,N_K^*\lambda).
\end{align*}
\end{proof}

\begin{lemma}\label{RepInKernel}
Suppose $t\in\mathbb{F}_{\geq 0}^r$.  Then, 
every class in $\IH^1\calH$ has a representative $\alpha\in Z^1(B\calH)$ with 
$\delta_t\alpha=0$. 
\end{lemma}

\begin{proof}
  Suppose $\alpha'\in Z^1(B\calH)$.  Using Proposition~\ref{ced6}, find
  $l=l(t)\in H$ such that $N(t)l=\alpha'(t)$.  Then set
  $\alpha=\alpha'-dl$.  We have
  $\delta_t(\alpha)=\delta_t\alpha'-\delta_t dl =\alpha'(t)-\Laplace_t
  l=\alpha'(t)-N(t)l=0$.
\end{proof}

\begin{corollary}\label{crk}
  Set $L^p_t\calH=\ker \delta_t \cap \ker d \subset B^P(\calH)$.   For each $p$, 
the canonical map $L^p_t\calH\to \IH^p\calH$ is a morphism of mixed Hodge structure. 
The map $L^1_t\calH\to \IH^1\calH$ is surjective.  In fact, the map 
$W_{k+1} L^1_t\calH\to \IH^1\calH$ is surjective as well. 
\end{corollary}

\begin{proof}
 Since $\IH^p\calH=H^p(B\calH)$ as a mixed Hodge structure, the map 
$L_t^p\calH\to \IH^p\calH$ is a morphism of mixed Hodge structure.
The second assertion follows from Lemma~\ref{RepInKernel}.
The last assertion follows from the purity theorem~\cite{CKS87, KK} and the strictness of morphisms
of mixed Hodge structure with respect to the weight filtration.
\end{proof}

\subsection{} The restriction of $q_t$ to $L^1_t$ gives a pairing
\begin{equation}\label{qtonL}
q_t: L^1_t(\calH)\otimes L^1_t(\calH^{*})\to \mathbb{F}(-1).
\end{equation}
Since $L^1_t$
is a sub-mixed Hodge structure of
$B^1(\calH)$, this is a morphism of mixed Hodge structures.  Using the notation that
$\calH^{\vee}=\calH^*(1)$, we get a pairing $q_t:L^1_t\calH\otimes L^1_t\calH^{\vee}\to \mathbb{F}$.

\begin{lemma}\label{qtonbound} Suppose $\alpha\in L^1_t(\calH), \lambda\in B^0(\calH^*)=\calH^*$.
Then $q_t(\alpha,d\lambda)=0$.   Similarly, if $h\in B^0\calH$ and $\beta\in L^1_t(\calH^*)$, 
then $q_t(dh,\beta)=0$. 
\end{lemma}
\begin{proof}
For the first assertion, we have $q_t(\alpha,d\lambda)=q_t(\delta_t\alpha,\lambda)=0$. 
The second assertion has the same proof (or follows by symmetry).
\end{proof}

\begin{theorem}\label{ahp0}
Suppose $t\in\mathbb{F}_{\geq 0}^r$.  
  The pairing $q_t:L^1_t\calH\otimes L^1_t\calH^{\vee}\to \mathbb{F}$ descends to a pairing
$$
h(t):\IH^1\calH\otimes\IH^1\calH^{\vee}\to \mathbb{F}.
$$
This pairing, which we call the \emph{asymptotic height pairing} is a morphism of mixed
Hodge structures.   Therefore if $\calH$ has weight $k$, the pairing factors through a 
pairing
$$
\bar h(t): \Gr^W_{k+1}\IH^1\calH\otimes \Gr^W_{-k-1}\IH^1\calH^{\vee}\to\mathbb{F}.
$$
If $c\in\mathbb{F}_{\geq 0}$, then $h(ct)=ch(t)$.
\end{theorem}

\begin{proof}
  The first assertion follows from Corollary~\ref{crk} and Lemma~\ref{qtonbound}.  The fact
  that $h(t)$
  is a morphism of Hodge structures follows from the fact that
  \eqref{qtonL} is a morphism of mixed Hodge structures.  If $\calH$
  is pure of weight $k$,
  then $\IH^1\calH$
  has weights in the interval $(-\infty,k+1]$.   Similarly, 
$\IH^1\calH^{\vee}$ has weighs in the interval $(-\infty, -k-1]$.
 So $h(t)$ factors through through a pairing $\bar h(t)$ by weight 
considerations.  

To prove the last assertion, note that $q_{ct}=cq_t$ as a pairing on 
from $B^1(\calH)\otimes B^1(\chv)$ to $\mathbb{F}$.  And $\delta_{ct}=c\delta_t$
on $B^1(\calH)$.  It follows that $L^1_{ct}(\calH)=L^1_t(\calH)$ and
similarly $L^1_{ct}(\chv)=L^1_{c}(\chv)$.  Therefore $h(ct)=c(t)$. 
\end{proof}

\begin{proposition}\label{ahp1}
The pairing
\begin{equation}\label{ahp2}
h(t):\IH^1\mathcal{H}\otimes 
\IH^1\mathcal{H}^{\vee}\to\mathbb{F}
\end{equation}
can be computed as follows. Suppose $\alpha=\sum_{i=1}^r\alpha_i\otimes e_i\in
B^1(\mathcal{H})$ and $\beta=\sum_{i=1}^r\beta_i\otimes e_i\in
B^1(\mathcal{H}^{\vee})$ represent intersection cohomology 
classes $\bar\alpha$ and $\bar\beta$ respectively.   Fix $t$ as above and
$l(t)$ such that $N(t)l(t)=\alpha(t)$.  Then 
\begin{equation}\label{ahp3}
h(t)(\alpha,\beta)=\sum_{i=1}^r t_i(\alpha_i-N_il(t),\beta_i)_{N_i}.
\end{equation}
\end{proposition}

\begin{proof}
  Given $l(t)$ as above, set $\alpha'=\alpha-d l(t)$.  Then $\delta_t\alpha'=
\alpha(t)-\Laplace_t l(t)=0$.  If we have $\alpha_i=N_ih_i$ for $h_i\in H$, then 
\begin{align*}
h(t)(\bar\alpha,\bar\beta)&=q_t(\alpha',\beta)
                            =q_t(\alpha-dl(t),\beta)\\
  &=\sum_{i=1}^r t_i (h_i-l(t), \beta)
    =\sum_{i=1}^r t_i(\alpha_i-N_i l(t),\beta_i)_{N_i}
\end{align*}
as desired. 
\end{proof}

\begin{example}
Let $L=\mathbb{Q}^2$ with basis $u=(1,0)$ and $v=(0,1)$, and write $u^*, v^*$
for the dual basis.   Set
$$
N=
\begin{pmatrix}
  0 & 0 \\
  1 & 0 \\
\end{pmatrix}.
$$
Pick a positive integer $r$ and let $N_i=N$ for $1\leq i\leq r$.  Let
$\mathcal{L}$ denote the local system on $\Delta^{*r}$ with monodromy
logarithms $N_i$.  Then
$\IH^1(\mathcal{L})=\mathbb{Q}^r/\mathbb{Q}$ with the copy
of $\mathbb{Q}$ embedded diagonally.  An element of
$\IH^1(\mathcal{L})$ is represented by a sum 
$\alpha=\sum_{i=1}^r a_iv\otimes e_i$ with $a_i\in\mathbb{Q}$.  Similarly, and
element of $\IH^1(\mathcal{L}^*)$ is given by a sum 
$\beta=\sum_{i=1}^r b_iu^*\otimes e_i$.   
Since $\mathcal{L}$ underlies a pure variation
of Hodge structure of weight $-1$, $\alpha$ is admissible.   We can take 
$$
l(t)=\frac{\sum_{i=1}^r a_it_i}{\sum_{i=1}^r t_i} u.
$$
Then $N(t)l(t)=\alpha(t)$. 
We claim that 
\begin{equation}
\label{ej}
h(t)(\alpha,\beta)=
\frac{\sum_{i<j} (a_i-a_j)(b_i-b_j)t_it_j}{\sum t_i}.
\end{equation}
To see this, we compute
\begin{align*}
  h(t)(\alpha,\beta)&=\sum_i t_i(\alpha_i-N_il(t),\beta_i)_{N_i}\\
&=\sum_i t_i (a_iv - \frac{\sum_j a_jt_j}{\sum_j t_j} N_iu, b_i u^*)_{N_i}\\
&=\sum_i t_i((a_i -  \frac{\sum_j a_jt_j}{\sum_j t_j}) v, b_i v^*)\\
&=\sum_i t_i(a_i  -\frac{\sum_j a_jt_j}{\sum_j t_j})b_i\\
&=\frac{\sum_{i,j} a_i b_i t_it_j - a_jb_i t_i t_j}{\sum_j t_j}\\
&=\frac{\sum_{i,j} (a_i-a_j)b_i t_it_j}{\sum_j t_j}
\end{align*}
This is easily manipulated into the form in \eqref{ej}.
\end{example}

\section{Higher dimensional pairing and polarizations}\label{higher-p}

In this section $\mathcal{H}$ is a pure variation of $\mathbb{F}$-Hodge
structure of weight $k$ with unipotent monodromy on $\Delta^{*r}$ as in 
\S\ref{HodgeOnDelta}.   We fix a polarization 
$$Q:\mathcal{H}\otimes\mathcal{H}\to
\mathbb{F}(-k).$$ This induces a morphism
$a_Q:\mathcal{H}\to \mathcal{H}^*(-k)$
by $a_Q(h_1)(h_2)=Q(h_1,h_2)$.
By composing $h(t)$
with $a_Q$
we get a pairing on $\IH^1\mathcal{H}$ with values in $\mathbb{F}(-k-1)$.
In fact, we will show that for $t\in\mathbb{F}_+^r$
and any $p\in\mathbb{Z}$, we get a pairing on $\IH^p\mathcal{H}$.
The goal of this section is to show that these pairings are, in fact,
polarizations.   We remark that, for $t=(1,\ldots, 1)$, this is essentially
contained in the paper by Cattani, Kaplan and Schmid on Intersection
Cohomology~\cite{CKS87}.  

\begin{definition}
  For each $t\in\mathbb{F}^r$ and each $p\in\mathbb{Z}$, write
$$R_t^p(\mathcal{H})=\ker d\cap \ker \Laplace_t\cap B^p\mathcal{H}.$$
\end{definition}

Note that, for each $t\in\mathbb{F}^r$,
$R_t^p\mathcal{H}$
is a mixed Hodge substructure of $B^p\mathcal{H}$
containing the Hodge substructure $L_t^p\mathcal{H}$.

\subsection{Real splittings}\label{dsplit} Let $\mathcal D$ be a classifying space of pure Hodge structures
of $k$ which are polarized by a morphism $Q:V\otimes V\to\mathbb R(-k)$.  Let $\theta$ be a
nilpotent orbit with values in $\mathcal D$ generated by $(N_1,\dots,N_r;F)$.  Recall by 
Theorem~\eqref{mono-cone}, the monodromy weight filtration $W(N)$ is constant on the monodromy cone
$\mathcal C$ of positive $\mathbb R$-linear combinations of $N_1,\dots,N_r$.  Moreover, 
if $W = W(N)[-k]$ then $(F,W)$ is a mixed Hodge structure for which each element of $\mathcal C$ is a
$(-1,-1)$-morphism.  The associated $\delta$-splitting of $\theta$ is the nilpotent orbit
$\tilde\theta$ generated by $(N_1,\dots,N_r;\tilde F)$ where $(\tilde F,W) = (e^{-i\delta}F,W)$
is the Deligne $\delta$-splitting of $(F,W)$.  

\par In particular, since $\theta$ and $\tilde\theta$ have the same monodromy logarithms $N_1,\dots,N_r$
they determine the same underlying complex $B=B(N_1,\dots,N_r)$.  Likewise, since $\delta$ acts trivially
on $Gr^W$ it follows that $(F,W)$ and $(\tilde F,W)$ induce the same Hodge structure on $Gr^W(B)$.   
In~\cite{CKS87}, Cattani, Kaplan and Schmid use this method to reduce questions about the intersection cohomology groups of a general nilpotent orbit to the case where the nilpotent orbit has limit split over $\mathbb{R}$.

\begin{theorem}
For each $p\in\mathbb{Z}$ and each $t\in\mathbb{F}_+^r$, the map 
$R_t^p\mathcal{H}\to\IH^p\mathcal{H}$ is onto.  Moreover, $R_t^p\mathcal{H}$
has weights in the interval $(-\infty,p+k]$.  
\end{theorem}
\begin{proof}
It suffices to prove this for real nilpotent orbits.  
When $t=(1,1,\ldots, 1)$ and the limit is split over $\mathbb{R}$, the first
assertion is~\cite[(3.8)]{CKS87}.   We can reduce to the case where the limit is 
split over $\mathbb{R}$ as in \S\ref{dsplit}. 
Then we can reduce to the case to $t=(1,1,\ldots, 1)$ by scaling the
$N_i$ using Lemma~\eqref{scale-nilp}.  The second assertion follows from the definition of the weight 
filtration $W$ in terms of $N(t)$.  
\end{proof}

\begin{definition}
For each $t\in\mathbb{F}^r$, define a pairing $Q_t$ on $B^p\mathcal{H}$ by 
setting 
$$
Q_t(N_Iu\otimes e_I, N_Jv\otimes e_J)=
\begin{cases}
  t_I Q(u, N_J v), & I=J;\\
                  0, &\text{else.}
\end{cases}
$$
\end{definition}

When $t=(1,\ldots, 1)$, this pairing is used in~\cite[(3.4)]{CKS87}.
(They write $S$ instead of $Q$ for the polarization.)

\begin{proposition}\label{p.Rpol}
For each $t\in(\mathbb{F}^{\times})^r$, $Q_t$ defines a pairing 
$$
Q_t:B^p\mathcal{H}\otimes B^p\mathcal{H}\to 
\mathbb{F}(-p-k)
$$
which is a non-degenerate morphism of mixed Hodge structure.   
For $t\in\mathbb{F}_+^{r}$, $Q_t$ restricts to a polarization on 
$\Gr^W_{p+k}R_t^p\mathcal{H}$.  
\end{proposition}
\begin{proof}
Again it suffices to check this for real nilpotent orbits.
  The non-degeneracy follows from \S\ref{ed.0}.   The restriction
of the pairing to $R_t^p\mathcal{H}$ then 
gives a pairing 
\begin{equation}\label{rtor}
R_t^p\mathcal{H}\otimes R_t^p\mathcal{H}\to\mathbb{F}(-p-k).
\end{equation}
Since $R_t^p\mathcal{H}\subset W_{p+k}B^p\mathcal{H}$, the pairing
\eqref{rtor} vanishes on $W_{p+k-1}R^p_t\mathcal{H}$ for weight reasons.  
In the case that $t=(1,\ldots, 1)$ and the limit is split over $\mathbb{R}$
this pairing is a polarization by~\cite{CKS87}.   We reduce to that case 
by applying the $\delta$-splitting of \S\ref{dsplit} and rescaling.
\end{proof}

Note that a polarization on a pure Hodge structure induces a canonical 
polarization on any sub Hodge structure (by restriction), and it also
induces a canonical polarization on any quotient Hodge structure 
(by orthogonal projection).  
Proposition~\ref{p.Rpol} therefore gives a canonical polarization
\begin{equation}\label{higherhp}
\overline{Q}_t:\Gr^W_{p+k} \IH^p\mathcal{H}\otimes\Gr^W_{p+k}\IH^p\mathcal{H}
\to\mathbb{F}(-p-k)
\end{equation}
for $t\in\mathbb{F}_+^r$. 
Since $\IH^p\mathcal{H}$ is concentrated in weights $(-\infty,p+k]$,
the pairing~\eqref{higherhp} gives a pairing
\begin{equation}
  \label{higherhp2}
  \IH^p\mathcal{H}\otimes\IH^p\mathcal{H}\to\mathbb{F}(-p-k)
\end{equation}
also denoted by $\overline{Q}_t$ 
(with kernel $W_{p+k-1}\IH^p\mathcal{H}$).

\begin{lemma}
  For $\alpha,\beta\in B(\mathcal{H})$ we have $Q_t(d\alpha,\beta)=
Q_t(\alpha,\delta_t\beta)$.  
\end{lemma}
\begin{proof}
The proof is essentially the same as the proof of Proposition~\ref{p.adjoint}.
\end{proof}

\begin{corollary}\label{LPerpToKer}
  Suppose $\alpha\in L_t^p\mathcal{H}$ and $\beta\in \ker(R^p_t\mathcal{H}\to 
\Gr^W_{p+k}\IH^{p}\mathcal{H})$.  Then $Q_t(\alpha,\beta)=0$.
\end{corollary}
\begin{proof}
  Since $\beta\in R^P_t\mathcal{H}$,
  we have $\beta\in W_{p+k}B^P\mathcal{H}$.
  Then, the hypothesis implies that $\beta=d\gamma + \eta$
  where $\gamma\in W_{p+k}B^{p-1}\mathcal{H}$
  and $\eta\in W_{p+k-1}B^P\mathcal{H}$.
  So
  $Q_t(\alpha,\beta)=Q_t(\alpha, d\gamma+\eta)=
  Q_t(\delta_t\alpha,\gamma)+Q_t(\alpha,\eta)=Q_t(\alpha,\eta)=0$ for
  weight reasons.
\end{proof}

\subsection{} Suppose $\mathcal{H}$ is pure of weight $-1$.  Then the map 
$a_Q:\mathcal{H}\to \mathcal{H}^*(1)=\mathcal{H}^{\vee}$ induces maps $B^P\mathcal{H}
\to B^p(\mathcal{H}^{\vee})$ and  
$\IH^p\mathcal{H}\to \IH^p(\mathcal{H}^{\vee})$, which we also denote by 
$a_Q$.   Write $h_Q(t):\IH^1\mathcal{H}\otimes \IH^1\mathcal{H}\to
\mathbb{Q}$ for the pairing
\begin{equation}
  \label{e.DefOfhQ}
  ([\alpha],[\beta])\mapsto h(t)(a_Q[\alpha],[\beta]).
\end{equation}

\begin{theorem}
Suppose $[\alpha], [\beta]\in \IH^1\mathcal{H}$ with $\mathcal{H}$ pure
of weight $-1$.   
$$
h_Q(t)([\alpha], [\beta])=\overline{Q}_t([\alpha],[\beta]).
$$
\end{theorem}
\begin{proof}
 We can find $\alpha,\beta\in L^1_t\mathcal{H}$ representing $[\alpha]$
 and $[\beta]$ respectively.   Then, by Corollary~\ref{LPerpToKer} and
 the definition of $\overline{Q}_t$, we have 
$h(t)(a_Q[\alpha],[\beta])=q_t(a_Q\alpha, \beta)=\overline{Q}_t([\alpha],[\beta])$.
\end{proof}

The remainder of this section concerns the pairing $h(t)$ in a special
case which will be useful for studying the Ceresa cycle.

\begin{lemma}\label{cer-3}
  Suppose $L$ is a local system with unipotent monodromy on
  $\Delta^{*2}$.  Write $\gamma_i$, $i=1,2$ for the counterclockwise
  loops around $s_i=0$.  (So that $\pi_1(\Delta^{*2})$ is the free
  abelian group generated by $\gamma_1$ and $\gamma_2$.)  Suppose that
  the monodromy logarithms $N_1$ and $N_2$ are equal.  Write $N=N_1$
  and suppose further that $N^2=0$.  Let
  $i:\mathbb{Z}\to \pi_1(\Delta^{*2})$ be given by
  $1\mapsto \gamma_2-\gamma_1$.  Then
  \begin{enumerate}
  \item The map $NL\to B^1(L)$ given by $Nv\mapsto (0,Nv)$
    induces an isomorphism $NL\to \IH^1(L)$.
  \item The composition
    $$NL\to\IH^1(L)\to H^1(\Delta^{*2},L)\stackrel{i^*}{\to} H^1(\mathbb{Z},L)=L
    $$
    is the inclusion.
  \end{enumerate}
\end{lemma}

  \begin{proof}
    Let $C(L)$ denote the complex $\ker N\to NL$ where the
    differential is $0$.  Then under the hypotheses, there is a
    morphism $C(L)\to B(L)$ which is the obvious inclusion on $C^0$
    and where $C^1(L)\to B^1(L)$ is given by $Nv\mapsto (0,Nv)$. It is
    easy to see that this morphism is a quasi isomorphism of complexes.  This proves (i).

    To prove (ii), we note that $H^1(\mathbb{Z},L)$ is computed by the Koszul
    complex $K(L)$ given by $L\to L$ with the $0$ differential.  Then the composition in (ii) is the map induced on the first cohomology by the following
    composition of complexes:
    $$\xymatrix{
      \ker N\ar[r]^{0}\ar[d] & 0\oplus NL\ar[d] \\
      L\ar[r]^{d}\ar[d]            & L\oplus L\ar[d]^{(-\id,\id)} \\
        L\ar[r]^{0}            & L.\\
        } $$
        (Here $d$ denotes the morphism in the Koszul complex.)
  \end{proof}

  \begin{proposition}\label{simple-formula}
    Suppose $\mathcal{H}$ is a variation of Hodge structure
    of weight $-1$ on $\Delta^{*2}$ with unipotent monodromy
    logarithms.  Suppose $N_1=N_2=N$ and $N^2=0$.
    Let $Q$ be a polarization of $\mathcal{H}$.   Identify $NH$ with
    $\IH^1(H)$ using Lemma~\ref{cer-3}.  Then, for each $t\in\mathbb{Q}^2_{\geq 0}$, the pairing
    $$
    h_Q(t):\IH^1\mathcal{H}\otimes\IH^1\mathcal{H}\to \mathbb{Q}
    $$
    amounts to a paring
    $$h_Q(t):NH\otimes NH\to \mathbb{Q}.$$
    This pairing is given, for $(t_1,t_2)\neq  (0,0)$ by
    $$
    h_Q(t)(Nh,Nk)=\frac{t_1t_2}{t_1+t_2}Q(h,Nk),
    $$
    and vanishes for $(t_1,t_2)=(0,0)$.  
  \end{proposition}

  \begin{proof}
    Under the isomorphism $NH\to\IH^1\mathcal{H}$ of Lemma~\ref{cer-3},
    the element $Nh$ (resp.  $Nk$) corresponds to the class
    $\alpha=Nh\otimes e_2$ (resp. 
    $\beta=Nk\otimes e_2$) in  $B^1H$.
    For, $t\in \mathbb{Q}_{\geq 0}^2\setminus\{(0,0)\}$, set 
    \begin{align*}\beta(t)&=Nk\otimes e_2 - d\left(\frac{t_2}{t_1+t_2} k\right)\\
                          &=-\frac{t_2}{t_1+t_2} Nk\otimes e_1+\frac{t_1}{t_1+t_2}Nk\otimes e_2.
    \end{align*}
    And set $\beta(t)=Nk\otimes e_2$ for $t=(0,0)$.  
    Then $\delta_t\beta=0$.   So $\beta(t)\in L^1_t\mathcal{H}$.
    Therefore, the pairing vanishes when $t=(0,0)$, and, when $t$ is non-zero, 
    $\displaystyle h_Q(t)(\alpha,\beta)=h(t)(a_Q\alpha,\beta)=
    q_t(a_Q\alpha,\beta(t))=
    t_2Q(h, \frac{t_1}{t_1+t_2}Nk)= \frac{t_1t_2}{t_1+t_2}Q(h, Nk)$.
  \end{proof}

\section{Comparison Theorem}\label{several}

\par In this section we relate the asymptotic height pairing \eqref{ahp0} and the 
asymptotics of the height of the biextension line bundle.  We begin with a minor
modification of Remark \ref{rat-mu}.

\begin{lemma}\label{stronger-mu} Let $V_{\mathbb Q}=\mathbb Q v_0\oplus U_{\mathbb Q}\oplus\mathbb Q v_{-2}$
be a finite dimensional vector space.  Suppose that $N$ is a nilpotent endomorphism of $V_{\mathbb Q}$
such that 
$$
   N(v_0) = \mu v_{-2},\qquad N(U_{\mathbb Q})\subseteq U_{\mathbb Q},\qquad N(v_{-2}) = 0.
$$
Let $e_0 = v_0 + u + c v_{-2}$ be another element of $V_{\mathbb Q}$ such that $N(e_0) = \mu' v_{-2}$
where $u\in U_{\mathbb Q}$.  Then, $\mu = \mu'$.
\end{lemma}
\begin{proof} 
By definition, 
$$
          (\mu' - \mu)v_{-2} = N(e_0 - v_0) = N(u + cv_{-2}) = N(u).
$$  
Since $N(u)\in U_{\mathbb Q}$ and $U_{\mathbb Q}\cap\mathbb Q v_{-2} = 0$ it follows that 
both the right and left hand sides of the previous equation vanish.
\end{proof}

\begin{theorem}\label{JDG5} Let $\mathcal V$ be an admissible biextension variation
with unipotent monodromy over $\Delta^{*r}$ with underlying normal functions 
$\nu\in\ANF(\Delta^{*r},\mathcal H)$ and $\omega\in\ANF(\Delta^{*r},\mathcal H^{\vee})$.
Let $\sing(\nu)=\alpha$ and $\sing(\omega) = \beta$.  Define $\mu(\mathcal V)$ as
in equation \eqref{eq:mu-def}.  Define $\mu_1,\dots,\mu_r$ as in Remark \eqref{rat-mu}.
Then, for any $t\in\mathbb Z^r_{\geq 0}$, 
\begin{equation}
  h(t)(\alpha,\beta) = - \mu(\mathcal V)(t) + \sum_j\, t_j\mu_j .   \label{eq:HeightJumpRel}
\end{equation}
In particular, $h(t)(\alpha,\beta)\equiv -\mu(\mathcal V)(t)$ modulo a linear
function of $t$. 
\end{theorem}

\begin{proof}
Let $V$ be the $\mathbb{Q}$-vector space corresponding to a reference fiber of $\mathcal V$.  Pick a splitting of the weight filtration $W$ of 
$V$ so that $V=\mathbb{Q}e_0 \oplus H
\oplus \mathbb{Q}e_{-2}$ where: 
\begin{enumerate}
\item the image of $e_0$ is the canonical generator of $\Gr^W_0 V=\mathbb{Q}$;
\item $H$ maps isomorphically to $Gr^W_{-1}$;
\item $e_{-2}$ is the canonical generator of $W_{-2} V=\mathbb{Q}(1)$. 
\end{enumerate}
Write $\bar N_i$ for the $i$-th monodromy operator on $H$ and write $N_i$ for the 
$i$-th monodromy operator on $V$.   Then, in terms of the splitting of $W$ chosen above,
we have 
\[
N_i=
\begin{pmatrix}
  0       &    0      &    0 \\
 \alpha_i &  \bar N_i  &    0\\
 \gamma_i & -\beta_i   &   0 
\end{pmatrix}
\]
where $\gamma_i\in \mathbb{Q}(1)$ and 
$(\alpha_1,\ldots, \alpha_r)$ (resp. $(\beta_1,\ldots, \beta_r)$)
is a representative of the class $\alpha$ (resp. $\beta$) 
in $Z^1(B(H))$ (resp. $Z^1(B(H^{\vee}))$.  (\S\ref{ed1} explains why
$\beta_i$ appears with a minus sign in the matrix of $N_i$.)

\par Set $\bar N(t)=\sum t_i \bar N_i, \gamma(t)=\sum t_i\gamma_i$, etc.
Let $l(t)\in H$ be an element such that $\bar N(t)l(t)=\alpha(t)$. 
Set $e_0(t)=e_0- l(t)$.
Then, 
\begin{align*}
  N(t)e_0(t)&=\alpha(t)+\gamma(t)e_{-2}-\bar N(t)l(t) + (l(t),\beta(t))e_{-2}\\
            &=(\gamma(t) + (l(t),\beta(t)))e_{-2}.
\end{align*}
Therefore, by Remark~\eqref{rat-mu} and Lemma \eqref{stronger-mu}, 
$\mu(\mathcal V)(t) = \gamma(t) + (l(t),\beta(t))$.
Likewise, we have $\mu_i = \gamma_i + (l_i,\beta_i)$ where $\bar N_i(l_i) = \alpha_i$.

\par On the other hand, by \eqref{ahp3}
$$
h(t)(\alpha,\beta)=\sum_{i=1}^r\, t_i(\alpha_i-N_il(t),\beta_i)_{N_i}.
$$
Consequently,
$$
\aligned
       \mu(\mathcal V)(t) + h(t)(\alpha,\beta) 
       &= \gamma(t) + (l(t),\beta(t)) + \sum_{i=1}^r\, t_i(\alpha_i-N_il(t),\beta_i)_{N_i}  \\
       &= \gamma(t) + (l(t),\beta(t)) + \sum_{i=1}^r\, t_i(\alpha_i,\beta_i)_{N_i}
          -\sum_{i=1}^r\, t_i(N_il(t),\beta_i)_{N_i}  \\
       &=  \gamma(t) + (l(t),\beta(t)) + \sum_{i=1}^r\, t_i(\alpha_i,\beta_i)_{N_i}
          -\sum_{i=1}^r\, t_i(l(t),\beta_i)  \\
       &=  \gamma(t) + (l(t),\beta(t)) + \sum_{i=1}^r\, t_i(\alpha_i,\beta_i)_{N_i}
          -(l(t),\beta(t))  \\
       &=  \gamma(t) + \sum_{i=1}^r\, t_i(\alpha_i,\beta_i)_{N_i}  \\
       &=  \gamma(t) + \sum_{i=1}^r\, t_i(l_i,\beta_i) \\
       &=  \sum_{i=1}^r\, t_i \mu_i
\endaligned
$$
\end{proof}

\section{Boundary Values}\label{bdv} Let $\mathcal V\to\Delta^{*r}$ be an admissible biextension variation
with unipotent monodromy, and $C = \mathbb R^r_{>0}$.  Then, as noted in Remark \eqref{rat-mu}, the 
function $\mu(\mathcal V)(m)$ is defined on the closure $\bar C$ of $C$.  

\begin{lemma} Let $J\subset\{1,\dots,r\}$.  Then $\mu(\mathcal V)(m_1,\dots,m_r)$ is given by a rational
function on the subset $\bar C_J\subset \bar C$ on which $m_j>0$ if $j\in J$ and $m_j=0$ if
$j\notin J$.
\end{lemma}
\begin{proof} The proof Theorem \eqref{JDG5} shows that $\mu(\mathcal V)(m) = \gamma(m) + (\ell(m),\beta(m))$.
Accordingly, the content of the lemma boils down to showing we can find a solution to the equation
$$
              N(m)\ell(m) = \alpha(m)
$$
using rational functions of $m$.  Without loss of generality, we assume that 
$J = \{1,\dots,a\}$.  Let $\mathcal H = Gr^W_{-1}\mathcal V$ with reference fiber $H$.  
Then, the nilpotent orbit of $\mathcal H$ determines an associated nilpotent orbit 
$$
          \tilde\theta(z) = e^{\sum_k\, z_k N_k}\tilde F
$$
with limit mixed Hodge structure split over $\mathbb R$ as in \S 2.  Moreover, $\tilde\theta(z)$
takes values in the appropriate classifying space of pure Hodge structure as soon as 
$\text{\rm Im}(z_1),\dots,\text{\rm Im}(z_r)>0$.  Therefore, 
$$
        \tilde\theta_a(z_1,\dots,z_a) = \tilde\theta(z_1,\dots,z_a,i,\dots,i)
$$
is a nilpotent orbit.  By~\cite{CKS86}, we therefore obtain an $sl_2$-triple $(N(m),H,N^+(m))$ for each 
$m\in \bar C_J$ where $H$ is a fixed semisimple element.  

\par To continue, we note that the fact that $H$ is constant reflects the fact that the monodromy weight
filtration $W_J := W(N(m))$ is constant on $\bar C_J$.  Moreover, $\alpha(m)$ and $\beta(m)\in W_{-1}(J)$.
Let $\alpha_{-1}(m)$ denote the projection of $\alpha(m)$ to the $-1$-eigenspace of $H$.  Then, since 
$N(m):Gr^{W_J}_1\to Gr^{W_J}_{-1}$ is an isomorphism, there exists a unique element $\ell_1(m)$ in 
the $+1$ eigenspace of $H$ such that $N(m)\ell_1(m) = \alpha_{-1}(m)$.  For weight reasons,
$$
         (\ell(m),\beta(m)) = (\ell_1(m),\beta(m))
$$
Accordingly, the rationality of $\mu(m)$ boils down to the rationality of the inverse map
$N(m)^{-1}:Gr^{W_J}_{-1}\to Gr^{W_J}_1$.
\end{proof}

\par Given this lemma, it is natural to ask of if $\mu(\mathcal V)$ has a continuous extension
from $C$ to $\bar C$.  To establish this, let $\mathcal G$ be a class of sequences of points
in $C$ with the property that any sequence of points $\{m(p)\}$ in $C$ which converges to 
$m_o\in\bar C-C$ contains a subsequence which belongs to $\mathcal G$.

\begin{lemma} Let $f:C\to\mathbb R$ be a function which has limit $L$ along every $\mathcal G$-sequence 
which converges to $m_o\in\bar C-C$.  Then, $f$ has limit $L$ along every sequence in $C$ converging to $m_o$.
\end{lemma}
\begin{proof} Let $m(p)$ be a sequence in $C$ which converges to $m_o$.  Then, there exists a real
number $B$ and an index $p'$ such that $|f(m(p))|<B$ for all $p>p'$.  Indeed, if this not true we
can find a subsequence $m(p_i)$ such that $|f(m(p_i)|>i$ for all $i$ sufficiently large.  By the
defining property of $\mathcal G$, we can find a $\mathcal G$-subsequence of $m(p_i)$ which converge 
to $m_o$.  By hypothesis, $f$ has limit $L$ along $\mathcal G$-sequences through $m_o$.

\par Suppose now that $\lim_{p\to\infty}\, f(m(p))\neq L$.  Then, there exists an $\epsilon>0$
and a subsequence $m(p_i)$ such that $|f(m(p_i))-L|\geq\epsilon$ for all $i$.  Passage to
a $\mathcal G$-subsequence of $m(p_i)$ again produces a contradiction.
\end{proof}

\par To apply this lemma, we recall that in~\cite{BP-Comp} the authors defined a notion of 
$\text{\rm sl}_2$-sequences on
$$
         I =\{ (z_1,\dots,z_r) \mid \text{\rm Re}(z)\in[0,1],\quad\text{Im}(z)\in[1,\infty)\,\}
$$
which has the property that every sequence in $I$ contains an $\text{\rm sl}_2$-sequence.
Given a point $m_o\in\bar C - C$ and a sequence of points $m(p)$ in $C$ which converges
to $m_0$, define
\begin{equation}
        \tilde m(p) = m(p)/\min(m_1(p),\dots,m_r(p))           \label{eq:scaled-m}
\end{equation}
keeping in mind that each $m_j(p)>0$.  Then, $\sqrt{-1}\tilde m(p)$ is a sequence of points in $I$.
We say that $m(p)$ is a $\mathcal G$-sequence if $i\tilde m(p)$ is a $\text{\rm sl}_2$-sequence.

\par As our next preliminary, given a mixed Hodge structure $(F,W)$ let $\hat Y_{(F,W)}$
denote the Deligne grading of the $\text{\rm sl}_2$-splitting of $(F,W)$.  Next we note
(cf.~\cite{BP-Comp}) that if $(N_1,\dots,N_r;F,W)$ generate an admissible nilpotent orbit
$\theta$ with limit mixed Hodge structure split over $\mathbb R$ then 
\begin{equation}
         \hat Y_{(e^{icN}F,W)} = \hat Y_{(e^{iN}F,W)}                 \label{eq:DS-scale}
\end{equation}
for any positive real number $c$ and any $N\in\mathcal C =\{\sum_j\, y_j N_j \mid  y_1,\dots,y_r>0\}$.

Finally, suppose that $\theta$ is a nilpotent orbit of biextension type with reference fiber
$V$ and let $e_0\in V$ project to $1\in Gr^W_0$ and $e_{-2}$ be the generator of $Gr^W_{-2}$.
Then, for any $m\in C$
\begin{equation}
         \mu(\theta)(m)e_{-2} = \frac{1}{2}[N(m),\hat Y_{(e^{iN(m)}F,W)}]e_0.    \label{eq:DS-mu}
\end{equation}
Indeed, by the properties of Deligne systems and the short length of $W$, the grading 
$\hat Y_{(e^{iN}F,W)}$ produces a lift $v_0$ of $1\in Gr^W_0$ such that $N(v_0)$
belongs to $W_{-2}$, so we can apply Lemma \eqref{stronger-mu} to conclude the stated
formula.

\begin{lemma}\label{slem}
  Suppose $f:\bar C\to\mathbb{R}$ is a function,
  and suppose that, for every sequence $m(p)$ in $C$ converging
  to $m_0\in\bar C$,
  $\lim f(m(p))=f(m_0)$.  Then $f$ is continuous on $\bar C$.
\end{lemma}

\begin{proof}
  By the definition of continuity, $f$ is continuous on $C$.  Suppose
  that $m(p)$ is a sequence in $\bar C$ converging to $m_0$.  For each $p$,
  we can find a sequence $\{n(p,q)\}_{q=1}^{\infty}$ of points in $C$
  converging to $m(p)$.   By our hypothesis, for each $p$, we have
  $\lim_{q\to\infty} f(n(p,q))=f(m(p))$.  So we can find a $q\in\mathbb{Z}_{>0}$
  such that $$\max(| n(p,q)-m(p)|, |f(n(p,q))-f(m(p))|)<2^{-p}$$
  where we write $|n(p,q)-m(p)|$ for the usual Euclidean norm.
  Set $n(p):=n(p,q)$.   Then $n(p)\to m_0$, so, by our hypotheses, $f(n(p))\to
  f(m_0)$.  Moreover, $\lim_{p\to\infty} f(n(p))=\lim_{p\to\infty} f(m(p))$.
  So, since $m(p)$ was an arbitrary sequence in $\bar C$ converging to $m_0$,
  this proves that $f$ is continuous. 
\end{proof}

\begin{theorem}\label{cext} $\mu(\mathcal V)$ has a continuous extension from $C$ to $\bar C$.
\end{theorem}
\begin{proof} The first step is to observe that $\mu(\mathcal V) = \mu(\theta)$
where $\theta$ is the nilpotent orbit generated by $(N_1,\dots,N_r;\hat F,M)$
where $(F,M)$ is $\text{\rm sl}_2$-splitting of the limit mixed Hodge
structure of $\mathcal V$.  Indeed, by Remark \ref{rat-mu}, the value of $\mu$
depends only on structure of the local monodromy logarithm.

\par To continue, let $m(p)$ be a sequence in $C$ which converges to $m_o\in\bar C-C$ and
$\tilde m(p)$ be the corresponding sequence \eqref{eq:scaled-m} in $I$.  Then, by 
property \eqref{eq:DS-scale} we can rewrite \eqref{eq:DS-mu} as
$$
      \mu(\theta)(m(p))e_{-2} = \frac{1}{2}[N(m(p)),\hat Y_{(e^{iN(\tilde m(p))}F,W)}]e_0
$$
By Theorem $(2.30)$ of \cite{BP-Comp}, the grading $\hat Y_{(e^{iN(\tilde m(p))}F,W)}$ has 
limits of the form
$$
           Y(N(\theta^1),Y(N(\theta^2),\dots,Y_{(F,M)}))
$$
along $\text{sl}_2$-sequences (see \cite{BP-Comp} for notation).  The first key thing 
to know about this limit is that $N(\theta^1) = N(m_o)$ since $N(m_0)/\min(m_j(p))$ is
the dominant term of $N(\tilde m(p))$.  The second key thing to know is that 
$$
           [N(m_o),Y(N(m_o),Y(N(\theta^2),\dots,Y_{(F,M)}))]
$$
belongs to $W_{-2}\text{\rm End}(V)$.  Therefore,
$$
        \mu(\theta)(m_o)e_{-2} = \frac{1}{2}[N(m_o),Y(N(m_o),Y(N(\theta^2),\dots,Y_{(F,M)}))]e_0
$$
which proves that $\lim_{m\to m_o}\, \mu(\theta)(m) = \mu(\theta)(m_o)$ for sequences $m(p)\in C$
tending to $m_o\in\bar C-C$.   Now apply Lemma~\ref{slem}.
%
\end{proof}

\section{Mixed Extensions and Biextensions}\label{meb}
In~\cite[IX.9.3, pp. 421--426]{SGA71}, Grothendieck develops the
theory of mixed extensions (extensions panach\'ees) in an abelian
category.  This turns out to be a convenient
way to think about mixed extensions of normal functions and the variations
of Hodge structure which we have been calling biextension variations.  In this section, we briefly review this notion and its relation to the Mumford-Grothendieck's concept
of a biextension.

\subsection{}\label{p.torsor}
Suppose $G$ is a group acting on a set $X$.  If $X$ is isomorphic
to $G$ regarded as a $G$-set with the action of left multiplication,
then $X$ is called a torsor for $G$.  If $X$ is either empty or isomorphic
to $G$, then $X$ is called an \emph{pseudo-torsor} for $G$.   

Similarly a sheaf $\mathcal{F}$ is a called a \emph{pseudo-torsor} for a 
sheaf of groups $\mathcal{G}$ if there is an action of $\mathcal{G}$ on 
$\mathcal{F}$ and, for each open $U$, either $\mathcal{F}(U)$ is empty
or $\mathcal{G}(U)$ acts simply transitively on $\mathcal{F}(U)$.  
(See the Stacks Project~\cite[Tag 03AH]{stacks-project}.)

\subsection{}\label{s.MixedExt} Suppose $Q_0,Q_1$ and $Q_2$ are
three objects in an abelian category $\mathbf{C}$.  A \emph{mixed extension}
of $Q_0$ by $Q_1$ by $Q_2$ is an object $X$ of $\mathbf{C}$ with an increasing
filtration $W_i X$ such that $W_{-3}X=0$ and $W_0X=X$ together with isomorphisms
$p_i:\Gr^W_{-i} X\stackrel{\sim}{\to} Q_i$.

Write $\EXTpan(Q_0,Q_1,Q_2)$ for the category of mixed extensions of
$Q_0$ by $Q_1$ by $Q_2$.  Here a morphism in $\EXTpan(Q_0,Q_1,Q_2)$ from $X$
to $X'$ is a morphism of objects in $\mathbf{C}$ commuting with the
isomorphisms from $p_i$. Let 
$\Extpan(Q_0,Q_1,Q_2)$ denote the set of isomorphism classes in
$\EXTpan(Q_0,Q_1,Q_2)$.

From a mixed extension $X$ of $Q_0$ by $Q_1$ by $Q_2$, we get an extension 
$W_0X/W_{-2} X$ of $Q_0$ by $Q_1$ and an extension $W_{-1} X$ of $Q_1$ by $Q_2$.  
In fact, we get a functor
\begin{equation}
  \label{e.biextpan}
  \EXTpan(Q_0,Q_1,Q_2)\to \EXT(Q_0,Q_1)\times \EXT(Q_1,Q_2).
\end{equation}
Here $\EXT(Q_0,Q_1)$ denotes the usual category of extensions of $Q_0$ by $Q_1$. 
Given extensions $E_{0}$ of $Q_0$ by $Q_1$ and $E_{1}$ of $Q_1$ by $Q_2$, Grothendieck
lets 
$\EXTpan(E_0,E_1)$
denote the category of all mixed extensions of $Q_0$ by $Q_1$
by $Q_2$ together with isomorphisms
$
\pi_i:W_{-i} X/W_{-i-2} X \stackrel{\sim}{\to} E_i
$
for $i=0,1$. 
The morphisms in $\EXTpan(E_0,E_1)$
are morphisms $X\to X'$ in $\mathbf{C}$ commuting with the isomorphism
$\pi_i$ for $i=0,1$.

\subsection{Action of $\Ext(Q_0,Q_2)$}
Given an extension $G$ of $Q_0$ by $Q_2$ 
and an object $X$ of $\EXTpan(E_0,E_1)$, Grothendieck defines a mixed extension 
$G\wedge X$ as follows.  First, regard $X$ as an object of the category
$\EXT(Q_0,E_1)$ of extensions of $Q_0$ by $E_1$.   The sequence 
$$
0\to Q_2\stackrel{i}{\to} E_1\to Q_1\to 0
$$ 
gives a functor $i_*$ from the category $\EXT(Q_0,Q_2)$ 
of extensions of $Q_0$ by $Q_2$ to $\EXT(Q_0,E_1)$.  Then using
Baer sum in $\EXT(Q_0,E_1)$ we can define $G\wedge X$ as $i_*G + E_1$.
This gives an action of $\EXT(Q_0,Q_2)$ on $\EXTpan(E_0,E_1)$. 

Alternatively we can define $G\wedge X$ as follows.  Suppose $G$
is given by
$$
0{\to} Q_2 \stackrel{\iota}{\to} G \stackrel{\pi}\to Q_0\to 0.
$$
Then $G\wedge X$ is the cohomology of the 
complex
\begin{equation}\label{tcomplex}
Q_2\to G\oplus X \to Q_0
\end{equation}
where the first map is the one sending $q$ to $(\iota q,-p_2^{-1}(q))$
and the second map sends $(g,x)$ to $\pi(g)-p_0(x)$.   

If we give $G$ the filtration $W_i G$ where $W_{-3} G=0,
W_{-2} G = W_{-1} G = Q_{2}$ and $W_0 G = G$, then
this induces a filtration on $G\oplus X$ making the
morphisms in the complex \eqref{tcomplex} strict (for the obvious
filtrations on $Q_0$ and $Q_2$).   We then get a filtration
on $G\wedge X$ making $G\wedge X$ into an object in $\EXTpan(E_0,E_1)$.
The isomorphism $\pi_0: G\wedge X/W_{-2} (G\wedge X)\stackrel{\sim}{\to} E_0$ is induced from
$\pi_0:X/W_{-2} X\stackrel{\sim}{\to} E_0$.  It sends a pair $(g,x)$ to $\pi_0 x$. 

A class in $W_{-1} (G\wedge X)$ can be represented by an
element $(g,x)\in W_{-1} G\oplus W_{-1} X$.  Since $g\in W_{-1} G$, we can
consider $g$ as an element in $Q_2\subset E_1$.  Then 
$\pi_{1}:W_{-1} (G\wedge X)\stackrel{\sim}{\to} E_1$ sends  $(g,x)$ 
to $g+\pi_1(x)$.

It is easy to check that the action $(G,X)\mapsto G\wedge X$ described above
gives a functorial action; that is, a functor
$\EXT(Q_0,Q_2)\times\EXTPAN(E_0,E_1)\to \EXTPAN(E_0,E_1)$.  
The same prescription gives a functor
$\EXT(Q_0,Q_2)\times \EXTPAN(Q_0,Q_1,Q_2)\to
\EXTPAN(Q_0,Q_1,Q_2)$.

\begin{proposition}[Grothendieck]\label{prop:Grothendieck}
  Suppose $Q_i$ are as above for $i=0,1,2$ and $E_i$ are as above
  for $i=0,1$.  
\begin{enumerate}
\item $\EXTpan(E_0,E_1)$ is a groupoid; that is, every morphism in 
$\EXTpan(E_0,E_1)$ is an isomorphism.
\item The set $\Extpan(E_0,E_1)$ of isomorphism classes of objects in
  $\EXTpan(E_0,E_1)$ is a pseudo-torsor for $\Ext(Q_0,Q_2)$.  In other words,
  $\Extpan(E_0,E_1)$ is either empty or a torsor for $\Ext(Q_0,Q_2)$ under the 
action $(G,X)\mapsto G\wedge X$.  
\item For $E\in\EXT(Q_0,Q_1)$ and $F\in\EXT(Q_1,Q_2)$ as above.  Consider the 
long exact sequence
$$
\Ext^1(Q_0,Q_2)\to \Ext^1(E,Q_2)\to \Ext^1(Q_1,Q_2)\stackrel{\partial}{\to} \Ext^2(Q_0,Q_2)
$$
arising from the short exact sequence 
$$
0\to Q_1\to E\to Q_0\to 0.
$$
Write $\xi\in\Ext^1(Q_1,Q_2)$ for the class of $F$ and define
$c(E,F)=\partial(\xi)$. Then $c(E,F)=0$ if and only if $\Extpan(E,F)$
is non-empty.
\end{enumerate}
\end{proposition}

\subsection{}
There is an obvious map
$\Extpan(Q_0,Q_1,Q_2)\to \Ext^1(Q_0,Q_1)\times \Ext^1(Q_1,Q_2)$ and if
$(E_0,E_1)\in \Ext^1(Q_0,Q_1)\times \Ext^1(Q_1,Q_2)$ we get a commutative
diagram
$$
\xymatrix{
  \Extpan(E_0,E_1)\ar[r]\ar[d]^{\phi}         & \{(E_0,E_1)\}\ar[d]\\
  \Extpan(Q_0,Q_1,Q_2)\ar[r]^-{\pi}     & \Ext^1(Q_0,Q_1)\times\Ext^1(Q_1,Q_2).   
}
$$
We also get an action of $\Hom(Q_0,Q_1)\times\Hom(Q_1,Q_2)$
on the $\Extpan(E_0,E_1)$
as follows.  Take $X\in\Extpan(E_0,E_1)$
and $(f_0,f_1)\in\Hom(Q_0,Q_1)\times \Hom(Q_1,Q_2)$.
Then
\begin{equation}\label{Gd}
  G(f_0,f_1):=f_1\cup [E_0] + [E_1]\cup f_0\in \Ext^1(Q_0,Q_2).
  \end{equation}
Since $\Extpan(E_0,E_1)$ is a $\Ext^1(Q_0,Q_2)$-torsor, $G(f_0,f_1)\wedge X$ is
and element of $\Extpan(E_0,E_1)$.

\begin{proposition}\label{PhiSurj}
  The map $\phi$ surjects onto $\pi^{-1}(E_0,E_1)$.
  Moreover, if $X,Y\in \Extpan(E_0,E_1)$, then
  $\phi(X)=\phi(Y)$ if and only if $Y=G(f_0,f_1)\wedge X$ for some
  $(f_0,f_1)\in \Hom(Q_0,Q_1)\times \Hom(Q_1,Q_2)$. 
\end{proposition}

\begin{proof}
  The first statement is obvious.   For the second,
  suppose $X\in \EXTpan(E_0,E_1)$
  and $Y=G(f_0,0)\wedge X$.  The extension $G:=G(f_0,0)$ is given by
  the following pullback diagram
  $$
\xymatrix{
    0\ar[r] & Q_2\ar[r]       & E_1\ar[r]       & Q_1\ar[r] & 0\\
    0\ar[r] & Q_2\ar[r]\ar@{=}[u] & G\ar[r]\ar[u]^{h} & Q_0\ar[r]\ar[u]^{f_0} & 0}
  $$
  of the extension $E_1$ by the map $f_0$.  We get a map
  $\psi:G\wedge X\to X$ sending a class represented by $(g,x)$
  to $x-h(g)$. This map is well-defined on $G\wedge X$ because, for
  $q\in Q_2$, $\psi(q,-q)=q-h(q)=0$.  The map $\psi$ does not, in
  general commute with the maps $\pi_0$ to $E_0$: we have
  $\pi_0(\psi(g,x))=\pi_0(x-h(g))$ while $\pi_0(g,x)=\pi_0(x)$.
  However, using the fact that $h(g)\in W_{-1} X$, we see that $\psi$
  does commute with $p_0$.  Then it is easy to see that $\psi$
  commutes with $p_1$ and $p_2$.  Thus $\psi$ induces a morphism in
  $\EXTPAN(Q_0,Q_1,Q_2)$ which is easily seen to be an isomorphism.
  So $G\wedge X$ is isomorphic to $X$ in
  $\EXTPAN(Q_0,Q_1,Q_2)$.

  The proof that $G(0,f_1)\wedge X$ is isomorphic to $X$ in
  $\EXTPAN(Q_0,Q_1,Q_2)$ is similar.  This proves the ``if'' part of
  the proposition.

  For the ``only if'' part, suppose that $X$ and $Y$ are objects in
  $\EXTPAN(E_0,E_1)$ and $\psi:Y\to X$ is an isomorphism between them
  as objects in $\EXTPAN(Q_0,Q_1,Q_2)$.  Using $\pi_i$, identify
  $W_{-i}X/W_{-i-2} X$ and $W_{-i}Y/W_{-i-2}Y$ with $E_i$ for $i=0,1$.  Then there exist morphisms
  $f_0:Q_0\to Q_1$ and $f_1:Q_1\to Q_2$ such that
  $(W_{-i}/W_{-i-2})(f)=\id_{E_i} + f_i$.   From this it is not hard to see
  that $Y\cong G(f_0,f_1)\wedge X$.  
\end{proof}

\begin{corollary}\label{DisjTors}
  Suppose that $\Hom(Q_0,Q_1)=\Hom(Q_1,Q_2)=0$.  Then the fiber
  of the map $$\Extpan(Q_0,Q_1,Q_2)\to \Ext^1(Q_0,Q_1)\times \Ext^1(Q_1,Q_2)$$
  over $(E_0,E_1)$ is $\Extpan(E_0,E_1)$.  
\end{corollary}
\begin{proof}
  Obvious.  
\end{proof}

\subsection{Pseudo-biextensions.}\label{psbi}
Grothendieck's concept of mixed extension interacts in a rather interesting
way with the Mumford-Grothendieck concept of a biextension~\cite[pp. 156--159]{SGA71}.  To explain this
fix objects $Q_0,Q_1$ and $Q_2$ as above.

Suppose $E_1,E_2\in\Ext^1(Q_0,Q_1)$ and $F_1,F_2\in \Ext^1(Q_1,Q_2)$.  
There are natural composition laws
\begin{align*}
  +_1:\Extpan(E_1,F_1)\times \Extpan(E_2,F_1)\to \Extpan(E_1+E_2,F_1)\\
  +_2:\Extpan(E_1,F_1)\times\Extpan(E_1,F_2)\to \Extpan(E_1,F_1+F_2)
\end{align*}
satisfying various compatibilities.   
This is explained in C.~Hardouin's 2005 thesis~\cite[Theorem 4.4.1]{Hardouin}.
For the convenience of the reader we explain the operations  
$+_1$ and $+_2$.

\subsection{}\label{s.POY}  Suppose $X_i\in\Extpan(E_i,F)$ for $i=1,2$.   Write
$\rho_i:X_i\to Q_0$ for the composition
$$X_i\to\Gr^W_0 X_i\stackrel{p_0}{\to} Q_0$$
and write $j_i:F\to X_i$
for the canonical injection induced by $\pi_1^{-1}$.  Then Hardouin defines $X_1+_1 X_2$
to be the middle homology group of the complex
$$
\xymatrix{
F\ar[r]^(.4){
\begin{pmatrix}
j_1 \\
-j_2 \\
\end{pmatrix}
}   
&
X_1\oplus X_2\ar[r]^*+<1em>{
\begin{pmatrix}
\rho_1 & -\rho_2 \\
\end{pmatrix}} &  Q_0\\
}$$
where we write the arrows in matrix notation.
We then get maps
\begin{align*}
\begin{pmatrix}
j_1 & 0 
\end{pmatrix}:& F\to X_1+_1 X_2,\\
\begin{pmatrix}
  \rho_1 \\ 
      0
\end{pmatrix}:&
X_1+_1 X_2\to Q_0.
\end{align*}
Combining these maps, we obtain an exact sequence
$$
0\to F\to X_1+_1 X_2\to Q_0\to 0.
$$
And it is not hard to see that this sequence makes 
$X_1+_1 X_2$ into an element of $\Extpan(E_1+E_2, F)$ in a canonical way.

\subsection{}\label{s.PO}
Here is another way to think about the composition $+_1$.   An object
$X_i\in\EXTpan(E_i,F)$ can be thought of as an object of $\EXT^1(Q_0,F)$ 
whose image in $\EXT^1(Q_0,Q_1)$ is $E_i$.   In other words, if we let $\pi:\EXT^1(Q_0,F)\to \EXT^1(Q_0,Q_1)$ denote the functor induced by pushforward
along the canonical map $F\to Q_1$, then $\EXTpan(E_i,F)$ is the
fiber category $\pi^{-1}(E_i)$.
If we take the Baer sum of $X_1$ and $X_2$ regarded as objects in $\Ext^1(Q_0,F)$,
then we clearly get an element of $\pi^{-1}(E_1+E_2)$.   It is easy to check
that this element is, in fact, $X_1+_1 X_2$ as defined above.  

\subsection{}  
The construction of $+_2$ is similar, and we explain it in the language
of \S\ref{s.PO}.    Suppose $X_i$ are objects in $\EXTpan(E,F_i)$
for $i=1,2$.  Write $\iota:\EXT^1(E,Q_2)\to \EXT^1(Q_1,Q_2)$ for the canonical 
functor (induced by the inclusion $Q_1\to E$).   Then we can regard $X_i$
as object in the fiber category $\iota^{-1}(F_i)$.  We then define the 
sum $X_1+_2 X_2$ to be the Baer sum of the two $X_i$ in $\EXT^1(E,Q_2)$.
It is easy to see that  $X_1+_2 X_2$ lies in the fiber category 
$\iota^{-1}(F_1+F_2)=\EXTpan(E,F_1+F_2)$.

\begin{definition}\label{d.PBE}
 Suppose $B,C$ and $A$ are three abelian groups.   A \emph{pseudo-biextension}
of $B\times C$ by $A$ is a set $E$ equipped with 
\begin{enumerate}
\item an $A$-action  and a map $\pi:E\to B\times C$ giving  $E$ the structure of an $A$ pseudo-torsor over $B\times C$;
\item commutative operations $+_1:E\times_B E\to E$ and $+_2:E\times_C E\to E$.
\end{enumerate}
Write $p_B=\pr_1\circ\pi$ and $p_C=\pr_2\circ \pi$.   
For each $b\in B$, write $E_b=p_B^{-1}(b)$ and, for each $C$, write
$E_c=p_C^{-1}(c)$. 
Then the above data is assumed to satisfy the following properties
\begin{enumerate}
\setcounter{enumi}{2}
\item For each $b\in B$, the operation $+_1$ makes $E_b$ into an abelian group 
and the canonical map $E_b\to C$ into a group homomorphism.  Moreover,
the action of $A$ on $\ker(E_b\to C)$ coming from the structure
of $E$ as an $A$ pseudo-torsor induces a group isomorphism between $A$
and $\ker(E_b\to C)$. 
\item For each $c\in C$, the operation $+_2$ makes $E_c$ into an abelian group
  and the canonical map $E_c\to B$ into a group homomorphism with kernel
  equal to $A$ as
in (iii).
\item Suppose $\pi(X_{ij})=(E_i,F_j)$ for $i=1,2$.   Then
$$(X_{11}+_1 X_{12})+_2 (X_{21}+_1 X_{22})
= (X_{11}+_2 X_{21})+_1 (X_{12}+_1 X_{22}).
$$ 
\end{enumerate}
\end{definition}

\subsection{} A surjective pseudo-biextension $\pi:E\to B\times C$ is
called a \emph{biextension}.  Grothendieck studied
the notion of biextensions in the category of sheaves of abelian
groups in a topos $T$ in~\cite{SGA71}.    Suppose $B, C$ and $A$ are three sheaves
of abelian groups in $T$.   
 We define a \emph{pseudo-biextension} of $B\times C$ by $A$ 
to be a sheaf $E$ in $T$ equipped with a morphism $\pi:E\to B\times C$
along with an $A$-action making $E$ into an $A$-torsor over $B\times C$ and operations $+_1, +_2$ satisfying the same axioms as in Definition~\ref{d.PBE} above.  A pseudo-biextension $\pi:E\to B\times C$ is a \emph{biextension}
if $\pi$ is surjective as a morphism of sheaves.

\begin{remark}\label{BClass} Suppose $A, B$ and $C$ are abelian groups, and $U$ is
  an extension of $B\otimes C$ by $A$.  Write $E$ for the pull-back of
  $U$ to $B\times C$ via the natural map $B\times C\to B\otimes C$
  given by $(b,c)\mapsto b\otimes c$.  The group $A$ acts on $U$ via
  right translation in such a way that $U$ becomes an $A$-torsor over
  $B\otimes C$.  Therefore, the pull-back $E$ of $U$ becomes an
  $A$-torsor over $B\times C$.  Write $p$ for the map
  $U\to B\otimes C$ and $\pi$ for the map $E\to B\times C$.  Suppose
  $b_i\in B$, $c_i\in C$ for $i=1,2$ and
  $u_{ij}\in p^{-1}(b_i\otimes c_j)$.  Then
  $u_{11}+u_{12}\in p^{-1}(b_1\times C)$.  So we can use the addition
  in $U$ to define a map
$$
+_1: E\times_B E\to E.
$$
We can similarly define $+_2:E\times_C E\to E$.  And it is slightly tedious
but not hard to show that $E$ becomes a biextension of $B\times C$ by $A$.

In fact, Grothendieck showed that all biextensions arise in this way
provided that $\Tor_1(B,C)=0$.  More generally we have an exact
sequence
$$
0\to \Ext^1(B\otimes C, A)\to \Biext(B\times C, A)\to
\Hom(\Tor_1(B\otimes C), A)
$$
where $\Biext(B\times C, A)$ denotes the set of isomorphisms of
biextensions of $B\times C$ by $A$.  
~\cite[p.220]{SGA71}.  
\end{remark}

\begin{theorem}[Hardouin]\label{t.Har}
  Suppose $\mathbf{C}$ is an abelian category and $Q,R,P$ are objects
  satisfying $\Hom(Q,R)=\Hom(R,P)=0$.
  Then the map
$$
\Extpan(Q,R,P)\to \Ext^1(Q,R)\times \Ext^1(R,P)
$$
together with the action of $\Ext^1(Q,P)$ and the operations 
$+_1$ and $+_2$ described above
makes $\Extpan(Q,R,P)$ into a pseudo-biextension of $\Ext^1(Q,R)\times\Ext^1(R,P)$
by $\Ext^1(Q,P)$. 
\end{theorem}

\begin{proof}[Explanation]  The main point is to check that the
  $\Ext^1(Q,P)$ action gives $\Extpan(Q,R,P)$ the structure of a pseudo-torsor.
  This follows from the condition that $\Hom(Q,R)=\Hom(R,P)=0$ and
  Corollary~\ref{DisjTors}.  The rest of the verification is straightforward.
\end{proof}

\subsection{Sheaf theoretic version}
Suppose now that $T$ is a topological space and $\mathcal{C}$ is a
stack of abelian categories over $T$.  So for each open subset $U$ of
$T$ we have an abelian category $\mathcal{C}(U)$ and for each
inclusion $i:V\to U$ we have restriction functor
$i^*:\mathcal{C}(U)\to \mathcal{C}(V)$, which we also write as
$X\leadsto X_{|V}$.  As in~\cite[p.~5396]{Treumann},
we assume that $i^*$ is exact.   The main example we have in mind is when
$T$ is a complex manifold and $\mathcal{C}(U)$ is the category of mixed Hodge
modules on $T$.  

For any objects $A,B\in \mathcal{C}(T)$, we get a sheaf $\sHom(A,B)$ sending on open
$U$ to $\Hom(A_{|U},B_{|U})$.

\begin{lemma}
Suppose that $A,B$ are objects in $\mathcal{C}(T)$ satisfying
$\sHom(A,B)=0$.  
\begin{enumerate}
\item If $E$ is an extension of $A$ by $B$ in $\mathcal{C}(T)$, then the identity is
  the only automorphism of $E$ as an extension.
\item The assignment $U\leadsto \Ext^1(A_{|U},B_{|U})$, which is naturally a
  presheaf by the exactness of the restriction functors, is a sheaf.
\end{enumerate}
\end{lemma}

\begin{proof}
  (i) If $\phi$ is an automorphism then $\phi-\id$ induces an element of $\Hom(A,B)$
  which must be $0$.

  (ii) 
  Suppose $\{V_i\}$ is an open cover of $U$ and, for each $i$, $E_i\in \Ext^1(A_{|V_i},B_{|V_i})$
  such that, for all $i,j$, $(E_i)_{|V_i\cap V_j}\cong (E_j)_{|V_i\cap V_j}$.  It follows from (i)
  that there is a \emph{unique} isomorphism $\varphi_{ij}:(E_i)_{|V_i\cap V_j}\cong (E_j)_{|V_i\cap V_j}$.
  In particular, $\varphi_{jk}\circ\varphi_{ij}=\varphi_{ik}$ for all $i,j,k$.  Therefore the
  $E_i$ glue together to form an object $E\in\Ext^1(A_{|U},B_{|U})$.  
\end{proof}

We write $\sExt^1(A,B)$  for the resulting sheaf when $\sHom(A,B)=0$.  

\begin{definition}
  We call a triple $Q_{\bullet}=(Q_0,Q_1,Q_2)$ \emph{disjoint} if $\sHom(Q_0,Q_1)=\sHom(Q_0,Q_2)=
  \sHom(Q_1,Q_2)=0$. 
\end{definition}

The main example we have in mind here is a sequence $Q_i$ of
torsion-free variations of pure Hodge structure of weight $-i$.

If $Q_{\bullet}$ is disjoint then the same argument as above shows
that the presheaf $$U\leadsto \Extpan(Q_0|U, Q_1|U, Q_2|U)$$ is a sheaf.
Similarly, for $E\in\Ext^1(Q_0,Q_1)$ and $F\in\Ext^1(Q_1,Q_2)$, the presheaf
$U\leadsto \Extpan(E_{|U},F_{|U})$ is a sheaf.  We write $\sExtpan(Q_0,Q_1,Q_2)$
for the first sheaf and $\sExtpan(E,F)$ for the second.  By Grothendieck's result,
$\sExtpan(E,F)$ is a $\sExt^1(Q_0,Q_2)$ pseudo-torsor.   Moreover,
$\sExtpan(E,F)$ is a $\sExt^1(Q_0,Q_2)$ torsor if the stalks of $\sExtpan(E,F)$
are non-empty.

Using Theorem~\ref{t.Har} we see
that 
$$
\sExtpan(Q_0,Q_1,Q_2)\to \sExt^1(Q_0,Q_1)\times\sExt^1(Q_1,Q_2)
$$
together with the action of $\sExt^1(Q_0,Q_2)$ and the operations $+_1$ and $+_2$
(which become sheaf morphisms) is a pseudo-biextension of the sheaf
$\sExt^1(Q_0,Q_1)\times \sExt^1(Q_1,Q_2)$ by $\sExt^1(Q_0,Q_2)$.
It is a biextension if the stalks of $\sExtpan(Q_0,Q_1,Q_2)$
are non-empty.

\section{Mixed Extensions of normal functions}\label{me}

\subsection{The sheaf of mixed extensions}\label{me1} 
Suppose $j:S\to\bar S$ is the inclusion of one complex manifold in
another as a Zariski open set.  By analogy with \S\ref{s1.1}, for an $A$-variation
of mixed Hodge structure $\mathcal{H}$
on $S$
with $W_{-1}\mathcal{H}=\mathcal{H}$, 
write $\NF(S,\mathcal{H})$
for $\Ext_{\VMHS(S)}^1(A,\mathcal{H})$ and $\ANF(S,\mathcal{H})_{\bar S}$
for $\displaystyle\Ext^1_{\VMHS(S)^{\ad}_{\bar S}}(A,\mathcal{H})$.
If $\mathcal{H}$ is torsion free and pure of weight $-1$, duality
gives an isomorphism $\NF(S,\mathcal{H}^*(1))=
\Ext^1_{\VMHS(S)}(\mathcal{H},A)$ and sending $\ANF(S,\mathcal{H}^*(1))$
to $\Ext^1_{\VMHS(S)^{\ad}_{\bar S}}(\mathcal{H},A)$.

Suppose now that $\mathcal{H}$ is pure of weight $-1$ and torsion-free.
Since the triple $(A, \mathcal{H}, A(1))$ is disjoint, we get sheaves on $S$:
$U\leadsto \NF(U,\mathcal{H}_{|U}), U\leadsto \NF(U,\mathcal{H}^{\vee}_{|U})$
and $U\leadsto \NF(U,A(1))$.   We also get a sheaf
$U\leadsto \Extpan(A,\mathcal{H}_{|U},A(1))$.   Write $\sNF(\mathcal{H})$
and $\sNF(A(1))$ for these sheaves on $S$.   We also get sheaves $\sANF(\mathcal{H})$ and $\sANF(A(1))$ on $\bar S$ given by admissible normal functions.

Fix
$\nu\in\ANF(S,\mathcal{H})_{\bar S}$ and
$\omega\in\ANF(S,\mathcal{H}^{\vee})_{\bar S}$.
Then, for any open set
$U\subset S$, we write $\calB_A(U)$ for the set of isomorphism
classes of mixed extensions of $A$ by $\mathcal{H}_A$ by $A(1)$
in $A-\VMHS(U)$.  Similarly, if $U\subset\bar S$, we write $\calB_A^{\ad}(U)$ 
for the set of isomorphism classes of mixed extensions of $A$
by $H_A$ by $A(1)$ in $A-\VMHS(U\cap S)^{\ad}_{U}$.  
The normal functions $\nu$ and $\omega$ give rise to extensions
$\displaystyle\nu_A\in   \Ext^1_{A-\VMHS(U\cap S)^{\ad}_{U}}(A,\mathcal{H}_A)$ and
$\displaystyle\omega_A\in\Ext^1_{A-\VMHS(U\cap S)^{\ad}_{U}}(\mathcal{H}_A^{\vee}, A(1))$.  
We write $\calB_A(\nu,\omega)(U)$ (resp. $\calB_A^{\ad}(\nu,\omega)(U)$)
for set of isomorphism classes of mixed extensions of
$\nu_A$ by $\omega_A$ in $A-\VMHS(U)$ (resp. $A-\VMHS(U\cap S)^{\ad}_{U}$).
When $A=\mathbb{Z}$ we drop the subscript $A$ from the notation and
simply write $\calB$ (resp. $\calB^{\ad}$).

By pullback of variations, $\calB_A$ (resp. $\calB^{\ad}_A$) is 
presheaf on $S$ (resp. $\bar S$).

\begin{lemma}
  The presheaf $\calB_A$ (resp. $\calB^{\ad}_A$) is a sheaf on $S$
  (resp. $\bar S$).  
\end{lemma}
\begin{proof}  This follows from the fact that $A, \mathcal{H}$ and
  $A(1)$ are disjoint in the category of $A$ variations.
\end{proof}

\begin{lemma}\label{me2}
 Write $j_*^{\mer}\mathcal{O}_S^{\times}$ for the sheaf of meromorphic
functions on $\bar S$ which are regular and non-vanishing on $S$.  Then 
\begin{enumerate}
\item The functor 
$$U\leadsto \Ext^1_{\VMHS(U)}(\mathbb{Z},\mathbb{Z}(1))=\NF(U,\mathbb{Z}(1))$$
defines a sheaf on $S$ which is canonically isomorphic to 
$\mathcal{O}_S^{\times}$. 
\item The functor 
$$
  U\leadsto \Ext^1_{\VMHS(U\cap S)^{\ad}_{U}} (\mathbb{Z},\mathbb{Z}(1))=\ANF(U\cap S,\mathbb{Z}(1))_{U}
$$ 
defines a sheaf on $\bar S$ which is canonically isomorphic to 
$j_*^{\mer}\mathcal{O}_S^{\times}$.  
\end{enumerate}
\end{lemma}

\begin{proof}
  Both of these statements are local.  The first is well known and
follows essentially from the canonical isomorphism 
$\Ext^1_{\MHS}(\mathbb{Z},\mathbb{Z}(1))=\mathbb{C}^{\times}$.  

In the case that $\bar S=\Delta^{a+b}$ and $S=\Delta^{*a}\times \Delta^b$
for non-negative integers $a$ and $b$, the second statement follows
from the local normal form of an admissible normal function.  
Since the statement is local on $\bar S$, this proves that (ii) holds
when $Y:=\bar S\setminus S$ is a normal crossing divisor. 

For the general case, by using Hironaka, we can find a proper morphism
$p:\bar S'\to \bar S$ such that $p$ is an isomorphism over $S$ and 
$D=p^{-1} Y$ is a normal crossing divisor (with $Y=\bar S\setminus S$ as above). 
Write $S'=p^{-1} S$, and $j':S'\to\bar S'$ for the inclusion. 
Then, if $U$ is open in $\bar S$, by M.~Saito we have  
$\ANF(U\cap S,\mathbb{Z}(1))_{U}=\ANF(p^{-1}U\cap S',\mathbb{Z}(1))_{p^{-1}U}$.     
Since $j_*^{\mer}\mathcal{O}_S^{\times}=p_*j_*^{'\mer}\mathcal{O}_{S'}^{\times}$,
this proves that (ii) holds in general.
\end{proof}

\begin{corollary}\label{me3}
The sheaf $\calB$ (resp. $\calB^{\ad}$)  is a biextension (resp. pseudo-biextension) 
of $\sNF(S,\mathcal{H})\times\sNF(S,\mathcal{H}^{\vee})$ by  $\mathcal{O}_S^{\times}$ (resp. of  
$\sANF(S,\mathcal{H})_{\bar S}\times\sANF(S,\mathcal{H}^{\vee})_{\bar S}$ 
by  $j_*^{\mer}\mathcal{O}_S^{\times}$).   
\end{corollary}

\begin{proof}
  The fact that $\calB$ and $\calB^{\ad}$
  are pseudo-biextensions follows directly from Lemma~\ref{me2} and Theorem~\ref{t.Har}. The fact that $\calB$ is a biextension
follows from Theorem~\ref{sec:mix-1} (applied to the case where $D=\emptyset$).
\end{proof}

\begin{remark}
  If $\bar S\setminus S$ is a normal crossing divisor, then
  Theorem~\ref{sec:mix-1} implies that $\calB^{\ad}$  in Corollary~\ref{me3} is a biextension. We will prove a stronger result  in Theorem~\ref{eb1}.
\end{remark}

\begin{corollary}\label{torsor}
  The sheaf $\calB(\nu,\omega)$
  (resp. $\calB^{\ad}(\nu,\omega)$)
  is a torsor (resp. pseudo-torsor) for $\mathcal{O}_S^{\times}$
  (resp. $j_*^{\mer}\mathcal{O}_S^{\times}$).
\end{corollary}

%

\section{The torsion pairing}\label{tp}
\newcommand\tors{\mathrm{tors}}
Suppose $\calL$ is a torsion-free local system of $\mathbb{Z}$ modules
on $\Delta^*$.  Write $\rH^1(\calL)_{\tors}$ for the torsion elements of
$\rH^1(\calL):=\rH^1(\Delta^*,\calL)$.  Note that $\rH^1(\calL^*)$ is
canonically isomorphic to
$\Ext^1_{\Delta^*}(\calL,\mathbb{Z})$ via the map taking an extension to
its dual.  Let $M$ denote the category of mixed extensions of $\mathbb{Z}$
by $\calL$ by $\mathbb{Z}$ (in the category of sheaves of abelian groups
over $\Delta^*$).  
We call a mixed extension
$X$ in $M$ \emph{restricted} if $X/W_{-2} X\in \rH^1(\calL)_{\tors}$ and
$W_{-1}X\in \rH^1(\calL^*)_{\tors}$.   Write $\mathcal{R}$ for the set of
isomorphism classes of such
mixed extensions.    It comes equipped with an obvious map
$\mathcal{R}\to \rH^1(\calL)_{\tors}\times \rH^1(\calL^*)_{\tors}$.

\begin{proposition}
  We have an action of $\mathbb{Z}=\Ext^1_{\Delta^*}(\mathbb{Z},\mathbb{Z})$
  on $\mathcal{R}$ along with operations $+_1$ and $+_2$ on $R$.  These
  make $\mathcal{R}$ into a biextension of $\rH^1(\calL)_{\tors}\times \rH^1(\calL^*)_{\tors}$
  by $\mathbb{Z}$. 
\end{proposition}

\begin{proof}
  Suppose $E_0$ is an extension of $\mathbb{Z}$ by $\mathcal{L}$
  and $E_1$ is an extension of $\mathcal{L}$ by $\mathbb{Z}$.
  If $E_0$ and $E_1$ have torsion cohomology classes, then, for $(f_0, f_1)\in
  \Hom(\mathbb{Z},\calL)\times \Hom(\calL,\mathbb{Z})$, the
  class $G(f_0,f_1)\in \Ext^1_{\Delta^*}(\mathbb{Z},\mathbb{Z})=
  \rH^1(\Delta^*,\mathbb{Z})$ (given by~\eqref{Gd}) is trivial.  It follows by
  Proposition~\ref{PhiSurj} that $\Extpan(E_0,E_1)$ injects into
  $\Extpan(\mathbb{Z},\calL,\mathbb{Z})$ (where the $\Extpan$ sets
  are taken with respect to the category of sheaves of abelian
  groups on $\Delta^*$).   The set $\Extpan(E_0,E_1)$ is an
  $\Ext^1_{\Delta^*}(\mathbb{Z},\mathbb{Z})$-torsor.   In other words,
  it is a $\mathbb{Z}$-torsor.  Consequently, $\mathcal{R}$ has an action
  of $\mathbb{Z}$ making it into a $\mathbb{Z}$-torsor over
  $\rH^1(\calL)_{\tors}\times \rH^1(\calL^*)_{\tors}$.  

  The operations $+_1$ and $+_2$ are defined as in \S\ref{psbi}.
  The rest of the verification is left to the reader.
\end{proof}

By the results in~\cite{SGA71} summarized in Remark~\ref{BClass}, biextensions of 
$\rH^1(\calL)_{\tors}\otimes\rH^1(\calL^*)_{\tors}$ by $\mathbb{Z}$ are classified by the group
$\Ext^1(\rH^1(\calL)_{\tors}\otimes\rH^1(\calL^*)_{\tors},\mathbb{Z})$. 
This group sits in an exact sequence
\begin{align*}
  \Hom(\rH^1(\calL)_{\tors}\otimes\rH^1(\calL^*)_{\tors},\mathbb{Q})
&\to 
\Hom(\rH^1(\calL)_{\tors}\otimes\rH^1(\calL^*)_{\tors},\mathbb{Q}/\mathbb{Z})\\
\to\Ext^1(\rH^1(\calL)_{\tors}\otimes\rH^1(\calL^*)_{\tors},\mathbb{Z})
&\to \Ext(\rH^1(\calL)_{\tors}\otimes\rH^1(\calL^*)_{\tors},\mathbb{Q}).
\end{align*}
Since the first and last groups are $0$, we have an isomorphism
between the second and third groups.  Thus, the biextension gives rise
to a bilinear pairing
$$
\tau:\rH^1(\calL)_{\tors}\otimes\rH^1(\calL^*)_{\tors}\to \mathbb{Q}/\mathbb{Z}
$$
which we call the \emph{Grothendieck torsion pairing} or just the
\emph{torsion pairing}. 

We want to compute the torsion pairing in an explicit way.  To do
this, write $\mathcal{R}_{\mathbb{Q}}$ for the set of isomorphism
  classes of mixed extensions $X$ of $\mathbb{Q}$ by
  $\calL_{\mathbb{Q}}$ by $\mathbb{Q}$ which are restricted in the
  sense that $X/W_{-2}$ and $W_{-1} X$ are both $0$ (in
  $\rH^1(\calL_{\mathbb{Q}})$ and $\rH^1(\calL^*_{\mathbb{Q}})$
  respectively).  The same argument as above shows that
  $\mathcal{R}_{\mathbb{Q}}$ has the structure of a biextension, but
  this time it is a biextension over the trivial group $0$.
  Therefore, there is a canonical isomorphism
  \begin{equation}\label{ciso}
    \mathcal{R}_{\mathbb{Q}}=\Ext^1_{\Delta^*}(\mathbb{Q},\mathbb{Q})=\mathbb{Q}.
    \end{equation}
  Moreover, tensoring with $\mathbb{Q}$ gives a morphism of biextensions
  $X\leadsto X_{\mathbb{Q}}$ from $\mathcal{R}$ to $\mathcal{R}_{\mathbb{Q}}$.
  (See~\cite[p.~162]{SGA71} for the notion of a morphism of biextensions.)
  So, for $X\in\mathcal{R}$, we get a rational number $\tilde\tau X$
  given by the image of $X_{\mathbb{Q}}$ under \eqref{ciso}.  

  \begin{definition}
    Write
    $\mathcal{R}'$ for the set of triples $(\alpha,\beta,\gamma)\in
    \rH^1(\calL_{\tors})\times \rH^1(\calL^*)_{\tors}\times\mathbb{Q}$
    such that there exists $X\in R$ with $\pi X=(\alpha,\beta), \tilde\tau X=\gamma$.
  \end{definition}

  Since $X\to X_{\mathbb{Q}}$ is a morphism of biextensions from $\mathcal{R}$
  to $\mathcal{R}_{\mathbb{Q}}$, we have, for $i=1,2$,
  $$\tilde\tau(X+_i X')=\tilde\tau(X)+\tilde\tau(X')$$ whenever
  $+_i$ is defined.   Moreover, for $n\in\mathbb{Z}$, $\tilde\tau(n+X)=
  n+\tilde\tau X$.  
  
From this, it is not hard to see that the group $\mathbb{Z}$ acts on
  $\mathcal{R}'$ by the rule
  $n+(\alpha,\beta,\gamma)= (\alpha,\beta, n+\gamma)$.  Moreover, by setting
\begin{align*}
      (\alpha,\beta,\gamma)+_1 (\alpha',\beta,\gamma')&=(\alpha+\alpha',\beta,\gamma+\gamma') \\
  (\alpha,\beta,\gamma)+_2 (\alpha,\beta',\gamma')&=(\alpha,\beta+\beta',\gamma+\gamma')
\end{align*}
we get the structure of a biextension on $\mathcal{R}'$.  In fact, we get an isomorphism
of biextensions
\begin{equation}\phi:\mathcal{R}\to\mathcal{R}'\label{isobiext}
\end{equation}
given by $X\mapsto (X/W_{-2} X, W_{-1}X,\tilde\tau X)$.
All this leads to the following proposition. 

\begin{proposition}
Suppose $(\alpha,\beta)=\pi X$ for $X\in\mathcal{R}$.  Then $\tau(\alpha\otimes\beta)= \tilde\tau X \pmod{\mathbb{Z}}$. 
\end{proposition}

\begin{proof}
  Write $t$ for the reduction of $\tilde\tau$ modulo $\mathbb{Z}$.
  Then $t$ gives a map
  \begin{equation}\label{tmap}
  t:\rH^1(\calL)_{\tors}\otimes \rH^1(\calL^*)_{\tors}\to\mathbb{Q}/\mathbb{Z},
\end{equation}
and what we have to prove is that $\tau=t$.
Using $t$ and the isomorphism
\begin{equation}\Hom(\rH^1(\calL)_{\tors}\otimes\rH^1(\calL^*)_{\tors}, \mathbb{Q}/\mathbb{Z})\to \Ext^1(\rH^1(\calL)_{\tors}\otimes\rH^1(\calL^*)_{\tors}, \mathbb{Z}),
\end{equation}
we see that $t$ corresponds to the extension of
$\rH^1(\calL)_{\tors}\otimes\rH^1(\calL^*)_{\tors}$ by $\mathbb{Z}$
whose fiber over $\alpha\otimes\beta$ is the set of triples
$(\alpha,\beta,\gamma)$ with $\gamma\in t(\alpha\otimes\beta)$ (where
we view $t(\alpha\otimes\beta)$ as a coset of $\mathbb{Z}$ in $\mathbb{Q}$).
But this is exactly the fiber of $\mathcal{R}'$ over $\alpha\times\beta$.
And this shows that the $\mathcal{R}'$ is the image of $t$ under the
map
\begin{equation}
  \label{etob}
  \Ext^1(\rH^1(\calL)_{\tors}\otimes\rH^1(\calL^*)_{\tors},\mathbb{Z})
  \to\Biext(\rH^1(\calL)_{\tors}\times\rH^1(\calL^*)_{\tors},\mathbb{Z}).
\end{equation}
Since $\mathcal{R}$ is isomorphic to $\mathcal{R}'$, this shows that $\tau=t$.
\end{proof}

To write the torsion pairing explicitly, let $L$ denote the fiber of $\calL$
at a chosen point $s_0\in\Delta^*$ and write $T\in\Aut L$ for the monodromy operator.   Then
$\rH^1(\calL)$ is computed by the Koszul complex $K_{\mathbb{Z}}(L)$
given by
$$
L\stackrel{T-1}{\to} L
$$
in degrees $0$ and $1$.   So that $H^0\calL=L^{T}$ and $H^1\calL=
L/(T-1)L$.  
We can make the computation of $H^1\calL$ explicit if we view and element
of $H^1\calL$ as an extension
\begin{equation}
\label{e.extz}
0\to L\to E\stackrel{p}{\to} \mathbb{Z}\to 0
\end{equation}
in the category of abelian groups equipped with a $\mathbb{Z}$-action 
(where the action on $\mathbb{Z}$ on itself is taken to be trivial).

If we pick an element $e\in E$ such that $p(e)=1$, then $(T-1)e\in L$.  
We can change $e$ to $e':=e+l$ for some $l\in L$.  Then 
$(T-1)e'=(T-1)e+ (T-1)l$.   Therefore the class $[(T-1)e]\in L/(T-1)L$
is well defined.  And it is easy to see that this identifies $H^1(\Delta^*,\calL)$ with $L/(T-1)L$.   

\begin{corollary}\label{c.torsion}
 Under the above identification, we have  
$$H^1(\calL)_{\tors}=\frac{L\cap (T-1)L_{\mathbb{Q}}}{(T-1)L}.$$
\end{corollary}
\begin{proof}
  Obvious.
\end{proof}

Suppose $X\in\mathcal{R}$ with $\pi(X)=(\alpha,\beta)\in
\rH^1(\calL)_{\tors}\times \rH^1(\calL^*)_{\tors}$.  Write $V$
for the fiber of $X$ at $s_0$, and write $\tilde{T}\in\Aut V$
for the monodromy action.   This action preserves the filtration
$W$ on $V$.

\begin{proposition}\label{ttauX}
  Suppose $X\in\mathcal{R}$ has monodromy matrix $\tilde{T}$.
  There exists an $e_0\in V_{\mathbb{Q}}$ with the following properties
  \begin{enumerate}
  \item the projection of $e_0$ to $\Gr^WV_{\mathbb{Q}}$ is
    equal to $1$ under the identification
    $p_0:\Gr^W_{\mathbb{Q}}\stackrel{\sim}{\to}\mathbb{Q}$;
  \item $(\tilde{T}-1)e_0\in W_{-2} V_{\mathbb{Q}}$. 
  \end{enumerate}
  For any such element $e_0$, we have $\tilde\tau(X)=p_{-2}((\tilde T-1)(e_0))$.
\end{proposition}

\begin{proof}
   When tensored with $\mathbb{Q}$, the extensions
  $X/W_{-2}X$ and $W_{-1}X$ in $\rH^1(\calL)$ respectively $\rH^1(\calL^*)$
  become trivial.  The existence of $e_0\in V_{\mathbb{Q}}$ follows from
  the triviality of $X_{\mathbb{Q}}/W_{-2} X_{\mathbb{Q}}$.  Then the equality
  $\tilde\tau X=p_{-2}((\tilde T-1)(e_0))$ follows from the definition of $\tilde\tau$
  as the image of $X_{\mathbb{Q}}$ in $\Ext^1_{\Delta^*}(\mathbb{Q},\mathbb{Q})=\mathbb{Q}$.
\end{proof}

\begin{proposition}
Suppose $\alpha\in L\cap (T-1)L_{\mathbb{Q}}$ represents a class
in $H^1(\calL)_{\tors}$ and  $\beta\in L^*$ represents a class in 
$H^1(\calL^*)_{\tors}$.  Pick $\ell\in L_{\mathbb{Q}}$ such that $\alpha=(T-1)l$.
Then 
$$
\tau([\alpha],[\beta])=-(\ell,\beta)\mod{\mathbb{Z}}.
$$
\end{proposition}
\begin{proof}
  We can represent a mixed extension $X\in p^{-1}(\alpha,\beta)$ by giving  
the monodromy with respect to a basis consisting of 
\begin{enumerate}
\item an element $e_0\in X$ projecting to $1$ under the isomorphism
$\Gr^W_0 X=\mathbb{Z}$,
\item elements in $W_{-1}X$ lifting a basis of $L$ under the map 
$\Gr^W_{-1} X=L$, 
\item the generator $e_{-2}$ of $W_{-2}X=\mathbb{Z}$. 
\end{enumerate}
In matrix form, we then have
$$
\tilde{T}=
\begin{pmatrix}
  1      & 0     & 0\\ 
 \alpha  & T     & 0\\
    m    & \beta & 1
\end{pmatrix}
$$
for some $m\in\mathbb{Z}$.
Over $\mathbb{Q}$ we can change $e_0$ to $e_0'=e_0-\ell$.  Then 
$\tilde{T}e_0'=e_0+\alpha +me_{-2} -T\ell - (\ell,\beta)e_{-2}=
e'_0+\alpha+me_{-2}-(T-1)\ell-(\ell,\beta)e_{-2}=e_0'+(m-(\ell,\beta))e_{-2}$.  
So if we change the basis by changing $e_0$ to $e_0'$ the matrix
for $\tilde{T}$ becomes
$$
\begin{pmatrix}
  1              & 0     & 0\\ 
  0              & T     & 0\\
  m-(\ell,\beta)  & \beta & 1
\end{pmatrix}.
$$
So $\tilde\tau(X)=m-(\ell,\beta)$, and $\tau(X)=-(\ell,\beta)$ as desired.
\end{proof}

Suppose now that $\calH$
is a torsion free weight $-1$
variation of Hodge structure on $\Delta^*$.
Fix admissible normal functions
$\nu\in\ANF(\Delta^*,\calH)_{\Delta}$
and $\omega\in\ANF(\Delta^*,\calH^{\vee})_{\Delta}$.
A biextension variation of Hodge structure
$\calV\in\calB^{\ad}(\nu,\omega)$
then gives a mixed extension in the category of admissible variations.
The underlying local system $\calV_{\mathbb{Z}}$ is then a restricted
mixed extension in $\mathcal{R}$.   

\begin{corollary} Suppose $\calH$ has unipotent monodromy.
 Then $\mu(\calV)=\tilde\tau\calV_{\mathbb{Z}}$.   
\end{corollary}
\begin{proof}
This follows from Proposition~\ref{ttauX}.
\end{proof}

\section{Meromorphic Extensions}\label{sec:merext}

The goal of this section is to show that the biextension line bundle
$\mathcal{L}(\nu,\mu)$ on $S$ has a canonical extension to a
meromorphic line bundle on $\bar S$.  Essentially this is a
consequence of Corollary~\ref{torsor} above, which shows that the
sheaf of biextensions is a $j_*^{\mer}\mathcal{O_S^{\times}}$-torsor.
(We remind the reader that the main work going into that Corollary was
done in \S\ref{sec:mix}.)  What remains to do is to recall the
definition of a meromorphic extension, which we take from Deligne's
book on differential equations~\cite[p.~65]{deligne-diffeq}, and to
show that $j_*^{\mer}\mathcal{O}_S^{\times}$-torsors are in one-one
correspondence with meromorphic line bundles.   Deligne considers meromorphic
extensions of coherent sheaves, and we have followed this but, in order
to understand the category of meromorphic sheaves better, we have rephrased his
definition in the language of stacks. (This gives us a category of meromorphic sheaves, and we use the category in the case of meromorphic line bundles
to recover the $j_*^{\mer}\mathcal{O}_S^{\times}$-torsor.)  We also give a bit of background on
the notion of meromorphic extensions and the (very significant) differences between the
analytic and algebraic cases.

In the next section, \S\ref{s.ethm}, we prove that
$\mathcal{L}(\nu,\mu)$ extends as a holomorphic line bundle.  (In
Deligne's language, we prove that the meromorphic extension is
effective.)  We felt that this section should go before \S\ref{s.ethm}
because the meromorphic extension is unique, while the holomorphic
extension depends on some choices.  However, \S\ref{s.ethm} does not
logically depend on this section.  So, the reader may want to skip
this section starting from subsection~\ref{extsh} (where we begin
the study of meromorphic sheaves) at first reading. 

\subsection{Notation}\label{notation} Following~\cite[p.~61]{deligne-diffeq}, we take
$X$ to be an analytic space with $Y$ a closed analytic subset and
$X^*=X\setminus Y$.  We write $j:X^*\to X$ and $i:Y\to X$ for the
inclusions.

\subsection{Extensions of line bundles} Before bringing up the
subject of extensions of coherent sheaves, we want to consider
extensions of holomorphic line bundles from $X^*$ to $X$ to illustrate
some of the differences between the analytic and algebraic settings.

\begin{definition}\label{lbext}
  Suppose $\mathcal{L}^*$ is a holomorphic line bundle on $X^*$, an
  \emph{extension} of $\mathcal{L}^*$ to a line bundle on $X$ is a
  pair $(\mathcal{L},r)$ where $\mathcal{L}$ is a line bundle on $X$
  and $r:\mathcal{L}_{|X^*}\to \mathcal{L}^*$ is an isomorphism of
  line bundles on $X^*$.  If $(\mathcal{L}_i,r_i)$ ($i=1,2$) are two
  extensions of $\mathcal{L}^*$, then a \emph{morphism} from
  $(\mathcal{L}_1, r_1)$ to $(\mathcal{L}_2,r_2)$ is a morphism of
  line bundles $\phi:\mathcal{L}_1\to \mathcal{L}_2$ commuting with
  the isomorphisms $r_i$.  In this way, we get a category
  $\mathcal{P}(\mathcal{L}^*)$ of extensions of $\mathcal{L}^*$ to
  $X$.
\end{definition}

\begin{lemma}\label{Hartogs}
  Suppose $X$ is complex manifold and $Y$ has codimension at least $2$ in $X$.
  Let $\mathcal{L}^*$ be a line bundle on $X^*$ which extends to a line
  bundle $X$.
  Then this extension is unique up to isomorphism.
\end{lemma}

\begin{proof}
  Suppose $(\mathcal{L}_i,r_i)$ are two extensions.  Then
  $r_2^{-1}\circ r_1:j^*\mathcal{L}_1\to j^*\mathcal{L}_2$ is an isomorphism
  of line bundles.   By Hartog's theorem, it extends to an isomorphism
  $R:\mathcal{L}_1\to\mathcal{L}_2$ commuting with the restrictions $r_i$. 
  \end{proof}

  There are two main problematic differences between extensions of
  line bundles in the holomorphic and the algebraic case.  The first
  is that not all holomorphic line bundles on $X^*$ extend to $X$ even
  when $X$ is a manifold.  For example, as we pointed out in
  Remark~\ref{remexth}, there are infinitely many analytic line bundles
  on $\mathbb{C}^2\setminus\{0\}$ which do not extend to
  $\mathbb{C}^2$.  In fact, by the exponential exact sequence, it is
  easy to see that every line bundle on $\mathbb{C}^2$ is trivial.  So
  the trivial line bundle is the only line bundle on
  $\mathbb{C}^2\setminus\{0\}$ which extends to $\mathbb{C}^2$.
  By Lemma~\ref{Hartogs}, in this case, the extension is
  unique.

  The second difference is that analytic line bundles can have too
  many extensions.  To be precise, they can have uncountably many
  extensions which are not even meromorphically equivalent.  (We show
  in Proposition~\ref{ame} that, when $X$ is smooth, meromorphic
  extensions of  line bundles in the algebraic setting are unique.)  Before giving Deligne's
  definition of meromorphic equivalence, we want to illustrate this
  problem with an example.

\begin{example}\label{ecurve}
  Let $E$ be an (algebraic) elliptic curve
  over $\mathbb{C}$ and let $p$ be a point in $E$.  Set
  $U=E\setminus\{p\}$ considered as an algebraic curve.   Then
write $X, X^*$ and $Y$ for $E, U$ and $\{p\}$ respectively regarded as
analytic varieties.

We have an exact sequence
$$
0\to \mathbb{Z}\to \Pic E\to \Pic U\to 0.
$$
where the first non-trivial map sends $1$ to the line bundle
$\mathcal{O}_{E}([p])$ and the second sends a line bundle to its
restriction to $U$.   From this, it is not hard to see that any algebraic line
bundle on $U$ extends to $E$, and, while the extension is not
unique, it is unique modulo tensoring with $\mathcal{O}_E([p])$.
(Here we use the obvious algebraic analogue of Definition~\ref{lbext}.)

The analytic case is very different.  Here, by GAGA~\cite{GAGA}, we have
$\Pic X=\Pic E$.  However, from the exponential exact sequence
$$
0\to 2\pi i\mathbb{Z}\to \mathcal{O}_{X^*}\to \mathcal{O}_{X^*}^{\times}\to 1
$$
and the fact that $X^*$ is Stein, it follows easily that $\Pic X^*=0$.
So the restriction of any line bundle on $X$ to $X^*$ is
trivial. The trivial line bundle on $X^*$ has uncountably
many non-isomorphic extensions to $X$ as every line bundle on $X$
gives rise to an extension of $\mathcal{O}_{X^*}$.  In fact, the
situation is even worse than it seems: every line bundle $\mathcal{L}$
on $X$ gives rise to uncountably many non-isomorphic extensions of $\mathcal{O}_{X^*}$.  To understand this phenomenon we make the following definition. 
\end{example}

\begin{definition}
  With $X, Y$ and $X^*$ as in~\eqref{notation}, write $P$ for the
  set of isomorphism classes of extensions $(\mathcal{L},r)$ of the trivial
  line bundle $\mathcal{O}_{X^*}$ to a line bundle on $X$.
  We give $P$  an (abelian) group structure  by
  setting $(\mathcal{L}_1,r_1)(\mathcal{L}_2,r_2)=
  (\mathcal{L}_1\otimes\mathcal{L}_2,r_1\otimes r_2)$.  
\end{definition}

\begin{proposition}
  Suppose that $X$ is a complex manifold and $X^*$ is a Zariski open subset.
  Then we have an exact sequence
  $$
  1\to \frac{H^0(X^*,\mathcal{O}_{X^*}^{\times})}{H^0(X,\mathcal{O}_X^{\times})}
  \to P\to \Pic X\to \Pic X^*.$$
\end{proposition}
\begin{proof}
  First we describe the maps in the sequence.  The last one is restriction.
  The second-to-last sends an extension $(\mathcal{L},r)$ to $\mathcal{L}$.
  The sequence is exact at $\Pic X$ by the definition of $P$ (as
  the set of isomorphism classes of extensions of the trivial line bundle on $X^*$).  We have a map $\phi:H^0(X,\mathcal{O}_{X^*}^{\times})\to P$
  given by $\phi(f)=(\mathcal{O}_X,f)$ where we think of $f$ as an isomorphism
  from $\mathcal{O}_{X^*}$ to itself.  The kernel of $\phi$ consists of
  function $f$ with $(\mathcal{O}_X,f)$ isomorphic to $(\mathcal{O}_X,1)$.
  This is exactly $H^0(X,\mathcal{O}_X^{\times})$.  Since any element in
  the kernel of $P\to \Pic X$ can be written as $(\mathcal{O}_X,f)$
  for some $f$ as above, this finishes the proof of the proposition.
\end{proof}

\begin{remark}
  In the case of the elliptic curve from Example~\ref{ecurve} above,
  both $\Pic X$ and
  the group $H^0(X^*,\mathcal{O}_{X^*}^{\times})/H^0(X,\mathcal{O}_X^{\times})=H^0(X^*,\mathcal{O}_{X^*}^{\times})/\mathbb{C}^{\times}$ are uncountable.
\end{remark}

\subsection{Meromorphic Extensions of Coherent Analytic Sheaves}\label{extsh}
For any analytic space $T$ we write $\Coh T$ for the category of coherent
sheaves on $T$.

Suppose $\mathcal{F}^*$ is a coherent analytic sheaf on $X^*$.  An
\emph{extension of $\mathcal{F}^*$ to $X$} is a coherent analytic
sheaf $\mathcal{F}$ on $X$ together with an isomorphism
$r_{\mathcal{F}}:\mathcal{F}_{|X^*}\to \mathcal{F}^*$.  A morphism
$\mathcal{F}\to \mathcal{F}'$ of extensions is then a morphism of
coherent sheaves on $X$ respecting the isomorphisms to
$\mathcal{F}$ over $X^*$.

We now give a theorem from~\cite[p.~65]{deligne-diffeq}

\begin{theorem}[Deligne]\label{eqconds}
Suppose $\mathcal{F}_1$ and $\mathcal{F}_2$ are extensions of $\mathcal{F}$.
Then the following conditions are equivalent.  
\leqnomode
  \begin{align}
  &\text{there is an extension $\mathcal{F}_3$ of $\mathcal{F}$ along with  morphisms from
    $\mathcal{F}_3$ to $\mathcal{F}_1$ and $\mathcal{F}_2$;}\label{merex1}\\
  &\text{there is an extension $\mathcal{F}_4$ of $\mathcal{F}$ along with a morphisms from $\mathcal{F}_1$ and
    $\mathcal{F}_2$ to $\mathcal{F}_4$;}\label{merex2}\\
  &\text{either \eqref{merex1} or \eqref{merex2} hold locally on $X$.}\label{merex3}
\end{align}
\reqnomode
\end{theorem}

Deligne says that $\mathcal{F}_1$ and $\mathcal{F}_2$ are \emph{meromorphically
  equivalent} if the conditions in Theorem~\ref{eqconds} above hold. 

\begin{proof}
For the convenience of the reader, we give a slightly
expanded version of Deligne's proof of the
equivalence of (\ref{merex1}---\ref{merex3}) above.   First note that,
if $\mathcal{G}$ is any coherent analytic sheaf on $X$, then the sheaf
$\Gamma_Y\mathcal{G}$ of sections with support in $Y$ is coherent
~\cite[Proposition 3,
p.~366]{SerreProlong}.
Moreover, we have a short exact sequence
\begin{equation}
  \label{e.supex}
  0\to \Gamma_Y\mathcal{G}\to \mathcal{G}\to j_*j^{-1}\mathcal{G}.
\end{equation}
So, the coherence of  $\Gamma_Y\mathcal{G}$ implies the coherence of the
image of the map $\mathcal{G}\to j_*j^{-1}\mathcal{G}$.

Now, suppose we have $\mathcal{F}_3$ as in \eqref{merex1}.  Set
$\mathcal{F}_4=(\mathcal{F}_1\oplus \mathcal{F}_2)/\mathcal{F}_3$
where the embedding of $\mathcal{F}_3$ is the diagonal embedding.
Then $\mathcal{F}_4$ with the obvious morphisms satisfies \eqref{merex2}.
On the other hand, if we have $\mathcal{F}_4$ as in \eqref{merex2},
setting $\mathcal{F}_3=\mathcal{F}_1\cap \mathcal{F}_2$ gives an
extension satisfying \eqref{merex1}.

Finally, suppose \eqref{merex1} holds locally.  Set $\mathcal{F}_4$
equal to the sum of the images of $\mathcal{F}_1$ and $\mathcal{F}_2$
in $j_*\mathcal{F}$.   The morphisms $r_i^{-1}:\mathcal{F}\to j^*\mathcal{F}_i$ for $i=1,2$ induce morphisms $a_i:\mathcal{F}\to j^*\mathcal{F}_4$.  We need to show
that $\mathcal{F}_4$ is coherent and that the $a_i$ ($i=1,2$) are two idenitical
isomorphisms.

Fortunately, both statements above are local on $X$. So we can assume
that \eqref{merex1} holds globally.  Then set $\mathcal{F}_3'$ equal to
the image of $\mathcal{F}_3$ in $\mathcal{F}_1\oplus \mathcal{F}_2$.
We get an exact sequence
\begin{equation}
  \label{e.f3p}
  0\to \Gamma_Y(\frac{\mathcal{F}_1\oplus\mathcal{F}_2}{\mathcal{F}_3'})
  \to \frac{\mathcal{F}_1\oplus\mathcal{F}_2}{\mathcal{F}_3'}\to
  \mathcal{F}_4\to 0.
\end{equation}
This shows that $\mathcal{F}_4$ is coherent.  And applying $j^*$ proves the
rest. 
\end{proof}

We say that two holomorphic line bundles $(\mathcal{L}_i,r_i)$ are
meromorphically equivalent if they are meromorphically equivalent as
coherent analytic sheaves.

We can also make the same definition of meromorphic equivalence in the
algebraic setting replacing $X^*, X, \mathcal{F}^*$ and $\mathcal{F}$ with
algebraic spaces and coherent algebraic sheaves.  Then we have the following proposition.

\begin{proposition}\label{ame}
  Suppose $X$ is a smooth complex algebraic variety and $X^*$ is a Zariski
  open subset.  
  Any two algebraic extensions $(\mathcal{L}_1,r_1)$ and $(\mathcal{L}_2, r_2)$
  of an algebraic line bundle $\mathcal{L}^*$ on $X^*$ are meromorphically
  equivalent. 
\end{proposition}

\begin{proof}
  As in the analytic case, by~\eqref{merex3}, the question is local on
  $X$.  So pick a point $x\in X$.  We can find an affine open
  neighborhood of $x$ where $\mathcal{L}_1$ and $\mathcal{L}_2$ are
  trivial, i.e., isomorphic to $\mathcal{O}_X$.  Replacing $X$ with
  this affine open neighborhood, we can regard
  $f:=r_2^{-1}\circ r_1:\mathcal{O}_{X^*}\to \mathcal{O}_{X^*}$ as a
  meromorphic function on $X$ with poles and zeros only on
  $Y=X\setminus X^*$.  Since $X$ is smooth, $\mathcal{O}_{X,x}$ is a UFD.
  So, after possibly shrinking $X$ further about $x$, we can assume that
  $f=g/h$ with $g$ and $h$ non-zero regular functions which are non-vanishing
  off of $Y$. 
Now, set $(\mathcal{L}_3,r_3):=(\mathcal{O}_X, r_1/h)$.  Multiplication
  by $h$ (resp. $g$) gives an inclusion of $\mathcal{L}_3$ in
  $\mathcal{L}_1$ (resp. $\mathcal{L}_2$) compatible with the isomorphisms
  $r_i$.  So, by \eqref{merex1},  $(\mathcal{L}_1,r_1)$ and  
  $(\mathcal{L}_2,r_2)$ are meromorphically equivalent. 
\end{proof}

\begin{remark}
  In the analytic case, there are, in general, uncountably many
  non-meromophically equivalent extensions of a given holomorphic line
  bundle.  This is not hard to see directly from Example~\ref{ecurve}. 
\end{remark}

\begin{lemma}
  Suppose $\mathcal{F}$ is an extension to $X$ of a coherent analytic
  sheaf $\mathcal{F}^*$ on $X^*$.  Set
  $\mathcal{G}:=\mathcal{F}/\Gamma_Y\mathcal{F}$.  Then the composition
  \begin{equation}\label{composition}
  \mathcal{F}^*\stackrel{r_{\mathcal{F}}}{\to} j^*\mathcal{F}\to j^*\mathcal{G}
  \end{equation}
  gives $\mathcal{G}$ the structure of an extension of $\mathcal{F}^*$.
  Moreover, $\Gamma_Y\mathcal{G}=0$, and  $\mathcal{G}$ is meromorphically
  equivalent to $\mathcal{F}$. 
\end{lemma}

\begin{proof}
Obvious. \end{proof}

In~\cite[p.~65]{deligne-diffeq}, Deligne gives the following definition.

\begin{definition}\label{MeroExt}
  A coherent analytic sheaf on $X^*$ meromorphic along $Y$ is a
  coherent analytic sheaf on $X^*$ together with a locally defined system of equivalence
  classes of extensions to $X$.  
\end{definition}

We want to reformulate this definition in a way that lends itself to
defining a category of meromorphic extensions.  For this, suppose
$U$ is an open subset of $X$.  Set $U^*=U\cap X^*$, $Y_U=U\cap Y$ and
write $j_U:U^*\to U$ for the inclusion.  Write $\Coh_{(Y)} U$ for the full
subcategory of $\Coh U$ consisting of sheaves supported on $Y$.
This is a thick subcategory.  
Set $\Coh^{\effm} U:=\Coh U/\Coh_{(Y)} U$, the quotient category.  We call this
the category of \emph{effective meromorphic extensions on $U$}.
Since $j^*\Coh_{(Y)} U=0$, the restriction functor $\Coh U\to \Coh U^*$ factors as $\Coh U\stackrel{q}{\to} \Coh^{\effm} U\to \Coh U^*$ naturally (where $q$ is the
quotient functor).

\begin{proposition}
  Suppose $V\subset U$ is the inclusion of an open set. Then the
  restriction functor $\Coh U\to \Coh V$ induces an exact functor
  $\Coh^{\effm} U\to \Coh^{\effm} V$. 
\end{proposition}

\begin{proof}
  Easy (so we leave it to the reader). 
\end{proof}

It follows that we get a category $\Coh^{\effm}_{X}$, or simply $\Coh^{\effm}$, fibered over the category
$X^{\topp}$ of open subsets of $X$ (whose fiber over $U$ is $\Coh^{\effm} U$).
(See~\cite{giraud} for the notion of fibered categories.)

\begin{lemma}\label{qeq}
  Suppose $\mathcal{F}_1$ and $\mathcal{F}_2$ are two extensions to $X$
  of a coherent analytic sheaf $\mathcal{F}^*$ on $X^*$.  Then the following
  are equivalent
  \begin{enumerate}
  \item[(a)] $\mathcal{F}_1$ and $\mathcal{F}_2$ are meromorphically
    equivalent.
  \item[(b)] There is an isomorphism $q(\mathcal{F}_1)\to q(\mathcal{F}_2)$
    commuting with the isomorphisms $j^*\mathcal{F}_i\to \mathcal{F}^*$. 
  \end{enumerate}
\end{lemma}

\begin{proof}
  Suppose that $\mathcal{F}_1$ and $\mathcal{F}_2$ are meromorphically
  equivalent.  Let $\mathcal{F}_3$ be as in \eqref{merex1}.   Then
  $\mathcal{F}_i/\mathcal{F}_3$ is supported on $Y$ for $i=1,2$.
  So the induced map $q(\mathcal{F}_3)\to q(\mathcal{F}_i)$ is an
  isomorphism for $i=1,2$.  Then (b) follows.

  Now assume that (b) holds.  If we set
  $\mathcal{G}_i:=\mathcal{F}_i/\Gamma_Y\mathcal{F}_i$, for $i=1,2$,  then the
  quotient map $\mathcal{F}_i\to \mathcal{G}_i$ induces both an
  equivalence of meromorphic extensions and an isomorphism in
  $\Coh^{\effm} X$.   So, by replacing $\mathcal{F}_i$ with $\mathcal{G}_i$,
  we see that we can assume that neither $\mathcal{F}_1$ nor$\mathcal{F}_2$
  has a non-trivial subsheaf supported on $Y$.
  
  Assume then that $f:q\mathcal{F}_1\to q\mathcal{F}_2$ is an
  isomorphism commuting as in (b).   By the definition
  of a quotient category, we have
  $$
  \Hom_{\Coh^{\effm} X}(q\mathcal{F}_1,q\mathcal{F}_2)
  =\varinjlim_{\mathcal{F}_1',\mathcal{F}_3'}\Hom(\mathcal{F}_3,\mathcal{F}_2/\mathcal{F}_4)$$
  where the limit runs over all pairs
  $\mathcal{F}_3,\mathcal{F}_4$ of coherent subsheaves of
  $\mathcal{F}_1$ and
  $\mathcal{F}_2$ respectively such that
  $\mathcal{F}_1/\mathcal{F}_3$ and
  $\mathcal{F}_4$ are supported on $Y$.  Since
  $\mathcal{F}_2$ has no non-trivial subsheaf supported on $Y$,
  $f$ is represented by a morphism $\tilde{f}:\mathcal{F}_3\to
  \mathcal{F}_2$.  Since
  $f=q(\tilde{f})$ is an isomorphism,
  $\ker\tilde{f}$ and $\coker\tilde{f}$ are both supported on $Y$.
  Since $\mathcal{F}_1$ has no non-trivial subsheaf supported on $Y$,
  this implies that $\tilde{f}:\mathcal{F}_3\to \mathcal{F}_2$ is a mono-morphism with cokernel supported on $Y$. From this (a) follows directly. 
\end{proof}

\begin{proposition}\label{resmono}
  Suppose $\mathcal{F}$ and $\mathcal{G}$ are coherent
  analytic sheaves on $X$.  Then the natural morphism 
  $$
  \rho:\Hom_{\Coh^{\effm} X}(q\mathcal{F},q\mathcal{G})\to
  \Hom_{\Coh X^*} (j^*\mathcal{F},j^*\mathcal{G})$$
   induced by restriction
  is a monomorphism.  
\end{proposition}

\begin{proof}
  We can assume that $\Gamma_Y\mathcal{F}=\Gamma_Y\mathcal{G}=0$.
  Then  $\Hom_{\Coh^{\effm} X}(q\mathcal{F},q\mathcal{G})=
  \varinjlim_{\mathcal{F}'} \Hom_{\Coh X}(\mathcal{F}',\mathcal{G})$
  where the limit is taken over all subsheaves $\mathcal{F}'$ of $\mathcal{F}$
  such that $\mathcal{F}/\mathcal{F}'$ is supported on $Y$.   Suppose
  $f$ is a homomorphism from $q\mathcal{F}$ to $q\mathcal{G}$ represented
  by a morphism $f':\mathcal{F}'\to\mathcal{G}$.   If $j^*(f')=0$, then
  $f'(\mathcal{F}')$ is supported on $Y$.  But, since $\Gamma_Y\mathcal{G}=0$,
  this implies that $f'=0$.  
\end{proof}

\begin{definition}
  We write $\Coh^m_X$ for the stackification  of the fibered category
  $\Coh^{\effm}_X$~\cite[p.~76]{giraud}.  This is the category of \emph{coherent analytic sheaves
    on $X$ meromorphic along $Y$}. 
\end{definition}

If $\Coh_X$ denotes the stack of coherent analytic sheaves on $X^{\topp}$,
then we get a sequence of morphisms of fibered categories
\begin{equation}
  \label{Q}
  \Coh_X\to \Coh_X^{\effm}\to \Coh_X^m.
\end{equation}
We call the composition $Q$.  

For any open subset $U\subset X$,  
we get a restriction functor
$\Coh^{\effm}_X(U)\to \Coh_{X^*}(U^*)$ (with $U^*=U\cap X$).
It is easy to see that this functor factors naturally through
$\Coh^{m}_X(U^*)$.  By abuse of notation, we write
\begin{equation}
  \label{ejst}
  j^*:\Coh_X^m(U)\to \Coh_X(U^*)
\end{equation}
for the the induced restriction functor.

\begin{proposition}
  A coherent analytic sheaf on $X^*$ meromorphic along $Y$ is the same
  thing as an object of $\Coh^m_X(X)$. 
\end{proposition}

\begin{proof}
  By the definition of stackification, an object $\mathcal{F}$ of $\Coh^m_X(X)$
  consists of a family of objects $\mathcal{F}_{\alpha}$ in
  $\Coh^{\effm}_X(U_{\alpha})$ for an open cover $\{U_{\alpha}\}$ of
  $X$ together with descent data for gluing the data together.
  Explicitly, this descent data consists of isomorphisms
  \begin{equation}\label{dm} \phi_{\alpha,\beta}:
  \mathcal{F}_{\alpha|U_{\alpha}\cap U_{\beta}}
  \to \mathcal{F}_{\beta|U_{\alpha}\cap U_{\beta}}
  \end{equation}
  in the category $\Coh^{\effm}_X(U_{\alpha}\cap U_{\beta})$ which are compatible
  in the sense that $\phi_{\gamma,\beta}\phi_{\alpha,\beta}=\phi_{\gamma,\alpha}$.

  Since the restriction of the descent data to $X^*$ gives descent data
  for a coherent sheaf on $X^*$, the object $\mathcal{F}$
  gives rise to a sheaf $\mathcal{F}^*$ on $X^*$.  Moreover, it is not
  difficult to see that this sheaf $\mathcal{F}^*$ is independent
  of the choice of presentation of $\mathcal{F}$ in terms of descent
  data.  This gives the
  functor $\Coh^m_X(X)\to \Coh X^*$ explicitly.   The sheaves $\mathcal{F}_{\alpha}$ are then by definition meromorphic extensions of $\mathcal{F}_{U_{\alpha}^*}$.   So, from this, we see that an object $\mathcal{F}$ in $\Coh^m_X$ gives
  rise to a coherent sheaf $\mathcal{F}^*$ on $X^*$ along with a meromorphic
  extension of $\mathcal{F}^*$ to $X$ in Deligne's sense.

  On the other hand, suppose we start with a coherent sheaf
  $\mathcal{F}^*$ on $X^*$, an open covering $\{U_{\alpha}\}$ of $X$
  and a family $(\mathcal{F}_{\alpha},r_{\alpha})$ of extensions of
  $\mathcal{F}^*$ from $U_{\alpha}^*$ to $U_{\alpha}$.  Assume that $\mathcal{F}_{\alpha}$ and $\mathcal{F}_{\beta}$ have meromorphically equivalent restrictions to
  $U_{\alpha\beta}:=U_{\alpha}\cap U_{\beta}$. Then there is a morphism
  $\phi_{\alpha,\beta}$ in $\Coh^{\effm}(U_{\alpha\beta})$ as in \eqref{dm}
  commuting with the restrictions $r_{\alpha}$ and $r_{\beta}$.
  The fundamental point to make now is that, in fact, by Proposition~\ref{resmono}, there is a \emph{unique} such morphism $\phi_{\alpha,\beta}$.
  This implies that the resulting family $\{\phi_{\alpha,\beta}\}$ automatically
  satisfies the descent condition required to give an object
  $\mathcal{F}$ in $\Coh^m_X(X)$.

  We leave the rest of the verification (e.g., the fact that $\mathcal{F}$
  is independent of the presentation $(\mathcal{F}_{\alpha},r_{\alpha})$)
  to the reader.
  \end{proof}

  \begin{lemma}\label{ror}
    Suppose $X$ is a complex manifold, and 
    $\mathcal{F}\in\Coh_X^{\effm}(X)$ is is a coherent analytic
    sheaf on $X$ such that $j^*\mathcal{F}$ is an invertible sheaf.
    Then the double dual $\mathcal{F}^{\vee\vee}$ is an invertible sheaf on
    $X$ and the canonical morphism $\mathcal{F}\to\mathcal{F}^{\vee\vee}$
    is an isomorphism in $\Coh_X^{\effm}(X)$.  
  \end{lemma}

  \begin{proof}
    The double dual $\mathcal{F}^{\vee\vee}$ is reflexive.  (This is
    proved in the algebraic setting in~\cite{HartshorneReflexive}, and
    the same proof works in the analytic setting.) Since
    $\mathcal{F}^{\vee\vee}$ is rank $1$ and reflexive, it is a line bundle
    ~\cite[Lemma 1.1.15, p.~154]{OkenekSchneiderSpindler}.
    The kernel and cokernel of the canonical morphism $\phi:\mathcal{F}\to \mathcal{F}^{\vee\vee}$ is clearly supported on $Y$.  So $\phi$ is an isomorphism
    in $\Coh_X^{\effm}(X)$.  
  \end{proof}

  \begin{corollary}\label{ror2}
    Suppose $X$ is a complex manifold, and 
    $\hat{\mathcal{F}}$ is an object in $\Coh_X^m$ such that
    $j^*(\hat{\mathcal{F}})$ is a line bundle.  Then $\hat{\mathcal{F}}$
    is locally isomorphic to a line bundle.
  \end{corollary}

  \begin{proof}
    Since this is (by definition) a local question, we can assume that
    the object $\hat{\mathcal{F}}$ is represented by a coherent analytic
    sheaf $\mathcal{F}$ on $X$. Then the result follows from Lemma~\ref{ror}.
  \end{proof}

  \begin{definition}
    We call an object $\hat{\mathcal{F}}\in \Coh^m_X$ which is locally
    isomorphic to a line bundle a 
    \emph{meromorphic line bundle (along $Y$)}.   
  \end{definition}

Write $\hPic_Y(X)$ for the set of isomorphism classes of
meromorphic line bundles along $Y$.   Obviously restriction gives
a map $\hPic_Y(X)\to \Pic(X^*)$, and we get a map $\Pic X\to \hPic_Y(X)$
by associating to each line bundle on $X$ its associated meromorphic line
bundle.   Moreover, the composition $\Pic X\to \hPic_Y X\to \Pic X^*$
is just the usual restriction.

Suppose $\hat{\mathcal{L}}$ and $\hat{\mathcal{M}}$ are meromorphic
line bundles.  Then, in the language of Deligne, $\hat{\mathcal{L}}$
and $\hat{\mathcal{M}}$ are meromorphic extensions of
$\mathcal{L}:=j^*\hat{\mathcal{L}}$ and
$\mathcal{M}:=j^*\hat{\mathcal{M}}$ respectively.  We can find a
covering $\{U_{\alpha}\}_{\alpha\in I}$ of $X$ and extensions
$\mathcal{L}_{\alpha}$, $\mathcal{M}_{\alpha}$ of
$\mathcal{L}_{|U_{\alpha}^*}$ and $\mathcal{M}_{U_{\alpha}^*}$
respectively representing the meromorphic equivalence classes
$\hat{\mathcal{L}}$ and $\hat{\mathcal{M}}$.  By Corollary~\ref{ror2},
we can, moreover, assume that $\mathcal{L}_{\alpha}$ and
$\mathcal{M}_{\alpha}$ are line bundles.  It is not hard to check
that, for $\alpha,\beta\in I$,
$\mathcal{L}_{\alpha}\otimes\mathcal{M}_{\alpha}$ is meromorphically
equivalent to $\mathcal{L}_{\beta}\otimes\mathcal{M}_{\beta}$ as an
extensions of $\mathcal{L}\otimes\mathcal{M}$ from
$U_{\alpha}\cap U_{\beta}\cap X^*$ to $U_{\alpha}\cap U_{\beta}$.
Thus,
$\{\mathcal{L}_{\alpha}\otimes\mathcal{M}_{\alpha}\}_{\alpha\in I}$
represents a meromorphic extension of $\mathcal{L}\otimes\mathcal{M}$.
It is then not hard to check that the meromorphic equivalence class of
this extension is independent of the choice of the covering
$\{U_{\alpha}\}$ and the choices of the $\mathcal{L}_{\alpha}$ and
$\mathcal{M}_{\alpha}$ representing $\hat{\mathcal{L}}$ and
$\hat{\mathcal{M}}$ respectively.  So it makes sense to define the
tensor product $\hat{\mathcal{L}}\otimes\hat{\mathcal{M}}$ to be the
meromorphic extension of $\mathcal{L}\otimes\mathcal{M}$ represented
by the data
$\{\mathcal{L}_{\alpha}\otimes\mathcal{M}_{\alpha}\}_{\alpha\in I}$.
This gives rise to an abelian group structure on $\hPic_Y X$ (defined
by taking isomorphism classes of tensor product).  Moreover, it is
easy to see that the maps $\Pic X \to \hPic_Y X \to \Pic X^*$ are
group homomorphisms under this tensor product.  
  
  Adjunction gives a morphism of sheaves
  $\mathcal{O}_X\to j_*j^{-1}\mathcal{O}_X$ (which is a monomorphism if $X$
  is a complex manifold and $Y$ is nowhere dense).  Then 
  $j_*^{\mer}\mathcal{O}_{X^*}$ is the subsheaf of $j_*j^{-1}\mathcal{O}_X$
  consisting of sections which can be locally written in the form
  $g/h$ with $g,h\in\mathcal{O}_X$ and $h$ invertible outside of $Y$. 
  We then get a morphism of sheaves $\mathcal{O}_X\to j_*^{\mer}\mathcal{O}_{X*}$.
  For a coherent sheaf $\mathcal{F}$ on $X$, restriction gives a morphism
  $R:\mathcal{F}\otimes_{\mathcal{O}_X}j_*^{\mer}\mathcal{O}_{X^*}\to j_*j^{-1}\mathcal{F}$
  (which is a monomorphism for $\mathcal{F}$ locally free).
  Write $\mathcal{F}^{\mer}$  for the image of $R$.

  Note that $\mathcal{O}_X^{\mer}=j_*^{\mer}\mathcal{O}_{X^*}$. 
  
  \begin{proposition}\label{mero}
    Suppose $\mathcal{F}$ is a reflexive sheaf on a
    smooth complex analytic space $X$.  Then
    $$\Hom_{\Coh_X^{m}}(Q\mathcal{O_X},Q\mathcal{F})=
    \mathcal{F}^{\mer}.$$
  \end{proposition}

  \begin{proof}
    Restriction gives a morphism of sheaves
    \begin{equation}\label{reshom}
      \rho:\Hom_{\Coh_X^m}(Q\mathcal{O}_X,Q\mathcal{F})\to
      j_*j^{-1}\mathcal{F}.
    \end{equation}
    Using Proposition~\ref{resmono}, we
    see that this is a monomorphism.  To see that it factors through
    $\mathcal{F}^{\mer}$ is a local question.  So suppose
    $\varphi:Q\mathcal{O}_X\to Q\mathcal{F}$ is a morphism in
    $\Coh_X^m$ defined near a point $x\in X$.  We then have a coherent
    analytic subsheaf $\iota:\mathcal{I}\to\mathcal{O}_X$ and a morphism
    $\psi:\mathcal{I}\to \mathcal{F}$ such that
    $\mathcal{O}_X/\mathcal{I}$ is supported on $Y$ and
    $\varphi=Q(\psi)\circ Q(\iota)^{-1}$.  Taking double duals and using
    the assumption that $\mathcal{F}$ is reflexive, we get a map
    $\sigma:\mathcal{I}^{\vee\vee}\to \mathcal{F}$.  Then
    $J:=\mathcal{I}^{\vee\vee}$ is a rank $1$ reflexive sheaf, and,
    therefore, invertible.  Moreover, we have
    $\mathcal{I}\subset \mathcal{J}\subset\mathcal{O}_X$.  So
    $\mathcal{O}_X/\mathcal{J}$ is also supported on $Y$.  By replacing $X$ with a suitably small open neighborhood of $x$,
    we can assume that $\mathcal{J}=h\mathcal{O}_X$ is a principal
    ideal.  Writing $\iota_{\mathcal{J}}:\mathcal{J}\to \mathcal{O}_X$ for the
    inclusion, it follows that
    $\varphi=Q(\sigma)\circ Q(\iota_J)^{-1}$.  We then get that
    $\rho(\varphi)=\sigma(h)/h\in\mathcal{F}^{\mer}$.

    To show that $\rho$ induces an isomorphism with
    $\mathcal{F}^{\mer}$, take $s\in\mathcal{F}^{\mer}(U)$ for
    some open set $U$.  We can work
    locally near a point $x\in X$, so we can assume that $s=a/h$ where
    $a$ and $h$ are holomorphic near $x$ and $h$ is a unit off of $Y$.
    Set $\mathcal{I}=h\mathcal{O}_X$, and let
    $\iota:\mathcal{I}\to \mathcal{O}_X$ denote the inclusion.
    Write $\alpha:\mathcal{I}\to \mathcal{F}$ for the $\mathcal{O}_X$ linear morphism sending $h$ to $a$.  Then 
    we see that $s=\rho(Q(\alpha))\circ Q(\iota)^{-1})$.  
  \end{proof}

  \begin{definition}\label{merosecs}
    Suppose $\hat{\mathcal{L}}$ is a meromorphic line bundle on $X$
    along $Y$.  We
    write
    $$\hat{\mathcal{L}}^{\mer}:=\Hom_{\Coh_X^m}(Q\mathcal{O}_X,\hat{\mathcal{L}}^{\mer}).$$
    By Proposition~\ref{mero}, the sheaf $\hat{\mathcal{L}}^{\mer}$, which
    we call the \emph{sheaf of meromorphic sections} of $\hat{\mathcal{L}}$,
    is a locally rank one $j_*^{\mer}\mathcal{O}_X$-module.
  \end{definition}

  
  Suppose $X$ is smooth and  $\hat{\mathcal{L}}$ is a meromorphic line bundle
  along $Y$. 
   Write $\hat{\mathcal{O}}_X$ for the object $Q(\mathcal{O}_X)$ in
  $\Coh_X^m$, and write 
  $\Isom_{\Coh_X^m}(\hat{\mathcal{O}}_X,\hat{\mathcal{L}})$
  for the subsheaf of $\Hom_{\Coh_X^m}(\hat{\mathcal{O}}_X,\hat{\mathcal{L}})$
  consisting of isomorphisms.   Then
  $\Isom_{\Coh_X^m}(\hat{\mathcal{O}}_X,\hat{\mathcal{L}})$ is a
  torsor for the sheaf $\Aut_{\Coh_X^m}(\hat{\mathcal{O}}_X)$.
  By Proposition~\ref{mero}, $\Aut_{\Coh_X^m}(\hat{\mathcal{O}}_X)$ is
  identified with the subsheaf $j_*^{\mer}\mathcal{O}_{X^*}^{\times}$
  consisting of invertible sections of $j_*^{\mer}\mathcal{O}_{X^*}$.
  So, this gives a map
  \begin{equation}\label{cmap}
    c:\hPic_Y X\to H^1(X,j_*^{\mer}\mathcal{O}_{X^*}^{\times}).
    \end{equation}

    On the other hand, suppose $E$ is a
    $j_*^{\mer}\mathcal{O}_{X^*}^{\times}$-torsor.  The restriction,
    $E^*$, of $E$ to $X^*$ gives a line bundle $\mathcal{L}^*$ on
    $X^*$ via the Borel construction: Explicitly
    \begin{equation}\label{explicit}
    \mathcal{L}^*:=E^*\times^{\mathcal{O}_{X^*}^{\times}} \mathcal{O}_{X^*}.
    \end{equation}
    In other words, for $U$ open in $X^*$,  $\mathcal{L}^*(U)$ is the quotient of $E^*\times\mathcal{O}_X$ by the diagonal action of $\mathcal{O}_X^{\times}$ (acting by $f(e,g)=(f^{-1}e,fg)$).
    We get a map of sheaves $E^*\to \mathcal{L}^*$ (given explicitly by
    $e\mapsto (e,1)$) sending the sections of $E$
    isomorphically onto the non-vanishing sections of $\mathcal{L}^*$. 

    Now, if $U$ is an open set in $X$ and $s\in E(U)$, we get an
    extension $\mathcal{L}_s$ of $\mathcal{L}^*$ from $U^*=U\cap X^*$
    to $U$ essentially by declaring the image of $s$ in
    $\mathcal{L}^*$ to be a generator of $\mathcal{L}_s$.  To be more
    precise, we set $\mathcal{L}_s=\mathcal{O}_U$ and choose an
    isomorphism $j^*\mathcal{L}_s\to \mathcal{L}$ by sending $1$ to
    $s$.  If we pick a different $t\in E(U)$, then $\mathcal{L}_t$ and
    $\mathcal{L}_s$ are meromorphically equivalent because $s/t$ is
    meromorphic and non-vanishing off of $Y$.  So we can glue together the
    different choices of $\mathcal{L}_s$ for $s\in E(U)$ (varying $s$ and $U$), to get a meromorphic line bundle $b(E)$.   This gives us a map
    \begin{equation}\label{b}
    b:H^1(X,j_*^{\mer}\mathcal{O}_{X^*}^{\times})\to \hPic_Y X.
    \end{equation}

    \begin{proposition} The maps $b$ and $c$ above are inverse to each other.
      So, when $X$ is smooth, the sets $\hPic_Y X$ and $H^1(X,j_*^{\mer}\mathcal{O}_X^{\times})$ are
      isomorphic.  
  \end{proposition}
  \begin{proof}
    This is just a matter of wading through the definitions, and we feel
    it is best to leave it to the reader.  
  \end{proof}

  Suppose now that $X$ is smooth and, for simplicity, assume that $Y$
  is nowhere dense in $X$.  A \emph{meromorphic extension} of a line
  bundle $\mathcal{L}^*$ on $X^*$ is then the same thing as a pair
  $(\hat{\mathcal{L}}, r)$ with $\hat{\mathcal{L}}$ a meromorphic line
  bundle and $r:j^*\hat{\mathcal{L}}\to \mathcal{L}^*$ is an isomorphism.
 
  \begin{lemma}\label{extcons}
    Suppose $X$ is smooth and $Y$ is nowhere dense.
    Let $\mathcal{L}^*$ be a line bundle on $X^*$. Then there is a one-one
    correspondence between 
    meromorphic extensions of $\mathcal{L}^*$ and pairs $(E,\rho)$ where
    $E$ is a $j_*^{\mer}\mathcal{O}_{X^*}^{\times}$ torsor and $\rho:E^*\to\mathcal{L}^{*\times}$ is an $\mathcal{O}_{X^*}^{\times}$-equivariant isomorphism from the
    restriction of $E$ to $X^*$ to the sheaf of non-vanishing sections of
    $\mathcal{L}^*$. 
  \end{lemma}

  \begin{proof}
    If $(\hat{\mathcal{L}},r)$ is a meromorphic extension of $\mathcal{L}^*$,
    we get a torsor $E:=\Isom_{\Coh^m_X}(\hat{\mathcal{O}}_X,\hat{\mathcal{L}})$
    and the map $r:j^*\hat{\mathcal{L}}\to \mathcal{L}^*$ provides us with the
    isomorphism from $E^*$ to the non-vanishing sections of $\mathcal{L}^*$.
    On the other hand, suppose we are given a pair $(E,\rho)$ as above.
    Then $E$ gives rise to a meromorphic line bundle $\hat{\mathcal{L}}=b(E)$
    as in Proposition~\ref{b} above.  And, the map $\rho$ provides
    an isomorphism of the holomorphic bundle $E^*\times^{\mathcal{O}_{X^*}^{\times}}\mathcal{O}_{X^*}$ with $\mathcal{L}^*$. 
    It is easy to check that these two correspondences are mutually inverse.
    So the lemma follows.
  \end{proof}
  
  \begin{theorem}\label{mexlb}
    Suppose $\nu,\omega, S,\bar S$ and
    $\mathcal{L}=\mathcal{L}(\nu,\omega)$ are as in Q\ref{query:Q1}.
    Set $Y=\bar S\setminus S$.  Then there is a unique (up to isomorphism) meromorphic
    extension $\hat{\mathcal{L}}$ of $\mathcal{L}$ to $\bar{S}$ whose
    non-vanishing meromorphic sections are the admissible
    biextensions.
  \end{theorem}

  \begin{proof}
    The $j_*^{\mer}\mathcal{O}_S^{\times}$-torsor
    $\mathcal{B}^{\ad}(\nu,\omega)$ of Corollary~\ref{torsor} gives us
    a torsor $E$ and the restriction isomorphism from
    $\mathcal{B}^{\ad}(\nu,\omega)$ (which is a torsor on $\bar{S}$) to
    $\mathcal{B}(\nu,\omega)$ (which is the $\mathcal{O}_S^{\times}$
    torsor of invertible section of $\mathcal{L}$) gives us an
    isomorphism $\rho$ as in Lemma~\ref{extcons}.  So we see that
    $(E,\rho)$ gives us a meromorphic extension $\hat{\mathcal{L}}$ of
      $\mathcal{L}$ to $\bar S$.

      The extension is clearly unique up to isomorphism because
      $\mathcal{B}^{\ad}(\nu,\omega)$ determines both the torsor $E$
      and the map $\rho$.
  \end{proof}

  \section{Extension Theorem}\label{s.ethm}

  In this section, the goal is to show that the meromorphic extension
  of $\mathcal{L}(\nu,\omega)$ constructed in Theorem~\ref{mexlb}
actually gives rise to a line bundle
$\overline{\mathcal{L}}(\nu,\omega)$ extending
$\mathcal{L}(\nu,\omega)$.
The extensions $\overline{\mathcal{L}}(\nu,\omega)$ is not unique by the
choices involved are easily understood. 
  
\subsection{Setup and torsion pairings of smooth divisors}\label{stp}
\newcommand\ihz{\IH^1_{\mathbb{Z}}}

We take $\bar S$ to be a complex manifold, $j:S\to \bar S$ the
inclusion of a Zariski open and $\calH$ to be a weight $-1$ variation
of pure Hodge structure without torsion on $S$.  We write
$Y:=\bar S\setminus S$, $Y^{(1)}=\{Y_1,\ldots, Y_k\}$ for the set of
components of $Y$ which are codimension $1$ in $\bar S$.  We write
$Y^{(2)}$ for the union of the codimension $2$ components of $Y$ and
the singular locus of $Y$.  We write $S':=\bar S\setminus Y^{(2)}$,
and $Y'=Y\cap S'$.  Then $Y'$ is a smooth divisor in $S'$ with
components $Y_i':=Y_i\cap S'$.

\begin{definition}  Suppose $\calL$ is a $\mathbb{Z}$ local system 
  on $S$.  Set 
  $$\ihz\calL:=\{\alpha\in \rH^1(S,\calL): \alpha_{\mathbb{Q}}\in \IH^1(S,\calL_{\mathbb{Q}})\}.$$
\end{definition}

By Theorem~\ref{factor}, the map
$\cl:\ANF(S,\calH)_{\bar S}\to \rH^1(S,\mathcal{H})$ factors through
$\IH^1_{\mathbb{Z}}\calH$.

For each $Y_i\in Y^{(1)}$ we are going to define a pairing
\begin{equation}
  \label{e.tp}
  \tau_i:\ihz\calH\otimes\ihz\calH^{\vee}\to\mathbb{Q}/\mathbb{Z}
\end{equation}
which we will call the $i$-th torsion pairing.

We define \eqref{e.tp} first in the case that $\bar S=\Delta$,
$S=\Delta^*$ and $Y=\{0\}$.  Note that, in this case,
$\ihz(\Delta,\calH)=\{\alpha\in \rH^1(\Delta^*,\calH):\alpha_{\mathbb{Q}}=0\}$.
In other words, it is just the set of torsion elements of $\rH^1(\Delta^*,\calH)$. 
So, in this case the pairing in \eqref{e.tp} is just given by the torsion pairing $\tau$.

In general we take a test curve $\bar\varphi:\Delta\to S'$ intersecting
$Y_i\in Y^{(1)}$ transversally at a point $\bar\varphi(0)$.
By Theorem~\ref{IH1Ad}, the pull back of $\ihz\calH$ to $\Delta^*$ consists
of torsion classes.  Then we can define $\tau_i$ to be the torsion
pairing restricted to the test curve. 
It is easy to see that
this definition does not depend on the test curve.

\subsection{Extension of the line bundle $\mathcal{L}(\nu,\omega)$}

\begin{lemma}\label{refl}
  Suppose $j:S\to \bar S$ is an in \S\ref{stp}.  Suppose that
  $\codim_{\bar S} Y\geq 2$.  Let $\pi:\bar T\to \bar S$ be a
  proper morphism from a complex manifold, $\bar T$, which is an isomorphism over $S$,
  and suppose $\mathcal{M}$ is a line bundle on $\bar T$.  Set
  $\mathcal{L}:=(\pi_*\mathcal{M})^{\vee\vee}$, the reflexive hull of
  $\pi_*\mathcal{M}$.  Then $\mathcal{L}$ is a line bundle on $\bar S$
  and there is a canonical isomorphism
  $\mathcal{L}_{|S}\cong \mathcal{M}_{|S}$.  (See Remark~\ref{canext}.)
\end{lemma}

\begin{proof}
  Write $i:S\to \bar T$ for the open immersion lifting $j$.  Since $\pi$ is
  proper, $\pi_*\mathcal{M}$ is coherent.   We have $j^*\pi_*\mathcal{M}=
  i^*\mathcal{M}$ by smooth base change.  So $j^*\mathcal{L}=j^*(\pi_*(\mathcal{M})^{\vee\vee})=(i^*\mathcal{M})^{\vee\vee}=i^*\mathcal{M}$.

  Now, $\mathcal{L}$ is reflexive and agrees with a line bundle outside of
  codimension $2$. It is therefore a rank $1$ reflexive sheaf.  Therefore, since $\bar S$ is a complex manifold, $\mathcal{L}$ is a line bundle. (This is in
  ~\cite{OkenekSchneiderSpindler}.) 
\end{proof}

\begin{theorem}\label{eb1}
  Suppose $S$, $\bar{S}$, $Y$ and $\mathcal{H}$ are as in the beginning of \S\ref{stp},
  so that $\mathcal{H}$ is a torsion free variation of Hodge structure on $S$ of weight $-1$
  and $\bar{S}$ is a smooth partial compactification of $S$.  Suppose
  $$(\nu,\omega)\in \ANF(S,\mathcal{H})_{\bar S}\times \ANF(S,\mathcal{H}^{\vee})_{\bar S}.
  $$
  Then each choice of coset representative $\tilde\tau_i$
  of $\tau_i:=\tau_i(\cl\nu,\cl\omega)\in\mathbb{Q}/\mathbb{Z}$ determines a unique extension
  $\overline{\mathcal{L}}$ of
  $\mathcal{L}:=\mathcal{L}(\nu,\omega)$ to
  $\bar{S}$ whose non-vanishing sections are admissible biextension
  variations
  $\mathcal{V}$ with the following property: for every test curve
  $\bar\varphi:\Delta\to S'$ interesting
  $Y_i$ transversally at
  $\bar\varphi(0)$, $\mu(\varphi^*\mathcal{V})=\tilde\tau_i$.
\end{theorem}
\begin{proof}
  First assume that $Y$ is a normal crossing divisor.
  Then,  by Corollary~\ref{me3}, the sheaf $\calB^{\ad}(\nu,\omega)$ of
  admissible biextension variations is a pseudo-torsor for the group
  of non-vanishing meromorphic functions with poles along $Y$.  By
  Theorem~\ref{sec:mix-1}, this pseudo-torsor is actually a torsor for the
  sheaf of non-vanishing meromorphic functions.  For every section $\mathcal{V}$ of $\calB^{\ad}(\nu,\omega)$
  and every test curve $\bar\varphi:\Delta\to S'$ intersecting the divisor $Y$ transversally at $y=\bar\varphi(0)$
  with $y$ a point in $Y_i$, $\mu(\varphi^*\mathcal{V})=\tilde\tau\varphi^*\mathcal{V}$ is a coset representative of $\tau_i$.  Fixing the $\tilde\tau_i$ then reduces the torsor to a
  $\mathcal{O}_{\bar S}^{\times}$ torsor $\overline{\mathcal{B}}^{\ad}(\nu,\omega)$.  Equivalently,
  it gives a line bundle $\overline{\mathcal{L}}$ as desired. 

  Now, in the general case, set $Y'=Y\cap S'$.  Then, as $Y'$ is a normal crossing divisor in $S'$, there is a extension $\mathcal{L}'$ of $\mathcal{L}$
  to $S'$.
  Now, use Hironaka to find a proper morphism $\pi:\bar{T}\to\bar{S}$
  which is an isomorphism over $S'$ such that the inverse image of $Y$ under $\pi$ is a
  normal crossing divisor $D$.  Write $D=\cup_{i=1}^m D_i$ in such a way that the
  $D_i$ ($i=1,\ldots, k$) are the strict transforms of the $Y_i$.   Pick rational numbers
  $\tilde\tau_i$ lifting the $\tau_i$ keeping them the same as for the $Y_i$ for $i=1,\ldots k$
  (and making arbitrary choices for $i=k+1,\ldots, m$).  Then we get a unique extension
  $\overline{\mathcal{M}}$ of $\mathcal{L}'$ to $\bar{T}$.  Finally Lemma~\ref{refl}
  produces the desired extension on $\bar{S}$. 

  The uniqueness follows from Hartog's theorem and the property in the
  statement of Theorem~\ref{eb1}, which defines 
  $\overline{\mathcal{L}}$ on $S'$.  The point is that two extensions
  of $\mathcal{L}$ which agree outside of a codimension $2$ set are equal
  by Hartog's theorem.
\end{proof}

\begin{remark}\label{canext}
  We get a canonical choice of extension
     $\overline{\calL}_{\can}\in \Pic \bar S\otimes\QQ$ defined by taking all the $\tilde\tau_i=0$. 
\end{remark}

\section{The Ceresa cycle and the Hain-Reed bundle}
\newcommand\mgbar{\overline{\mathcal{M}}_g}
\newcommand\mg{\mathcal{M}_g}

\subsection{}\label{c.symp} Fix an integer $g>1$ and let $\mathcal{T}_g$ denote the
Teichm\"uller space of a smooth, projective genus $g$ Riemann surface $X$.
Write $\Gamma_g$ for the mapping class group, the space of orientation
preserving diffeomorphisms of $X$ taken modulo isotopy.  Then, $\mathcal{T}_g$ is
a complex manifold, which is isomorphic as a real manifold 
to $\mathbb{R}^{6g-6}$.
Moreover, $\Gamma_g$ acts on $\mathcal{T}_g$ with orbifold quotient
$\mg$, the moduli stack of smooth, projective genus $g$
curves.  Write $H=\mathrm{H}_1(X,\mathbb{Z})$ and write $Q:\wedge^2
H\to \mathbb{Z}$ for the intersection pairing.  Fix a basis
$e_1,\ldots, e_g, f_1,\ldots, f_g$ for $H$ with the property that 
$$
  Q(e_i,f_j)=\delta_{ij}, Q(e_i,e_j)=Q(f_i,f_j)=0
$$
for $1\leq i,j\leq g$.  Write $\mathbf{Sp}_{2g}(\mathbb{Z})=\mathbf{Sp}(H)$ for
the group of automorphisms of $H$ with determinant one preserving $Q$.
The action of an element of $\Gamma_g$ on $H$ determines a surjection
$\Gamma_g\to \mathbf{Sp}_{2g}(\mathbb{Z})$.   The kernel $T_g$ of this surjection
is called the Torelli group.

The pairing $Q$ induces pairings $Q_k:\wedge^k H\otimes\wedge^k H\to\mathbb{Z}$ for
all non-negative integers $k$.   These are 
$\mathbf{Sp}_{2g}(\mathbb{Z})$-equivariant and $(-1)^k$-symmetric. 
If we set $\theta:=\sum_{i=1}^g e_i\wedge f_i\in \wedge^2 H$, then 
$\theta$ is $\mathbf{Sp}_{2g}(\mathbb{Z})$-invariant.  It follows that the maps
$u:\wedge^k H\to \wedge^{k+2}H$
 induced by $v\mapsto v\wedge \theta$
are $\mathbf{Sp}_{2g}(\mathbb{Z})$-equivariant as well.

\subsection{Boundary components of $\mgbar$}\label{cer-2} Suppose $g>2$ and $h$ is an integer
such that $1\leq h\leq \lfloor g/2\rfloor$.  Then $D_h$ denotes
the Zariski closure of the locus of stable curves consisting of a
smooth curve of genus $h$ and another smooth curve of genus $g-h$
meeting at one point.  $D_0$ denotes the Zariski closure of the
locus of stable curves consisting of a curve of geometric genus $g-1$
with one node.  Then
$D=D_0\cup \cdots \cup D_{\lfloor g/2\rfloor}$ is a normal
crossing divisor whose support is the complement of $\mg$ in $\mgbar$
(the Deligne-Mumford compactification of $\mg$). 

The divisor $D_0$ intersects itself in components
which we will label as $D_{0,h}$ for $0\leq h\leq \lfloor g/2\rfloor$.
($D$ is not a strict normal crossing divisor.) For $h>0$, the
component $D_{0,h}$ is the Zariski closure of the locus of stable
curves consisting of two smooth curves of genus $h$ and $g-h-1$ respectively
meeting at two points.  The generic point of the component $D_{0,0}$
is a curve of geometric genus $g-2$ with $2$ nodes.

\subsection{Dehn twists and bounding pairs}\label{s.dehn} If $\gamma$ is any simple closed curve,
we let $T_{\gamma}$ denote the Dehn twist of $X$ determined by $\gamma$.  This
is an element of the mapping class group.

A simple closed curve $\gamma$
in $X$
is said to be \emph{bounding} if $X\setminus \gamma$
is a union of two open Riemann surfaces.  A pair $(\gamma,\delta)$
of homologous simple closed curves, which are not homologously
trivial, is said to be a \emph{bounding pair} if $\gamma$
and $\delta$
are disjoint, homologous and not homologically trivial.  (See
Johnson~\cite{Johnson}.)   Contracting the simple closed curves 
$\gamma$ and $\delta$ in a bounding pair to two distinct points 
produces a curve $C$ which is a union of two smooth curves
of genus $h=:h(\gamma,\delta)$ and $g-h-1$ respectively for $1\leq h\leq g-1$ meeting at two points.   Thus contracting the simple closed curves
produces a curve in the interior of $D_{0,h}$.
It is possible to pick the symplectic basis 
from \S\ref{c.symp} in such a way that $\{e_1,\ldots, e_h, f_1,\ldots, f_h\}$ 
and $\{e_{h+2},\ldots ,e_{g},f_{h+2},\ldots, f_{g}\}$ are symplectic bases
for the cohomology of the two components.   We say that such a symplectic
basis is \emph{adapted to the bounding pair}. 

\subsection{Johnson Homomorphism} The lattice $\wedge^3 H$ decomposes as a sum
of two sublattices as follows [See~\cite{HainReedJDG}
or \cite{Johnson}]:  Let $u:H\to\wedge^3 H$ and $c:\wedge^3 H\to H$ be the $\mathbf{Sp}(H)$-equivariant 
maps given by 
$$
	u(x) = \theta\wedge x, \qquad c(x\wedge y \wedge z)=Q(x,y)z + Q(y,z)x+ Q(z,x)y
$$
Direct computation shows $c\circ u(x) = (g-1)x$. Define $I:\wedge^3 H\to\wedge^3 H$ by the rule
$$
       I(\omega) = (g-1)\omega - u\circ c(\omega)
$$

\begin{lemma}\label{l.VtorsionFree} $\ker(I) = \im(u)$.  Moreover, there exists a subgroup
$L$ of $\wedge^3 H$ such that $\wedge^3 H = u(H)\oplus L$.
\end{lemma}

\begin{proof}
  The assertion that $\im u\subset \ker I$ follows from the fact that
  $c\circ u(x)=(g-1)x$.  To see that $\ker I\subset \im u$, suppose $\omega\in\ker I$.
  Then $(g-1)\omega = u\circ c(\omega)$.   So $(g-1)\omega\in \im u$.
  So it suffices to show that $\im u$ has a complement, $L$, in $\wedge^3 H$. 
  
  \par The required subgroup $L$ is generated by the following elements:
   \begin{enumerate}
  \item all products of the form $v_i\wedge v_j\wedge v_k$ where
    each $v_l$ is either $e_l$ of $f_l$ and $i<j<k$,
  \item all products of the form $v_i\wedge e_j\wedge f_j$ where $v_i$
    is either $e_i$ or $f_i$ and 
    $i-j$ is not congruent to $0$ or $1$ modulo $g$.
  \end{enumerate}
\end{proof}

\par Define $V=\wedge^3 H/\im(u)$.  By the previous lemma, the quotient
map $\wedge^3 H\to V$ restricts to an isomorphism $I(\wedge^3 H)\to V$.
Let $j:V\to I(\wedge^3 H)$ denote the inverse isomorphism.  We note
that, since $V\cong L$, $V$ is torsion free.

\begin{theorem}[Johnson]\label{t.johnson} Suppose $g>1$.   There is a surjective
  group homomorphism $\tau:T_g\to V$.  If $(\gamma,\delta)$ is a
  bounding pair then there is a symplectic basis adapted to
  $(\gamma,\delta)$ such that
$$
\tau(T_{\gamma}T_{\delta}^{-1}) = [(\sum_{i=1}^h e_i\wedge f_i)\wedge f_{h+1}]\in V
$$
where $h=h(\gamma,\delta)$.  
\end{theorem}

\subsection{The Variation $V$} By abuse of notation, we can view $H$ as a variation of Hodge 
structure of weight $-1$ on $\mg$.  We get an exact sequence of variations Hodge structure of 
weight $-1$
\begin{equation}
0\to H\stackrel{u}{\to} (\wedge^3 H)(-1) \to V\to 0       \label{eq:third-rep}
\end{equation}
where $V\cong\im(I)(-1)$ as in the previous section.

  \par Suppose now that $C\in D_{0,h}$ is a curve obtained by
  contracting a bounding pair $(\gamma,\delta)$ as in \S\ref{s.dehn}.
  Since $\gamma$ and $\delta$ are homotopic, $T_{\gamma}$
  and $T_{\delta}$ act identically on $H$ and, thus, on $V$.
  The action of $T=T_{\gamma}$ on $H$ is given by $h\mapsto h+Q(h,\gamma)\gamma$.
  So $T=\id +N$ is unipotent with monodromy logarithm $N$ given by
  $h\mapsto Q(h,\gamma)\gamma$.   We have $N^2=0$.

  We can find a polydisk $P=\Delta^{3g-3}$ and an \'etale map $j:
  P\to \mgbar$ such that $j(0)=C$ and $P':=j^{-1}\mg\cong
  \Delta^{*2}\times\Delta^{3g-5}$.   If $C$ has no automorphisms, then
  (by shrinking $P$ if necessary) we can arrange it so that $j$ is an
  isomorphism onto its image. 
    Then the monodromy action of $\mathbb{Z}^2=\pi_1(P')$
  on the pullback of the universal curve to $P'$ is given $(a,b)\mapsto
  T_{\gamma}^aT_{\delta}^b$.   In particular, if we write $N_1$ and $N_2$ for the
  logarithms of the monodromy on $H$, or, rather, its pullback to $P'$,
  we see that $N_1=N_2$.  Moreover, if $N=N_1$, then $N^2=0$. 

  \begin{corollary}
    We have $\IH^1(P, V)=NV$.
  \end{corollary}

 \subsection{Normal functions on $\mg$} 

\begin{theorem}
\cite{HainMSRI}\label{t.hain}
   There is an element $\xi\in H^1(\Gamma_g,V)$, which
  is the class of a normal function
  $\nu\in\ANF(\mg,V)_{\bar{\mg}}$.  The restriction of $\xi$ to $T_g$ under the map
  $$
  H^1(\Gamma_g,V)\to H^1(T_g,V)=\Hom(T_g,V)
  $$
  is twice the Johnson homomorphism.
\end{theorem}
  
  \begin{theorem}
    Suppose $C\in\overline{\mathcal{M}}_g$ is a curve without
    automorphism obtained by contracting a bounding pair
    $(\gamma,\delta)$ as in \S\ref{s.dehn}.  Then, in terms of a
    symplectic basis adapted to the bounding pair, we have
    $$
    \sing_{C}\xi=[2 (\sum_{i=1}^h e_i\wedge f_i)\wedge f_{h+1}]\in NV.
    $$
  \end{theorem}

  \begin{proof}
    Consider the commutative diagram
    $$
    \xymatrix{
      \IH^1(\mg,V)\ar[d]\ar[rd]    & \\
      \IH^1(P,V)\ar[d]       & H^1(T_g,V)=\Hom(T_g, V)\ar[d]\\
      H^1(\mathbb{Z}^2,V)=H^1(P',V)\ar[r] &H^1(\mathbb{Z},V)=V.\\
      }
      $$
      The map along the bottom row is induced by the map
      $\mathbb{Z} \to \mathbb{Z}^2=\pi_1(P')$ sending the generator to
      $(1,-1)$.   The map on the right is induced by
      the homomorphism $\mathbb{Z}\to T_g$ which sends the generator
      to $T_{\gamma}T_{\delta}^{-1}$. 

      By Theorem~\eqref{t.hain}, the
      image of $\xi$ in $\Hom(T_g,V)$ is the homomorphism which takes
      $T_{\gamma}T_{\delta}^{-1}$ to
      $[2 (\sum_{i=1}^h e_i\wedge f_i)\wedge f_{h+1}]$.  By restriction, the image of the
      Johnson homomorphism in $H^1(\mathbb{Z},V)=V$ is precisely
      $[2 (\sum_{i=1}^h e_i\wedge f_i)\wedge f_{h+1}]$.  The result
      then follows from Lemma~\ref{cer-3}.
  \end{proof}

  The variation $V$ has a polarization $q:V\otimes V\to \mathbb{Z}(1)$
  defined by
  $$
       q(u,v):=\frac{1}{g-1}Q(j(u),j(v)).
  $$ (See~\cite[p.~203]{HainReedJDG}.)
  This polarization gives an isomorphism $a_q:V\to V^{\vee}$ and thus
  a normal function
  $\nu^{\vee}=a_q(\nu)\in\ANF(\mg, V^{\vee})_{\mgbar}$.  This
  in turn gives a metrized line bundle
  $\mathcal{L}:=\mathcal{L}(\nu,\nu^{\vee})$ on $\mg$.

  \begin{theorem}
    Suppose $C$ is a generic curve in $D_{0,h}$.  Then 
    $$
    h_q(\sing_C\xi, \sing_C\xi)=
    \frac{4t_1t_2}{t_1+t_2}(g-h-1)h
    $$ 
  \end{theorem}
  \begin{proof} To simplify the notation below, we let $h' = g-h-1$, and recall that 
    $$
     \sing_C\xi=2[\sum_{i=1}^h e_i\wedge f_i\wedge f_{h+1}] =
      2N[\sum_{i=1}^h e_i\wedge f_i\wedge e_{h+1}]
    $$ 
    where $N=N_1=N_2$ is the monodromy logarithm around the two branches of 
    $D_0$ intersecting at $C$.   

    \par Invoking Proposition~\eqref{simple-formula}, it follows that:
    \begin{align*}
      h_q(\sing_C\xi,\sing_C\xi)
    &=\frac{t_1t_2}{t_1+t_2}q(2[\sum_{i=1}^h e_i\wedge f_i\wedge e_{h+1}],
      2[\sum_{i=1}^h e_i\wedge f_i\wedge f_{h+1}])\\
    &=\frac{4t_1t_2}{(g-1)(t_1+t_2)}
      Q(j[(\sum_{i=1}^h e_i\wedge f_i)\wedge e_{h+1}],
      j[(\sum_{i=1}^h e_i\wedge f_i)\wedge f_{h+1}]).
    \end{align*}

    Now, for $v=e_{h+1}$ or $f_{h+1}$, $c((\sum_{i=1}^h e_i\wedge f_i)\wedge v)=
    h\theta\wedge v$.   So using this to compute $j$, we see that
    \begin{align*}
      h_q(\sing_C\xi,\sing_C\xi)&=
                                  \frac{4t_1t_2}{(g-1)(t_1+t_2)}Q\left(h'\sum_{i=1}^h e_i\wedge f_i\wedge e_{h+1}-h\sum_{i=h+2}^g e_i\wedge f_i\wedge e_{h+1},\right.\\
                                &\left.\hspace{1.4in}
h'\sum_{i=1}^h e_i\wedge f_i\wedge f_{h+1}-h\sum_{i=h+2}^g e_i\wedge f_i\wedge f_{h+1}\right)\\
                                &=\frac{4t_1t_2}{(g-1)(t_1+t_2)}[(h')^2 h+h^2 (h')]\\
                                &=\frac{4t_1t_2}{(g-1)t_1+t_2}(h')h(h+h')\\
      &=\frac{4t_1t_2}{t_1+t_2}(g-h-1)h
    \end{align*}
since $h+h' = g-1$.
 \end{proof}

\section{General jump pairing}\label{gjp}

\subsection{General Pairing}

The goal of this section 
is to give a definition of the jump pairing on local intersection
cohomology without the assumption that the boundary divisor is normal crossing.

To this end, we fix our usual notation that $\mathcal{H}$ is a
variation of pure Hodge structure with $\mathbb{Q}$ coefficients on a complex manifold $S$
which is a Zariski open subset of another manifold $\bar{S}$.  Write $j:S\to\bar{S}$ for the embedding.  Since we will only be concerned with the local
situation, we will assume that $\bar{S}=\Delta^r$ is a polydisk.

The intermediate extension $j_{!*}\mathcal{H}_{\mathbb{Q}}[r]$ gives
rise to a perverse sheaf $\IC(\mathcal{H})$ on $\bar{S}$.  Write
$\mathcal{P}_{\bar S}$ $\mathcal{P}$ for the set of all isomorphism classes mixed extensions
of $\IC(\mathbb{Q}_{\bar S})$ by $\IC(\mathcal{H})$ by
$\IC(\mathbb{Q}_{\bar{S}})$ in the category of perverse sheaves on
$\bar S$.  For a local system of $\mathbb{Q}$-vector spaces $\calL$ on
$S$, we can write $\IH^k(\mathcal{L})$ for
$\IH^k_0(\mathcal{L})$.   By shrinking $\bar S$ if necessary,
we can assume that $\IH^k_0(\mathcal{L})=\IH^k(\bar S, \mathcal{L})$. 

We get a map 
$$
\pi:\mathcal{P}\to \IH^1(\mathcal{H}_{\mathbb{Q}})\times
\IH^1(\mathcal{H}^{\vee}_{\mathbb{Q}})$$
by identifying the intersection cohomology groups with the appropriate
extension groups. 
Moreover, since $\IH^1(\bar S, \mathbb{Q})=\IH^2(\bar S,\mathbb{Q})=0$,
Proposition~\ref{PhiSurj} shows that $\pi$ is surjective and, for
any pair
\begin{equation}\label{abpair}
  (\alpha,\beta)\in \IH^1(\mathcal{H}_{\mathbb{Q}})\times
  \IH^1(\mathcal{H}^{\vee}_{\mathbb{Q}}),
\end{equation}
the set 
$\Extpan(\alpha,\beta)$ injects into $\mathcal{P}$.  It follows that
$\mathcal{P}$ has the structure of a biextension of
$\IH^1(\mathcal{H}_{\mathbb{Q}})\times \IH^1(\mathcal{H}^{\vee}_{\mathbb{Q}})$
by $\IH^1(\bar S,\mathbb{Q})=0$. 

\begin{corollary}
  For each pair $(\alpha,\beta)$ as in \eqref{abpair}, there exists
  a unique element $X=X(\alpha,\beta)$ of $\mathcal{P}$ with $\pi(X)=(\alpha,\beta)$. 
\end{corollary}

\begin{proposition} Suppose
  $(\bar\varphi,\varphi):(\Delta,\Delta^*)\to(\bar S,S)$ is a test
  curve.  Then $\varphi^*X(\alpha,\beta)$ is in the
  biextension $\mathcal{R}_{\mathbb{Q}}$ of~\eqref{ciso}.  In other words,
  both $\varphi^*\alpha$ and $\varphi^*\beta$ vanish.  
\end{proposition}

\begin{proof}
  This follows directly from Theorem~\ref{IH1Ad}.
\end{proof}

\begin{definition}\label{gdef}
  The \emph{jump} $j(\alpha,\beta,\bar\varphi)$ of $\alpha$ and $\beta$
  along a test curve $\bar\varphi$ is the number $\tilde\tau\varphi^* X(\alpha,\beta)$. 
\end{definition}

Let $\mathcal{Q}_S$ (or simply $\mathcal{Q}$ if $S$ is clear) denote the set of isomorphisms of mixed extensions
of $\mathbb{Q}$ by $\mathcal{H}$ by $\mathbb{Q}$ on $S$
such that the associated classes
$\alpha\in\rH^1(S,\mathcal{H})$ and
$\beta\in\rH^1(S,\mathcal{H}^{\vee})$ lie in
$\IH^1(\mathcal{H})$ and $\IH^1(\mathcal{H}^{\vee})$ respectively.
Here we view $\alpha$ and $\beta$ as extension classes (of $\mathbb{Q}$ by
$\mathcal{H}$ and of $\mathcal{H}$ by $\mathbb{Q}$ respectively).

\begin{lemma}
  For $(\alpha,\beta)\in\IH^1(\mathcal{H})\times\IH^1(\mathcal{H}^{\vee})$,
  $\Extpan(\alpha,\beta)$ is the fiber of $\mathcal{Q}$ over $(\alpha,\beta)$.
\end{lemma}

\begin{proof}
  We need to show that $G(f_0,f_1)=0$ for every pair $(f_0,f_1)\in
  \Hom(\mathbb{Q},\mathcal{H})\times \Hom(\mathcal{H},\mathbb{Q})$.
  By symmetry and additivity of $G(f_0,f_1)$, it suffices to show
  that the composition
  $$
  \Hom(\mathcal{H},\mathbb{Q})\times \IH^1(\mathcal{H})\to
  \Ext^1(\mathbb{Q},\mathbb{Q})=\rH^1(S,\mathbb{Q})
  $$
  induced by the cup product vanishes. (Here the $\Hom$ and $\Ext$
  groups are taken in the category of local systems on $S$.)

  To see this, note that $\Hom(\mathcal{H},\mathbb{Q})=\Hom(\IC(\mathcal{H}),\IC(\mathbb{Q}))$ (by restriction).  So we get a commutative diagram
  $$
  \xymatrix{
    \Hom_{\bar S}(\IC(\mathcal{H}),\IC(\mathbb{Q}))\times
    \Ext^1_{\bar S}(\IC(\mathcal{H}),\IC(\mathbb{Q})))\ar[r]\ar@{=}[d] &
    \Ext^1_{\bar S}(\IC(\mathbb{Q}),\IC(\mathbb{Q}))\ar[d]\ar@{=}[r] &\rH^1(\bar S,\mathbb{Q})=0\\
    \Hom_S(\mathcal{H},\mathbb{Q})\times \IH^1(\mathcal{H})\ar[r] &
  \Ext^1_S(\mathbb{Q},\mathbb{Q})\ar@{=}[r] & \rH^1(S, \mathbb{Q})
}$$
where the downward arrows are restriction.
The result follows immediately.    
\end{proof}

\begin{corollary}
  The map $\mathcal{Q}\to \IH^1(\mathcal{H})\times \IH^1(\mathcal{H}^{\vee})$
  makes $\mathcal{Q}$ into a biextension of  $\IH^1(\mathcal{H})\times \IH^1(\mathcal{H}^{\vee})$ by $\rH^1(S,\mathbb{Q})$.  Moreover, the map
  $\mathcal{P}\to\mathcal{Q}$ induced by restriction is a morphism
  of biextensions.  
\end{corollary}
\begin{proof}[Explanation]  Here the operations $+_1$ and $+_2$ are
  the obvious ones coming from \S\ref{psbi} as is the action of $\rH^1(S,\mathbb{Q})=\Ext^1_{S}(\mathbb{Q},\mathbb{Q})$. 
\end{proof}

Now suppose $\bar\varphi:\Delta\to\bar S$ is a test curve.  Since
the classes $\alpha$ and $\beta$ vanish on restriction to $\Delta^*$
via $\varphi$, we get a homomorphism of biextensions
$\varphi^*:\mathcal{Q}_S\to \mathcal{Q}_{\Delta^*}$.  Via the
isomorphism $\tilde\tau:\mathcal{Q}_{\Delta^*}\to \mathbb{Q}$, we then
get a number $\tilde\tau\varphi^*X$ for any isomorphism class
$X\in\mathcal{Q}$.

\subsection{Comparison with the asymptotic height pairing}

Now, we want to compare the general jump pairing from Definition~\ref{gdef}
with the asymptotic height pairing defined earlier.  To do this, we first
want to generalize Theorem~\ref{JDG5} to the local systems
$X\in\mathcal{Q}_S$ where $S=\Delta^{*r}$.   Pick $a\in\Delta^*$.
Then, for each $r$-tuple, $t=(t_1,\ldots, t_r)\in\mathbb{Z}_{\geq 0}$,
we can write
$\bar\varphi_t:\Delta\to\Delta^r$ for the test curve
$s\mapsto a(s^{t_1},\ldots, s^{t_r})$.   Write
$$
\tilde\tau_t X:= \tilde\tau \varphi^* X.
$$
It is easy to see that this rational number does not depend on the choice
of $a$.   Recall that $\epsilon_i=(0,\ldots, 0, 1,\ldots, 0)$ with the
$1$ in the $i$-th place.

For any non-negative integer $r$, write $\Perv(\Delta^r)$ for the
category of perverse sheaves on $\Delta^r$ and, following
Saito's notation from~\cite[\S 3.1]{MHM}, write
$\Perv(\mathbb{Q}_X)_{nc}$ for the full
subcategory consisting of perverse sheaves which are constructible
with respect to the stratification induced by the coordinate hyperplanes.

\begin{lemma}
  Suppose $a$ and $b$ are non-negative integers with $a+b=r$, 
  $y=(y_1,\ldots, y_a)\in(\Delta^{*})^a$ and write $\bar\varphi:\Delta^{b}\to\Delta^r$
  for the map sending $z=(z_1,\ldots, z_b)$ to $(y_1,\ldots, y_a,z_1,\ldots, z_b)$.
  If $\mathcal{F}$ is a perverse sheaf in $\Perv(\Delta^r)_{nc}$ then
  $\bar\varphi^*\mathcal{F}$ is an object in $\Perv(\Delta^b)_{nc}$.  Consequently the
functor $\bar\varphi^*:\Perv(\Delta^r)_{nc}\to\Perv(\Delta^b)_{nc}$ is exact.
\end{lemma}

\begin{proof}
  As in~\cite[\S 3.1]{MHM}, the perverse sheaves in $\Perv(\Delta^r)_{nc}$ consist of
  perverse sheaves on $\Delta^r$ with characteristic variety contained in the conormal
  bundles of the intersections of the coordinate hyperplanes.    Consequently, the map
  $\bar\varphi:\Delta^b\to\Delta^r$ is non-characteristic.   The result then follows
  from Kashiwara's theorem on non-characteristic restriction.
\end{proof}

\begin{corollary}
  We have $\tilde\tau_{\epsilon_i} X(\alpha,\beta)=0$ for all integers
  $i$ with $1\leq i\leq n$.
\end{corollary}

\begin{proof}
  Pick $a\in \Delta^*$ to define $\bar\varphi_t:\Delta\to \Delta^r$ as
  above for any $t\in\mathbb{Z}_{\geq 0}$.  Since
  $\bar\varphi^*_{\epsilon_i}:\Perv(\Delta^r)_{nc}\to
  \Perv(\Delta)_{nc}$ is exact,
  $\varphi^*_{\epsilon_i} X(\alpha,\beta)$ is in the trivial
  biextension $\mathcal{P}_{\Delta}$.  So
  $\tilde\tau\varphi^*_{\epsilon_i} X(\alpha,\beta)=0$.
\end{proof}

\begin{corollary}
  Suppose $X\in \mathcal{Q}_{\Delta^{*r}}$, and $t\in\mathbb{Z}_{\geq 0}^r$.
    Then
    $$
    \tilde\tau_t X = \tilde\tau_t X(\alpha,\beta)+ \sum_{i=1}^r
    \tilde\tau_{\epsilon_i} X.
    $$
    \end{corollary}
    \begin{proof}
      We have $X=X(\alpha,\beta)+ E$ for some element
      $E\in\Ext^1_{\Delta^{*r}}(\mathbb{Q},\mathbb{Q})=\rH^1(\Delta^{*r},\mathbb{Q})=\mathbb{Q}^r$.   Consequently, $\tilde\tau_t X= \tilde\tau_t X(\alpha,\beta)+
      \tilde\tau_t E$.
      The result follows from the (easy) fact that
      $\tilde\tau_t E= \sum \tilde\tau_{\epsilon_i} (E) t_i$.  
    \end{proof}

    We now state an analogue of Theorem~\ref{JDG5}.
    \begin{proposition}\label{pjdg5}
      Suppose $X\in \mathcal{Q}_{\Delta^{*r}}$ with $\pi(X)=(\alpha,\beta)$.
      Then
      $$h(t)(\alpha,\beta)=-\tilde\tau_t(X) + \sum_{i=1}^r
      \tilde\tau_{\epsilon_i} t_i.
      $$
    \end{proposition}

    \begin{proof}
      The proof is essentially the same as the proof of Theorem~\ref{JDG5}.
    \end{proof}

\begin{theorem}
  Suppose $\mathcal{H}$ is a torsion free variation of pure Hodge structure on
  $(\Delta^*)^r$ and $(\alpha,\beta)\in\IH^1(\mathcal{H})\times\IH^1(\mathcal{H}^{\vee})$.
  Let $(t_1,\ldots, t_r)\in\mathbb{Z}_{\geq 0}^r$, and write $\bar\varphi:\Delta\to\Delta^r$
  for the test curve $s\mapsto (s^{t_1},\ldots, s^{t_r})$.  Then
  $$
  h(\alpha,\beta)(t)= -j(\alpha,\beta,\bar\varphi_t).
  $$
\end{theorem}

\begin{proof}
  Apply Proposition~\ref{pjdg5} to $X=X(\alpha,\beta)$ using the
  fact that $\tilde\tau_{\epsilon_i} X=0$ for all $i$. 
\end{proof}

\bibliographystyle{plain}

\begin{thebibliography}{10}

\bibitem{ABGF}
O.~Amini, S.~J. Bloch, J.~I. Burgos~Gil, and J.~Fres\'an.
\newblock Feynman amplitudes and limits of heights.
\newblock {\em Izv. Ross. Akad. Nauk Ser. Mat.}, 80(5):5--40, 2016.

\bibitem{Bertrand13}
Daniel Bertrand.
\newblock Extensions panach\'ees autoduales.
\newblock {\em J. K-Theory}, 11(2):393--411, 2013.

\bibitem{BFNP}
Patrick Brosnan, Hao Fang, Zhaohu Nie, and Gregory Pearlstein.
\newblock Singularities of admissible normal functions.
\newblock {\em Invent. Math.}, 177(3):599--629, 2009.
\newblock With an appendix by Najmuddin Fakhruddin.

\bibitem{BP-Comp}
Patrick Brosnan and Gregory Pearlstein.
\newblock On the algebraicity of the zero locus of an admissible normal
  function.
\newblock {\em Compos. Math.}, 149(11):1913--1962, 2013.

\bibitem{bpannals}
Patrick Brosnan and Gregory~J. Pearlstein.
\newblock The zero locus of an admissible normal function.
\newblock {\em Ann. of Math. (2)}, 170(2):883--897, 2009.

\bibitem{BHJ}
J.~Burgos~Gil, D.~Holmes, and R.~de~Jong.
\newblock Singularities of the biextension metric for families of abelian
  varieties.
\newblock http://arxiv.org/abs/1604.00686.

\bibitem{BHR2}
Jose Burgos~Gill, David Holmes, and Robin de~Jong.
\newblock Positivity of the height jump divisor.
\newblock arxiv:1701.00370.

\bibitem{Luminy}
Eduardo Cattani and Aroldo Kaplan.
\newblock Degenerating variations of {H}odge structure.
\newblock {\em Ast\'erisque}, (179-180):9, 67--96, 1989.
\newblock Actes du Colloque de Th\'eorie de {H}odge ({L}uminy, 1987).

\bibitem{CKS86}
Eduardo Cattani, Aroldo Kaplan, and Wilfried Schmid.
\newblock Degeneration of {H}odge structures.
\newblock {\em Ann. of Math. (2)}, 123(3):457--535, 1986.

\bibitem{CKS87}
Eduardo Cattani, Aroldo Kaplan, and Wilfried Schmid.
\newblock {$L^2$} and intersection cohomologies for a polarizable variation of
  {H}odge structure.
\newblock {\em Invent. Math.}, 87(2):217--252, 1987.

\bibitem{Lear}
Lear D.
\newblock Extensions of normal functions and asymptotics of the height pairing.
\newblock PhD Thesis, University of Washington, 1990.

\bibitem{dCM}
Mark Andrea~A. de~Cataldo and Luca Migliorini.
\newblock On singularities of primitive cohomology classes.
\newblock {\em Proc. Amer. Math. Soc.}, 137(11):3593--3600, 2009.

\bibitem{deligne-diffeq}
Pierre Deligne.
\newblock {\em \'{E}quations diff\'erentielles \`a points singuliers
  r\'eguliers}.
\newblock Springer-Verlag, Berlin, 1970.
\newblock Lecture Notes in Mathematics, Vol. 163.

\bibitem{DL}
Pierre Deligne.
\newblock Letter to {E}. {C}attani and {A}. {K}aplan, 1993.

\bibitem{DeligneLetter}
Pierre Deligne.
\newblock Letter to {G}.~{P}earlstein, 2005.

\bibitem{giraud}
Jean Giraud.
\newblock {\em Cohomologie non ab\'elienne}.
\newblock Springer-Verlag, Berlin, 1971.
\newblock Die Grundlehren der mathematischen Wissenschaften, Band 179.

\bibitem{GR}
H.~Grauert and R.~Remmert.
\newblock Plurisubharmonische functionen in komplexen r\"aumen.
\newblock {\em Math. Z.}, 65:175--194, 1957.

\bibitem{GG}
Mark Green and Phillip Griffiths.
\newblock Algebraic cycles and singularities of normal functions.
\newblock In {\em Algebraic cycles and motives. {V}ol. 1}, volume 343 of {\em
  London Math. Soc. Lecture Note Ser.}, pages 206--263. Cambridge Univ. Press,
  Cambridge, 2007.

\bibitem{HainBiext}
Richard Hain.
\newblock Biextensions and heights associated to curves of odd genus.
\newblock {\em Duke Math. J.}, 61(3):859--898, 1990.

\bibitem{HainJump}
Richard Hain.
\newblock Normal functions and the geometry of moduli spaces of curves.
\newblock In {\em Handbook of moduli. {V}ol. {I}}, volume~24 of {\em Adv. Lect.
  Math. (ALM)}, pages 527--578. Int. Press, Somerville, MA, 2013.

\bibitem{HainReedJDG}
Richard Hain and David Reed.
\newblock On the {A}rakelov geometry of moduli spaces of curves.
\newblock {\em J. Differential Geom.}, 67(2):195--228, 2004.

\bibitem{HainMSRI}
Richard~M. Hain.
\newblock Torelli groups and geometry of moduli spaces of curves.
\newblock In {\em Current topics in complex algebraic geometry ({B}erkeley,
  {CA}, 1992/93)}, volume~28 of {\em Math. Sci. Res. Inst. Publ.}, pages
  97--143. Cambridge Univ. Press, Cambridge, 1995.

\bibitem{Hardouin}
Charlotte Hardouin.
\newblock Structure galoisienne des extensions it\'er\'es de modules
  diff\'erentiels.
\newblock \url{http://www.math.univ-toulouse.fr/~hardouin/thesclau2.pdf}.

\bibitem{HartshorneReflexive}
Robin Hartshorne.
\newblock Stable reflexive sheaves.
\newblock {\em Math. Ann.}, 254(2):121--176, 1980.

\bibitem{HP}
Tatsuki Hayama and Gregory Pearlstein.
\newblock Asymptotics of degenerations of mixed {H}odge structures.
\newblock {\em Adv. Math.}, 273:380--420, 2015.

\bibitem{Johnson}
Dennis Johnson.
\newblock An abelian quotient of the mapping class group {${\mathcal
  I}\sb{g}$}.
\newblock {\em Math. Ann.}, 249(3):225--242, 1980.

\bibitem{Kashiwara}
Masaki Kashiwara.
\newblock A study of variation of mixed {H}odge structure.
\newblock {\em Publ. Res. Inst. Math. Sci.}, 22(5):991--1024, 1986.

\bibitem{KK}
Masaki Kashiwara and Takahiro Kawai.
\newblock The {P}oincar\'e lemma for variations of polarized {H}odge structure.
\newblock {\em Publ. Res. Inst. Math. Sci.}, 23(2):345--407, 1987.

\bibitem{KNU-IV}
Kazuya Kato, Chikara Nakayama, and Sampei Usui.
\newblock Classifying spaces of degenerating mixed hodge structures, {IV}: The
  fundamental diagram.
\newblock http://arxiv.org/abs/1602.00811.

\bibitem{MR2784750}
Kazuya Kato, Chikara Nakayama, and Sampei Usui.
\newblock Classifying spaces of degenerating mixed {H}odge structures, {II}:
  spaces of {${\rm SL}(2)$}-orbits.
\newblock {\em Kyoto J. Math.}, 51(1):149--261, 2011.

\bibitem{Kiselman}
C.~Kiselman.
\newblock Plurisubharmonic functions and potential theory in several complex
  variables.

\bibitem{OkenekSchneiderSpindler}
Christian Okonek, Michael Schneider, and Heinz~and Spindler.
\newblock {\em Vector bundles on complex projective spaces}.
\newblock Modern Birkhäuser Classics. Birkhäuser/Springer Basel AG, Basel,
  2011.
\newblock Corrected reprint of the 1988 edition, With an appendix by S. I.
  Gelfand.

\bibitem{PP}
{G}. {P}earlstein and {P}eters {P}.
\newblock Differential geometry of the mixed {H}odge metric.
\newblock https://arxiv.org/abs/1407.4082, July 2015.

\bibitem{GregJDG}
Gregory Pearlstein.
\newblock {${\rm SL}\sb 2$}-orbits and degenerations of mixed {H}odge
  structure.
\newblock {\em J. Differential Geom.}, 74(1):1--67, 2006.

\bibitem{PearlsteinManuscripta}
Gregory~J. Pearlstein.
\newblock Variations of mixed {H}odge structure, {H}iggs fields, and quantum
  cohomology.
\newblock {\em Manuscripta Math.}, 102(3):269--310, 2000.

\bibitem{MHM}
Morihiko Saito.
\newblock Mixed {H}odge modules.
\newblock {\em Publ. Res. Inst. Math. Sci.}, 26(2):221--333, 1990.

\bibitem{SaitoANF}
Morihiko Saito.
\newblock Admissible normal functions.
\newblock {\em J. Algebraic Geom.}, 5(2):235--276, 1996.

\bibitem{Schmid73}
Wilfried Schmid.
\newblock Variation of {H}odge structure: the singularities of the period
  mapping.
\newblock {\em Invent. Math.}, 22:211--319, 1973.

\bibitem{GAGA}
Jean-Pierre Serre.
\newblock G\'eom\'etrie alg\'ebrique et g\'eom\'etrie analytique.
\newblock {\em Ann. Inst. Fourier, Grenoble}, 6:1--42, 1955--1956.

\bibitem{SerreProlong}
Jean-Pierre Serre.
\newblock Prolongement de faisceaux analytiques coh\'erents.
\newblock {\em Ann. Inst. Fourier (Grenoble)}, 16(fasc. 1):363--374, 1966.

\bibitem{SGA71}
{\em Groupes de monodromie en g\'eom\'etrie alg\'ebrique. {I}}.
\newblock Springer-Verlag, Berlin, 1972.
\newblock S\'eminaire de G\'eom\'etrie Alg\'ebrique du Bois-Marie 1967--1969
  (SGA 7 I), Dirig\'e par A. Grothendieck. Avec la collaboration de M. Raynaud
  et D. S. Rim, Lecture Notes in Mathematics, Vol. 288.

\bibitem{stacks-project}
The {Stacks Project Authors}.
\newblock {\itshape {Stacks Project}}.
\newblock \url{http://stacks.math.columbia.edu}, 2016.

\bibitem{Steenbrink}
J.~H.~M. Steenbrink.
\newblock Semicontinuity of the singularity spectrum.
\newblock {\em Invent. Math.}, 79(3):557--565, 1985.

\bibitem{SZ}
Joseph Steenbrink and Steven Zucker.
\newblock Variation of mixed {H}odge structure. {I}.
\newblock {\em Invent. Math.}, 80(3):489--542, 1985.

\bibitem{Treumann}
David Treumann.
\newblock Stacks similar to the stack of perverse sheaves.
\newblock {\em Trans. Amer. Math. Soc.}, 362(10):5395--5409, 2010.

\end{thebibliography}
\def\noopsort#1{} \def\cprime{$'$} \def\noopsort#1{} \def\cprime{$'$}


\end{document}